\documentclass[12 pt]{article}

\usepackage{amsmath, amssymb, amsthm, enumerate}

\usepackage{amsmath,amssymb,amsthm,enumerate,mathrsfs}

\usepackage[dvipdfmx]{color}
\usepackage[dvipdfmx]{graphicx}

\usepackage{bm}
\usepackage{url}
\usepackage{subfigure}
\usepackage{float}

\newtheorem{thm}{Theorem}[section]
\newtheorem{prop}{Proposition}[section]
\newtheorem{lem}{Lemma}[section]
\newtheorem{cor}{Corollary}[section]
\newtheorem{rem}{Remark}[section]

\newtheorem{prob}{Problem}[section]
\newtheorem{algo}{Algorithm}[section]
\newtheorem{assum}{Assumption}[section]

\numberwithin{equation}{section}

\begin{document}
\makeatletter

\begin{center}
\large{\bf Convergence Analysis of Iterative Methods for
Nonsmooth Convex Optimization over Fixed Point Sets of Quasi-Nonexpansive Mappings}\\
{\small This work was supported by the Japan Society for the Promotion of Science through a Grant-in-Aid for Scientific Research (C) (15K04763).}
\end{center}\vspace{5mm}
\begin{center}

\textsc{Hideaki Iiduka}\\
Department of Computer Science, Meiji University
              1-1-1 Higashimita, Tama-ku, Kawasaki-shi, Kanagawa, 214-8571 Japan.
              (iiduka@cs.meiji.ac.jp)
\end{center}

\vspace{2mm}

\footnotesize{
\noindent\begin{minipage}{14cm}
{\bf Abstract:}
This paper considers a networked system with a finite number of users and supposes that each user tries to minimize its own private objective function over its own private constraint set.
It is assumed that each user's constraint set can be expressed as a fixed point set of a certain quasi-nonexpansive mapping.
This enables us to consider the case in which the projection onto the constraint set cannot be computed efficiently. 
This paper proposes two methods for solving the problem of minimizing the sum of their nondifferentiable, convex objective functions over the intersection of their fixed point sets of quasi-nonexpansive mappings in a real Hilbert space.
One method is a parallel subgradient method that can be implemented under the assumption that each user can communicate with other users.
The other is an incremental subgradient method that can be implemented under the assumption that each user can communicate with its neighbors. 
Investigation of the two methods' convergence properties for a constant step size reveals that, with a small constant step size, they approximate a solution to the problem. 
Consideration of the case in which the step-size sequence is diminishing demonstrates that 
the sequence generated by each of the two methods strongly converges to the solution to the problem under certain assumptions. 
Convergence rate analysis of the two methods under certain situations is provided to illustrate the two methods' efficiency.
This paper also discusses nonsmooth convex optimization over sublevel sets of convex functions
and provides numerical comparisons
that demonstrate the effectiveness of the proposed methods.
\end{minipage}
 \\[5mm]

\noindent{\bf Keywords:} {fixed point, incremental subgradient method, nonsmooth convex optimization, parallel subgradient method, quasi-nonexpansive mapping}\\
\noindent{\bf Mathematics Subject Classification:} {65K05, 90C25, 90C90}

\hbox to14cm{\hrulefill}\par

\section{Introduction}\label{sec:1}
This paper focuses on a networked system consisting of a finite number of participating users and considers the problem of minimizing the sum of their nondifferentiable, convex functions over the intersection of their fixed point constraint sets of quasi-nonexpansive mappings in a real Hilbert space.
Optimization problems with a {\em fixed point constraint} (see, e.g., \cite{com,iiduka_mp2014,iiduka,yamada}) enable consideration of constrained optimization problems in which the explicit form of the metric projection onto the constraint set is not always known; i.e., the constraint set is not simple in the sense that the projection cannot be easily calculated.

The motivations for considering this problem are to devise optimization algorithms that have a wider range of application than previous algorithms for convex optimization over fixed point sets of nonexpansive mappings \cite{com,iiduka_mp2014,iiduka,yamada}
and to solve the problem by using {\em parallel and incremental optimization techniques} \cite[Chapter 27]{b-c}, \cite[Section 8.2]{bert}, \cite{neto2009,solo1998}, \cite[PART II]{censor1998}.

Many optimization algorithms have been presented for smooth or nonsmooth optimization. 
The parallel proximal algorithms \cite[Proposition 27.8]{b-c}, \cite[Algorithm 10.27]{comb2011}, \cite{pes2012} are useful for minimizing the sum of nondifferentiable, convex functions over the whole space. 
They use the ideas of the Douglas-Rachford algorithm \cite[Chapters 25 and 27]{b-c}, \cite{comb2007,comb2011,eckstein,lions} and the forward-backward algorithm \cite[Chapters 25 and 27]{b-c}, \cite{com2009,com2008,comb2011}, which use the {\em proximity operators} \cite[Definition 12.23]{b-c} of nondifferentiable, convex functions. 
The incremental subgradient method \cite[Section 8.2]{bert}, \cite{blatt,neto2009,joh,kiwi,nedic2001,solo1998} 
and projected multi-agent algorithms \cite{lobel2011,nedic2009_1,nedic2009,nedic} can minimize the sum of nondifferentiable, convex functions over certain constraint sets by using the {\em subgradients} \cite[Section 23]{rock1970} of the nondifferentiable, convex functions instead of the proximity operators.
The random projection algorithms \cite{nedic2011,wang2015} and the distributed random projection algorithm \cite{lee2013} are useful for constrained convex optimization when the constraint set is not known in advance 
or the projection onto the whole constraint set cannot be computed efficiently.
The incremental subgradient algorithm \cite[Sections 3.2 and 3.3]{neto2009} 
and the asynchronous distributed proximal algorithm \cite[Section 6]{pesquet2015} can work on nonsmooth convex optimization over sublevel sets of convex functions
onto which the projections cannot be easily calculated.
The incremental and parallel gradient methods \cite{iiduka_siopt2013,iiduka_hishinuma_siopt2014}
and an algorithm to accelerate the search for fixed points \cite{iiduka_mp2014} can perform smooth convex optimization over the fixed point sets of nonexpansive mappings. 
There have been no reports, however, on optimization algorithms for nonsmooth convex optimization with fixed point constraints of quasi-nonexpansive mappings.

This paper describes two methods for solving the main problem considered in the paper. 
One is a {\em parallel subgradient method} that can be implemented under the assumption that each user can communicate with other users.
The other is an {\em incremental subgradient method} that can be implemented under the assumption that each user can communicate with its neighbors. 
The proposed methods do not use proximity operators, in contrast to conventional asynchronous distributed or parallel proximal algorithms. 
Moreover, they can optimize over fixed point sets of quasi-nonexpansive mappings, in contrast to conventional incremental subgradient algorithms. 

The intellectual contribution of this paper is to enable one to deal with {\em nonsmooth convex optimization over the fixed point sets of quasi-nonexpansive mappings}, especially in contrast to recent papers \cite{iiduka_siopt2013,iiduka_mp2014} that discussed smooth convex optimization over the fixed point sets of nonexpansive mappings.

To clarify this contribution, 
let us consider the case where each user in the networked system tries to minimize its own private objective function over 
a sublevel set of a nonsmooth convex function, where one assumes each user can use the subgradients of the nonsmooth convex function. 
Although the projection onto the sublevel set cannot be easily computed within a finite number of arithmetic operations,
each user can compute the {\em subgradient projection} \cite[Proposition 2.3]{b-c2001}, \cite[Subchapter 4.3]{vasin}
that satisfies the quasi nonexpansivity condition, {\em not} the nonexpansivity condition (see Section \ref{sec:4} for the definition of the subgradient projection).  
Since the sublevel set coincides with the fixed point set of the subgradient projection, 
the problem considered in the whole system can be expressed as the problem of minimizing the sum of all users' objective functions over the intersection of the fixed point sets of quasi-nonexpansive mappings (see \cite{yamada2011} for applications of the problem and the relaxation method for the problem).
The proposed methods can thus be applied to nonsmooth convex optimization over sublevel constraint sets of nonsmooth convex functions.

The previously reported algorithms \cite{iiduka_siopt2013,iiduka_mp2014} cannot work on nonsmooth convex optimization over sublevel sets of nonsmooth convex functions. This is because they can be applied only when the constraint sets can be represented by fixed point sets of nonexpansive mappings and can work under the restricted situation such that all users' objective functions are smooth and the gradients of their objective functions are Lipschitz continuous and strongly or strictly monotone. The numerical examples section (Section \ref{sec:4}) considers a concrete nonsmooth convex problem over the intersection of sublevel sets of nonsmooth convex functions and describes how the proposed methods can solve it.

Another contribution of this paper is analysis of the proposed methods' convergence for different step-size rules. A small constant step size is shown to result in an approximate solution to the main problem. 
It is also shown that the sequence generated by each proposed method with a diminishing step size strongly converges to the solution to the problem under certain assumptions. In contrast to the convergence analyses of the previously reported algorithms \cite{iiduka_siopt2013,iiduka_mp2014}, we cannot directly apply smooth convex analysis and fixed point theory for nonexpansive mappings to convergence analysis of the proposed methods. However, this problem is solved by using the subgradients of nonsmooth convex objective functions and by modifying the algorithms presented in \cite{iiduka_siopt2013} to make fixed point theory for quasi-nonexpansive mappings applicable.
The rates of convergence of the two methods under certain situations are also provided to illustrate the two methods' efficiency. 

This paper is organized as follows. Section \ref{sec:2} gives the mathematical preliminaries and states the main problem. 
Section \ref{sec:3} presents the proposed parallel subgradient method for solving the main problem and describes its convergence properties for a constant step size and for a diminishing step size and the rates of convergence under certain situations. 
Section \ref{sec:3-1} presents the proposed incremental subgradient method for solving the main problem and describes its convergence properties for a constant step size and for a diminishing step size and the rates of convergence under certain situations. Section \ref{sec:4} considers a nonsmooth convex optimization problem over the intersection of sublevel sets of convex functions and compares numerically the behaviors of the two methods with that of a previous method.
Section \ref{sec:5} concludes the paper with a brief summary and mentions future directions for improving the proposed methods.

\section{Mathematical Preliminaries}\label{sec:2}
Let $H$ be a real Hilbert space with inner product $\langle \cdot, \cdot \rangle$ and its induced norm $\| \cdot \|$.
Let $\mathbb{N}$ denote the set of all positive integers including zero. The identity mapping on $H$ is denoted by $\mathrm{Id}$; i.e., $\mathrm{Id}(x) := x$ $(x\in H)$.

\subsection{Nonexpansivity, demiclosedness, convexity, and subdifferentiability}\label{subsec:2.1}
The fixed point set of a mapping $Q \colon H \to H$ is denoted by $\mathrm{Fix}(Q) := \{x\in H \colon Q (x) =x \}$.
$Q \colon H \to H$ is said to be {\em quasi-nonexpansive} \cite[Definition 4.1(iii)]{b-c} if $\| Q(x) - y \| \leq \| x-y \|$ for all $x\in H$ and for all $y\in \mathrm{Fix}(Q)$.  
When a quasi-nonexpansive mapping has one fixed point, its fixed point set is closed and convex \cite[Proposition 2.6]{b-c2001}. 
$Q \colon H \to H$ is said to be {\em quasi-firmly nonexpansive} \cite[Section 3]{b-c2014} 
if $\| Q(x) - y \|^2 + \| (\mathrm{Id} - Q)(x) \|^2 \leq \| x-y \|^2$ for all $x\in H$ and for all $y\in \mathrm{Fix}(Q)$.
It is observed that any quasi-firmly nonexpansive mapping satisfies the quasi-nonexpansivity condition.
It is proven from \cite[Proposition 4.2]{b-c} that $Q$ is quasi-firmly nonexpansive if and only if 
$R:= 2Q - \mathrm{Id}$ is quasi-nonexpansive.
This means that $(1/2)(\mathrm{Id} + R)$ is quasi-firmly nonexpansive when $R$ is quasi-nonexpansive.

$Q \colon H \to H$ is said to be {\em nonexpansive} \cite[Definition 4.1(ii)]{b-c} 
if $\| Q(x) - Q(y) \| \leq \| x-y \|$ for all $x,y\in H$.
It is obvious that any nonexpansive mapping satisfies the quasi-nonexpansivity condition.
The {\em metric projection} \cite[Subchapter 4.2, Chapter 28]{b-c} onto a nonempty, closed convex set $C$ 
$(\subset H)$, denoted by $P_C$, is defined for all $x\in H$ by $P_C(x) \in C$ and $\| x - P_C(x) \| =
\inf_{y\in C} \| x-y \|$.
$P_C$ is nonexpansive with $\mathrm{Fix}(P_C) = C$ \cite[Proposition 4.8, (4.8)]{b-c}. 

$T \colon H \to H$ is referred to as a {\em demiclosed} mapping \cite[p.108]{goebel1}, \cite[Theorem 4.17]{b-c} if, 
for any $(x_n)_{n\in\mathbb{N}}$ $(\subset H)$, the following implication holds:
the weak convergence of $(x_n)_{n\in\mathbb{N}}$ to $x \in H$ and $\lim_{n\to\infty} \| T(x_n) - w \| = 0$ $(w\in H)$
imply $T(x) = w$.
Section \ref{sec:4} will provide an example of mappings satisfying both quasi-firm nonexpansivity and demiclosedness conditions.

The following proposition indicates the properties of quasi-firmly nonexpansive mappings.

\begin{prop}\label{mainge1} 
Suppose that $Q \colon H \to H$ is quasi-firmly nonexpansive with $\mathrm{Fix}(Q) \neq \emptyset$ and $\alpha \in [0,1)$ and that $Q_\alpha := \alpha \mathrm{Id} + (1-\alpha)Q$. Then, the following hold:
\begin{enumerate}
\item[{\em (i)}] $\mathrm{Fix}(Q) = \mathrm{Fix}(Q_\alpha)$.
\item[{\em (ii)}] $Q_\alpha$ is quasi-nonexpansive.
\item[{\em (iii)}] $\langle x - Q_{\alpha} (x), x - y \rangle \geq (1-\alpha) \| x-Q(x) \|^2$ $(x\in H, y\in \mathrm{Fix}(Q))$.\footnote{If $Q$ is quasi-nonexpansive, $\langle x - Q (x), x -y \rangle \geq  (1/2)\|x - Q (x) \|^2$ $(x\in H, y\in \mathrm{Fix}(Q))$.
Hence,
$\langle x - Q_{\alpha} (x), x - y \rangle \geq ((1-\alpha)/2) \| x-Q(x) \|^2$ $(x\in H, y\in \mathrm{Fix}(Q))$.
We need to use the property in Proposition \ref{mainge1}(iii) to prove Lemmas \ref{lem:0} and \ref{lem:0-1}.
Accordingly, it is assumed that each user has a quasi-firmly nonexpansive mapping (see (A1)).} 
\end{enumerate}
\end{prop}

\begin{proof}
Remarks 2.1(i0) and (i1) in \cite{mainge2010} imply Proposition \ref{mainge1}(i) and (ii).
From $\| x-y \|^2 = \|x\|^2 -2 \langle x,y \rangle + \|y\|^2$ $(x,y\in H)$,
it is found that, for all $x\in H$ and for all $y\in \mathrm{Fix}(Q)$,
$\langle x - Q (x), x -y \rangle 
= (1/2)( \|x - Q (x) \|^2 + \|x-y\|^2 - \|Q(x) -y  \|^2 )$,
which, together with the quasi-firm nonexpansivity of $Q$, implies that
$\langle x - Q (x), x -y \rangle \geq  \|x - Q (x) \|^2$.
Hence, for all $x\in H$ and for all $y\in \mathrm{Fix}(Q)$,
$\langle x - Q_{\alpha} (x), x - y \rangle  = (1-\alpha) \langle x - Q (x), x -y \rangle \geq (1-\alpha) \| x-Q(x) \|^2$.
\end{proof}

A function $f \colon H \to \mathbb{R}$ is said to be {\em strictly convex} \cite[Definition 8.6]{b-c}
if, for all $x,y\in H$ and for all $\alpha \in (0,1)$,
$x \neq y$ implies $f(\alpha x + (1-\alpha)y) < \alpha f(x) + (1-\alpha) f(y)$.
$f$ is {\em strongly convex} with constant $\beta$ \cite[Definition 10.5]{b-c} 
if there exists $\beta > 0$ such that, for all $x,y\in H$ and for all $\alpha \in (0,1)$,
$f(\alpha x + (1-\alpha)y) + (\beta/2) \alpha (1-\alpha) \|x-y\|^2 \leq \alpha f(x) + (1-\alpha) f(y)$.

The {\em subdifferential} \cite[Definition 16.1]{b-c}, \cite[Section 23]{rock1970} of $f \colon H \to \mathbb{R}$
is defined for all $x\in H$ by 
\begin{align*}
\partial f (x) := \left\{u\in H \colon f (y) \geq f (x) + \left\langle y-x,u \right\rangle \text{ } \left(y\in H \right)  \right\}.
\end{align*}
We call $u$ $(\in \partial f(x))$ the {\em subgradient} of $f$ at $x\in H$.
If $f$ is strictly convex, $\partial f$ is strictly monotone; i.e., 
$\langle x-y,u-v \rangle > 0$ ($x,y\in H$ with $x\neq y$, $u \in \partial f(x), v \in \partial f(y)$) \cite[Example 22.3(ii)]{b-c}. 
If $f$ is strongly convex with constant $\beta$, $\partial f$ is strongly monotone; i.e.,
$\langle x-y,u-v \rangle \geq \beta \|x-y\|^2$ ($x,y\in H, u \in \partial f(x), v \in \partial f(y)$)
\cite[Example 22.3(iv)]{b-c}.

\begin{prop}{\em \cite[Propositions 16.14(ii), (iii)]{b-c}}\label{prop:sub}
Let $f \colon H \to \mathbb{R}$ be continuous and convex with $\mathrm{dom}(f) := \{ x\in H \colon f(x) < \infty \}
=H$.
Then, $\partial f(x) \neq \emptyset$ for all $x\in H$.
Moreover, for all $x\in H$, there exists $\delta > 0$ such that $\partial f(B(x;\delta))$ is bounded,
where $B(x;\delta)$ stands for a closed ball with center $x$ and radius $\delta$.
\end{prop}

The following proposition is used to prove the main results in the paper.
\begin{prop}{\em \cite[Lemma 2.1]{mainge2010}}\label{mainge}
Let $(\Gamma_n)_{n\in \mathbb{N}} \subset \mathbb{R}$ and suppose that 
$(\Gamma_{n_j})_{j\in \mathbb{N}}$ $(\subset (\Gamma_n)_{n\in \mathbb{N}})$
exists such that $\Gamma_{n_j} < \Gamma_{n_j +1}$ for all $j\in \mathbb{N}$.
Define $(\tau(n))_{n \geq n_0} \subset \mathbb{N}$ by $\tau(n) := \max \{ k \leq n \colon \Gamma_k < \Gamma_{k+1} \}$ for some $n_0 \in \mathbb{N}$.
Then, $(\tau(n))_{n \geq n_0}$ is increasing and $\lim_{n\to \infty} \tau(n)=\infty$.
Moreover, $\Gamma_{\tau(n)} \leq \Gamma_{\tau(n)+1}$ and $\Gamma_n \leq \Gamma_{\tau(n)+1}$ for all $n \geq n_0$. 
\end{prop}

\subsection{Main problem}\label{subsec:2.3}
The focus here is a networked system with $I$ users. 
Let 
\begin{align*}
\mathcal{I} := \{1,2,\ldots, I\}.
\end{align*}
Suppose that user $i$ $(i\in {\mathcal{I}})$ has its own private objective function, denoted by $f^{(i)} \colon H \to \mathbb{R}$, 
and its own mapping, denoted by $Q^{(i)} \colon H \to H$.
The following notation is used. 
\begin{align*}
&Q_\alpha^{\left(i \right)} := \alpha^{\left(i \right)} \mathrm{Id} 
    + \left( 1 - \alpha^{\left(i\right)}  \right) Q^{\left(i \right)} \text{ } \left( \alpha^{(i)} \in (0,1) \right),
\text{ } 
X := \bigcap_{i\in {\mathcal{I}}} \mathrm{Fix}\left( Q^{(i)} \right),\\
&f := \sum_{i\in {\mathcal{I}}} f^{(i)}, \text{ }
X^\star := \left\{x\in X \colon f(x) = f^\star := \inf_{y\in X} f \left(y\right) \right\}.
\end{align*}
 
The following problem is discussed in this paper.

\begin{prob}\label{prob:1} 
Suppose that the following (A1)--(A4) hold.
\begin{enumerate}
\item[{\em (A1)}] 
$Q^{(i)} \colon H \to H$ $(i\in {\mathcal{I}})$ is quasi-firmly nonexpansive.
\item[{\em (A2)}] 
$f^{(i)} \colon H \to \mathbb{R}$ $(i\in {\mathcal{I}})$ is continuous and convex
with $\mathrm{dom}(f^{(i)}) = H$.\footnote{When $H=\mathbb{R}^N$, a convex function $f^{(i)}$ satisfies the continuity condition \cite[Corollary 8.31]{b-c}.
Therefore, (A2) can be replaced by the convexity condition of $f^{(i)}$ with $\mathrm{dom}(f^{(i)})=\mathbb{R}^N$.}
\item[{\em (A3)}] 
User $i$ $(i\in {\mathcal{I}})$ can use its own private $Q^{(i)}$ and $\partial f^{(i)}$.
\item[{\em (A4)}]
$X^\star \neq \emptyset$.
\end{enumerate} 
Then, find $x^\star \in X^\star$.
\end{prob}

\section{Parallel Subgradient Method}\label{sec:3}
The section presents a method for solving Problem \ref{prob:1}
under the assumption that 
\begin{enumerate}
\item[(A5)]
each user can communicate with other users.
\end{enumerate}

\begin{algo}\label{algo:1}
\text{}

Step 0. User $i$ $(i\in {\mathcal{I}})$ sets 
$\alpha^{(i)}$, $(\lambda_n)_{n\in\mathbb{N}} \subset (0,\infty)$, 
and $x_0 \in H$. 

Step 1. User $i$ $(i\in {\mathcal{I}})$ computes $x_n^{(i)} \in H$ using
\begin{align*}
x_n^{\left( i \right)} :=   
Q_{\alpha}^{\left( i \right)} \left( x_n \right) 
  - \lambda_n g_n^{\left(i \right)}, \text{ where } 
g_n^{\left(i\right)} \in \partial f^{\left(i \right)} \left( Q_{\alpha}^{\left(i\right)} \left(x_n \right) \right).
\end{align*}
User $i$ $(i\in \mathcal{I})$ transmits $x_n^{(i)}$ to all users. 

Step 2. User $i$ $(i\in {\mathcal{I}})$ computes $x_{n+1} \in H$ as 
\begin{align*}
x_{n+1} := \frac{1}{I} \sum_{i\in {\mathcal{I}}} x_n^{(i)}.
\end{align*} 
The algorithm sets $n := n+1$ and returns to Step 1.
\end{algo}

Algorithm \ref{algo:1} requires that all users set the same step-size sequence $(\lambda_n)_{n\in\mathbb{N}}$ before algorithm execution and that they synchronize at each iteration. See \cite[Section 6]{pesquet2015} for the asynchronous distributed proximal algorithm that was used for solving nonsmooth convex optimization. Assumption (A5) ensures that each user has access to all $x_n^{(i)}$ and can compute $x_{n+1} = (1/I)\sum_{i\in {\mathcal{I}}} {x}_n^{(i)}$.
This means that a common variable $x_n$ $(n\in \mathbb{N})$ is shared by all users. To illustrate this situation, let us assume that there exists an operator who manages the system. Even in a situation where (A5) is not satisfied, the operator can still communicate with all users. Accordingly, if the operator sets an initial point $x_0$ and transmits $x_n$ to all users at each iteration $n$, user $i$ can compute $x_n^{(i)}$ by using its own private information. Since the operator has access to all $x_n^{(i)}$, the operator can compute $x_{n+1}$ and transmit it to all users. Therefore, assuming the existence of an operator guarantees that all users can share $x_n$ $(n\in \mathbb{N})$ in Algorithm \ref{algo:1}.

The convergence of Algorithm \ref{algo:1} depends on two assumptions.

\begin{assum}\label{assumption0}
For all $i\in \mathcal{I}$, there exist $M_1^{(i)}, M_2^{(i)} \in \mathbb{R}$ such that 
\begin{align*}
&\sup \left\{ \left\| g \right\|  \colon g\in \partial f^{\left(i\right)} \left(Q_{\alpha}^{\left(i\right)} \left(x_n\right) \right), \text{ } n\in \mathbb{N}  \right\} \leq M_1^{(i)},\\
&\sup \left\{ \left\| g \right\|  \colon g\in \partial f^{\left(i\right)} \left(x_n \right), \text{ } n\in \mathbb{N}  \right\} \leq M_2^{(i)}.
\end{align*} 
\end{assum}
\begin{assum}\label{assum:0}
The sequence $(x_n^{(i)})_{n\in\mathbb{N}}$ $(i\in \mathcal{I})$ is bounded.
\end{assum}

Assumption \ref{assum:0} implies Assumption \ref{assumption0}.
Indeed, the definition of $x_n$ $(n\in\mathbb{N})$ and the boundedness of $(x_n^{(i)})_{n\in\mathbb{N}}$ $(i\in\mathcal{I})$ 
ensure that $(x_n)_{n\in\mathbb{N}}$ is bounded.
From the quasi-nonexpansivity of $Q_\alpha^{(i)}$ $(i\in{\mathcal{I}})$, 
we have $\|Q_\alpha^{(i)} (x_n) - x \| \leq \| x_n - x \|$ $(x\in X)$, which, together with the boundedness of $(x_n)_{n\in\mathbb{N}}$, means that $(Q_\alpha^{(i)} (x_n))_{n\in\mathbb{N}}$ $(i\in{\mathcal{I}})$ is bounded. 
Hence, Proposition \ref{prop:sub} implies that Assumption \ref{assumption0} holds.

A convergence analysis of Algorithm \ref{algo:1} with a constant step size when Assumption \ref{assumption0} holds is given in Subsection \ref{subsec:3.1}. The discussion in Subsection \ref{subsec:3.2} needs to satisfy Assumption \ref{assum:0}, which is stronger than Assumption \ref{assumption0}, to enable the convergence property of Algorithm \ref{algo:1} with a diminishing step-size sequence to be studied. This is because, in the case where Assumption \ref{assum:0} does not hold and $(\| x_n - x \|)_{n\in\mathbb{N}}$ $(x\in X)$ is not monotone decreasing (see Case 2 in the proof of Theorem \ref{thm:2}), a weak convergent subsequence of $(x_n)_{n\in\mathbb{N}}$ does not exist; i.e., we cannot discuss weak convergence of Algorithm \ref{algo:1} to a point in $X^\star$.

Here we provide an example satisfying Assumption \ref{assum:0} and (A4). Let us assume that user $i$ $(i\in {\mathcal{I}})$ can choose in advance a simple, bounded, closed convex set $X^{(i)}$ (e.g., $X^{(i)}$ is a closed ball with a large enough radius) satisfying $X^{(i)} \supset \mathrm{Fix}(Q^{(i)})$. Then, user $i$ can compute $P^{(i)} := P_{X^{(i)}}$ and
\begin{align}\label{equation:0}
x_n^{\left( i \right)} :=   
P^{\left(i \right)} \left( Q_{\alpha}^{\left( i \right)} \left( x_n \right) 
  - \lambda_n g_n^{\left(i \right)} \right)
\end{align}
instead of $x_n^{(i)}$ in Algorithm \ref{algo:1}. 
Since $(x_n^{(i)})_{n\in \mathbb{N}} \subset X^{(i)}$ and $X^{(i)}$ is bounded, $(x_n^{(i)})_{n\in\mathbb{N}}$ is bounded. 
Since $X^{(i)}$ is bounded and $X \subset X^{(i)}$ $(i\in\mathcal{I})$, $X$ is also bounded. 
Hence, the continuity and convexity of $f$ ensure that $X^\star \neq \emptyset$; i.e., (A4) holds \cite[Proposition 11.14]{b-c}.
We can show that Algorithm \ref{algo:1} with \eqref{equation:0}
satisfies the convergence properties in the main theorems in this paper by referring to the proofs of the theorems.

The following is an important lemma that will be used to prove the main theorems.

\begin{lem}\label{lem:0}
Suppose that $(x_n)_{n\in\mathbb{N}}$ is the sequence generated by Algorithm \ref{algo:1} and that Assumptions (A1)--(A5) and \ref{assumption0} hold. The following properties then hold:
\begin{enumerate}
\item[{\em (i)}] 
For all $n\in \mathbb{N}$ and for all $x\in X$,
\begin{align*}
\left\| x_{n+1} - x \right\|^2 
&\leq \left\| x_n -x \right\|^2 
- \frac{2}{I}\sum_{i\in {\mathcal{I}}} \alpha^{\left(i \right)} \left( 1 -\alpha^{\left(i\right)} \right)
\left\| x_n - Q^{\left(i \right)} \left( x_n  \right)  \right\|^2\\
&\quad + 2 M_1 \lambda_n^2 - \frac{2\lambda_n}{I} \sum_{i\in {\mathcal{I}}} 
\left\langle x_n - x, g_n^{\left(i\right)} \right\rangle,
\end{align*}
where $M_1 := \max_{i\in {\mathcal{I}}} M_1^{{(i)}^2} < \infty$.
\item[{\em (ii)}]
For all $n\in \mathbb{N}$ and for all $x\in X$,
\begin{align*}
\left\| x_{n+1} - x \right\|^2 
&\leq \left\| x_n -x \right\|^2 + 2 M_1 \lambda_n^2
 + \frac{2\lambda_n}{I} \left( f(x) - f(x_n)  \right)\\
&\quad + \frac{2\lambda_n}{I} \left( \sqrt{M_1} +  M_2 \right) \sum_{i\in {\mathcal{I}}} 
 \left\| x_n - Q_\alpha^{\left(i \right)} \left( x_n  \right)  \right\|,
\end{align*}
where $M_2 := \max_{i\in {\mathcal{I}}} M_2^{(i)} < \infty$.
\end{enumerate}
\end{lem}

\begin{proof}
(i) 
Choose $x\in X \subset \mathrm{Fix}(Q^{(i)})$ $(i\in{\mathcal{I}})$ and $n\in \mathbb{N}$ arbitrarily. From $-2 \langle x,y \rangle = \| x-y \|^2 - \|x\|^2 - \|y\|^2$ $(x,y\in H)$, we find that, for all $i\in {\mathcal{I}}$, 
\begin{align*}
&\quad 2 \left\langle x_n^{(i)} - x_n + \lambda_n g_n^{\left(i\right)}, x_n - x \right\rangle\\
&= -2 \left\langle x_n - x_n^{(i)}, x_n - x \right\rangle 
  + 2 \lambda_n \left\langle x_n - x, g_n^{\left(i\right)} \right\rangle\\
&= \left\| x_n^{(i)} - x \right\|^2 - \left\| x_n - x_n^{(i)} \right\|^2  - \left\| x_n - x \right\|^2
 + 2 \lambda_n \left\langle x_n - x, g_n^{\left(i\right)} \right\rangle.
\end{align*}
Moreover, Proposition \ref{mainge1}(iii) ensures that, for all $i\in {\mathcal{I}}$,
\begin{align*}
2 \left\langle Q_\alpha^{(i)} (x_n) - x_n, x_n - x \right\rangle 
\leq -2 \left( 1 - \alpha^{(i)} \right) \left\| x_n - Q^{(i)} (x_n) \right\|^2.
\end{align*}
Accordingly, from $x_n^{(i)} := Q_\alpha^{(i)} (x_n) - \lambda_n g_n^{(i)}$ $(i\in {\mathcal{I}})$, 
\begin{align*}
2 \left\langle x_n^{(i)} - x_n +  \lambda_n g_n^{\left(i\right)}, x_n - x \right\rangle
&= 2 \left\langle Q_{\alpha}^{\left(i\right)} \left(x_n \right) - x_n, x_n - x    \right\rangle\\
&\leq -2 \left( 1 - \alpha^{(i)} \right) \left\| x_n - Q^{(i)} (x_n) \right\|^2.
\end{align*}
Therefore, for all $i\in {\mathcal{I}}$,
\begin{align*}
\left\| x_n^{(i)} - x \right\|^2 
&\leq \left\| x_n - x \right\|^2 + \left\| x_n - x_n^{(i)} \right\|^2
 - 2 \lambda_n \left\langle x_n - x, g_n^{\left(i\right)} \right\rangle\\
&\quad  -2 \left( 1 - \alpha^{(i)} \right) \left\| x_n - Q^{(i)} (x_n) \right\|^2.
\end{align*}
Moreover, from $\| x - y\|^2 \leq 2 \|x\|^2 + 2 \|y\|^2$ $(x,y\in H)$, 
\begin{align*}
\left\| x_n - x_n^{(i)} \right\|^2 
&= \left\| \left( x_n -  Q_{\alpha}^{\left( i \right)} \left( x_n \right) \right)
  + \lambda_n g_n^{\left(i\right)} \right\|^2\\
&\leq 2  \left\| x_n -  Q_{\alpha}^{\left( i \right)} \left( x_n \right) \right\|^2
  + 2  \lambda_n^2 \left\| g_n^{\left(i\right)} \right\|^2\\
&\leq 2 \left( 1 - \alpha^{(i)} \right)^2 \left\| x_n -  Q^{\left( i \right)} \left( x_n \right) \right\|^2 
  + 2  M_1 \lambda_n^2,  
\end{align*}
where $M_1 := \max_{i\in {\mathcal{I}}} M_1^{{(i)}^2} < \infty$
holds from Assumption \ref{assumption0}.
Hence, for all $i\in{\mathcal{I}}$, 
\begin{align*}
\left\| x_n^{(i)} - x \right\|^2 
&\leq \left\| x_n - x \right\|^2 - 2 \alpha^{(i)}  \left( 1 - \alpha^{(i)} \right)\left\| x_n - Q^{(i)} (x_n) \right\|^2 +2 M_1 \lambda_n^2\\ 
&\quad -2  \lambda_n \left\langle x_n - x, g_n^{\left(i\right)} \right\rangle,
\end{align*}
which, together with the convexity of $\| \cdot \|^2$, implies that
\begin{align*}
\left\| x_{n+1} - x \right\|^2
&\leq \frac{1}{I} \sum_{i\in {\mathcal{I}}} \left\| x_n^{(i)} - x \right\|^2\\
&\leq \left\| x_n -x \right\|^2 
- \frac{2}{I} \sum_{i\in {\mathcal{I}}} \alpha^{\left(i \right)} \left( 1 -\alpha^{\left(i\right)} \right)
\left\| x_n - Q^{\left(i \right)} \left( x_n  \right)  \right\|^2 + 2 M_1 \lambda_n^2\\
&\quad-\frac{2  \lambda_n}{I}\sum_{i\in {\mathcal{I}}} 
\left\langle x_n - x, g_n^{\left(i\right)} \right\rangle. 
\end{align*}

(ii)
Choose $x\in X \subset \mathrm{Fix}(Q^{(i)})$ $(i\in{\mathcal{I}})$ and $n\in \mathbb{N}$ arbitrarily.
From $g_n^{(i)} \in \partial f^{(i)}(Q_\alpha^{(i)} (x_n))$ $(i\in{\mathcal{I}})$, 
$\langle x - Q_{\alpha}^{( i )} ( x_n), g_n^{(i)}  \rangle
\leq f^{(i )} (x) - f^{(i )} ( Q_{\alpha}^{( i )} ( x_n ))$.
Hence, the Cauchy-Schwarz inequality ensures that, for all $i\in{\mathcal{I}}$,
$\langle x -  x_n, g_n^{(i)} \rangle
= \langle x - Q_{\alpha}^{( i )} ( x_n ), g_n^{(i)} \rangle 
   + \langle Q_{\alpha}^{( i )} ( x_n ) - x_n, g_n^{(i)}  \rangle
\leq f^{(i )} (x) - f^{(i )} ( Q_{\alpha}^{( i )} ( x_n ) )
   + \sqrt{M_1} \| Q_{\alpha}^{( i )} ( x_n ) - x_n \|$, 
which, together with $f := \sum_{i\in {\mathcal{I}}} f^{(i)}$, implies that
\begin{align*}
\sum_{i\in {\mathcal{I}}} \left\langle x -  x_n, g_n^{\left(i\right)}  \right\rangle
&\leq f(x) - f(x_n) + \sum_{i\in {\mathcal{I}}} \left( f^{\left(i\right)}(x_n) 
  -  f^{\left(i \right)} \left( Q_{\alpha}^{\left( i \right)} \left( x_n \right) \right) \right)\\
&\quad +\sqrt{M_1} \sum_{i\in {\mathcal{I}}}  \left\| Q_{\alpha}^{\left( i \right)} \left( x_n \right) - x_n \right\|.
\end{align*}
Set $M_2 := \max_{i\in {\mathcal{I}}} M_2^{(i)}$.
Then, Assumption \ref{assumption0} ensures that $M_2 < \infty$.
Since $g\in \partial f^{(i)} (x_n)$ implies that, for all $i\in {\mathcal{I}}$,
$f^{(i)} (x_n) - f^{(i)} (Q_\alpha^{(i)} (x_n)) \leq \langle x_n - Q_\alpha^{(i)} (x_n), g \rangle
\leq M_2 \|x_n - Q_\alpha^{(i)} (x_n)\|$,
it is found that 
\begin{align*}
\sum_{i\in {\mathcal{I}}} \left\langle x -  x_n, g_n^{\left(i\right)}  \right\rangle
\leq f(x) - f(x_n) + \left( \sqrt{M_1} + M_2  \right)
\sum_{i\in {\mathcal{I}}}  \left\| Q_{\alpha}^{\left( i \right)} \left( x_n \right) - x_n \right\|.
\end{align*}
Accordingly, Lemma \ref{lem:0}(i) leads to Lemma \ref{lem:0}(ii).
This completes the proof.
\end{proof}

\subsection{Constant step-size rule}\label{subsec:3.1}
The discussion in this subsection is based on the following assumption.

\begin{assum}\label{assum:1}
User $i$ $(i\in {\mathcal{I}})$ has $(\lambda_n)_{n\in\mathbb{N}}$ satisfying
\begin{align*}
\text{{\em (C1)} } \lambda_n := \lambda  \in (0,\infty) \text{ } (n\in \mathbb{N}).
\end{align*}
\end{assum}
Let us perform a convergence analysis of Algorithm \ref{algo:1} under Assumption \ref{assum:1}.

\begin{thm}\label{thm:1}
Suppose that Assumptions (A1)--(A5), \ref{assumption0}, and \ref{assum:1} hold. Then, $(x_n)_{n\in\mathbb{N}}$ in Algorithm \ref{algo:1} satisfies the relations
\begin{align*}
&\liminf_{n\to\infty}  \left\| x_n -  Q^{\left(i \right)} \left(x_n \right)  \right\|^2 
  \leq \frac{I M_{\lambda} \lambda}{\alpha^{\left(i\right)} \left(1-\alpha^{\left(i\right)} \right)}
  \text{ } \left(i\in {\mathcal{I}} \right),\\
&\liminf_{n\to \infty} f \left(x_n \right) 
\leq f^\star  + I M_1 \lambda
 + \left( \sqrt{M_1} + M_2 \right) \sum_{i\in{\mathcal{I}}}  
  \sqrt{\frac{\left( 1 - \alpha^{\left(i\right)} \right) I M_\lambda \lambda}{\alpha^{\left(i\right)}}},
\end{align*} 
where $M_1$ and $M_2$ are as in Lemma \ref{lem:0} and, for some $x \in X$, 
$M_\lambda := \sup_{n\in\mathbb{N}} (M_1 \lambda + (1/I) |\sum_{i\in {\mathcal{I}}} \langle x - x_n, g_n^{(i)}\rangle|)$. 
\end{thm}
Theorem \ref{thm:1} implies that Algorithm \ref{algo:1} with a small enough $\lambda$ may approximate a solution to Problem \ref{prob:1}.

\begin{proof}
Choose $x\in X$ arbitrarily and set 
$M_\lambda := \sup_{n\in\mathbb{N}} (M_1 \lambda + (1/I) |\sum_{i\in {\mathcal{I}}} \langle x - x_n, g_n^{(i)}\rangle|)$.
It is obvious that Theorem \ref{thm:1} holds when $M_\lambda = \infty$. 
Consider the case where $M_\lambda < \infty$.
First, let us show that
\begin{align}\label{inf:1}
\liminf_{n\to\infty} \sum_{i\in {\mathcal{I}}} \alpha^{\left(i\right)} \left(  1- \alpha^{\left(i\right)} \right) 
\left\| x_n - Q^{\left(i \right)} \left(x_n \right) \right\|^2 \leq I M_\lambda \lambda.
\end{align}
Here we assume that \eqref{inf:1} does not hold. Accordingly, $\delta$ $(> 0)$ can be chosen such that
\begin{align*}
\liminf_{n\to\infty} \sum_{i\in {\mathcal{I}}} \alpha^{\left(i\right)} \left(  1- \alpha^{\left(i\right)} \right) 
\left\| x_n - Q^{\left(i \right)} \left(x_n \right) \right\|^2 > I M_\lambda \lambda + 2 \delta.
\end{align*}
The property of the limit inferior of $(\sum_{i\in {\mathcal{I}}} \alpha^{(i)} (1-\alpha^{(i)})\| x_n - Q^{(i)} (x_n) \|^2)_{n\in \mathbb{N}}$ guarantees that there exists $n_0 \in \mathbb{N}$ such that $\liminf_{n\to \infty} \sum_{i\in {\mathcal{I}}} \alpha^{(i)} (1-\alpha^{(i)})\| x_n - Q^{(i)} (x_n) \|^2 - \delta \leq \sum_{i\in {\mathcal{I}}} \alpha^{(i)} (1-\alpha^{(i)})\| x_n - Q^{(i)} (x_n) \|^2$ for all $n \geq n_0$. Accordingly, for all $n \geq n_0$,
\begin{align*}
\sum_{i\in {\mathcal{I}}} \alpha^{\left(i\right)} \left(  1- \alpha^{\left(i\right)} \right) 
\left\| x_n - Q^{\left(i \right)} \left(x_n \right) \right\|^2
> I M_\lambda \lambda + \delta.
\end{align*}
Hence, Lemma \ref{lem:0}(i) leads to the finding that, for all $n \geq n_0$,
\begin{align*}
\left\| x_{n+1} - x \right\|^2 &\leq \left\| x_n -x \right\|^2 
- \sum_{i\in {\mathcal{I}}} \frac{2 \alpha^{\left(i \right)} \left( 1 -\alpha^{\left(i\right)} \right)}{I}
\left\| x_n - Q^{\left(i \right)} \left( x_n  \right)  \right\|^2
 + 2 M_\lambda \lambda\\
&< \left\| x_n -x \right\|^2  
- \frac{2}{I} \left\{ I M_\lambda \lambda + \delta \right\} + 2 M_\lambda \lambda\\
&=  \left\| x_n -x \right\|^2 -  \frac{2}{I} \delta.
\end{align*}
Therefore, induction ensures that, for all $n \geq n_0$,
\begin{align*}
0 \leq \left\| x_{n+1} - x \right\|^2 < \left\| x_{n_0} -x \right\|^2 -  \frac{2}{I} \delta \left(n + 1 -n_{0} \right).
\end{align*}
Since the right side of this inequality approaches minus infinity as $n$ diverges, there is a contradiction. Therefore, \eqref{inf:1} holds. Since $\liminf_{n\to\infty} \alpha^{(i)} (1-\alpha^{(i)}) \| x_n - Q^{(i)} (x_n)\|^2 \leq \liminf_{n\to\infty} \sum_{i\in {\mathcal{I}}} \alpha^{(i)} (1-\alpha^{(i)}) \| x_n - Q^{(i)} (x_n)\|^2$ $(i\in {\mathcal{I}})$, there is also another finding:
\begin{align}\label{inf:2}
\liminf_{n\to\infty} \left\| x_n - Q^{\left(i \right)} \left(x_n \right) \right\|^2 \leq \frac{I M_\lambda \lambda}{\alpha^{\left(i \right)} \left(1- \alpha^{\left(i\right)} \right)} \text{ } 
\left( i\in {\mathcal{I}} \right).
\end{align}

Let $i\in {\mathcal{I}}$ be fixed arbitrarily. 
Inequality \eqref{inf:2} and the property of the limit inferior of $(\| x_n - Q^{(i)} (x_n) \|^2)_{n\in\mathbb{N}}$ guarantee the existence of a subsequence $(x_{n_k})_{k\in\mathbb{N}}$ of $(x_n)_{n\in\mathbb{N}}$ such that 
\begin{align*}
\lim_{k\to\infty} \left\| x_{n_k} - Q^{\left(i \right)} \left(x_{n_k} \right) \right\|^2 =
\liminf_{n\to\infty} \left\| x_n - Q^{\left(i \right)} \left(x_n \right) \right\|^2
\leq \frac{I M_\lambda \lambda}{\alpha^{\left(i \right)} \left(1- \alpha^{\left(i\right)} \right)}.
\end{align*}
Therefore, for all $\epsilon > 0$, there exists $k_0 \in \mathbb{N}$ such that, for all $k \geq k_0$,
\begin{align}\label{4}
\left\| x_{n_k} - Q^{\left(i \right)} \left(x_{n_k} \right) \right\| \leq 
\sqrt{\frac{I M_\lambda \lambda}{\alpha^{\left(i \right)} \left(1- \alpha^{\left(i\right)} \right)} + \epsilon}.
\end{align}
Here, it is proven that, for all $k \geq k_0$,
\begin{align}\label{3}
\liminf_{n\to\infty} f(x_n)
&\leq f^\star  + I M_1 \lambda 
+ \left( \sqrt{M_1} + M_2 \right) \sum_{i\in{\mathcal{I}}} \left\| x_{n_k} - Q_{\alpha}^{\left(i \right)} \left( x_{n_k} \right)   \right\| + 2 \epsilon.
\end{align}
Now, let us assume that \eqref{3} does not hold for all $k \geq k_0$, i.e., there exists $n_1 \in \mathbb{N}$ such that, for all $n \geq n_1$,
\begin{align*}
\liminf_{n\to\infty} f(x_n)
&> f^\star  + I M_1 \lambda 
+ \left( \sqrt{M_1} + M_2 \right) \sum_{i\in{\mathcal{I}}} \left\| x_{n} - Q_{\alpha}^{\left(i \right)} \left( x_{n} \right)   \right\| + 2 \epsilon.
\end{align*}
Assumption (A4) means the existence of $x^\star \in X$ such that $f(x^\star) = f^\star$. 
Since the property of the limit inferior of $(f(x_n) )_{n\in\mathbb{N}}$ implies the existence of $n_2 \in \mathbb{N}$ such that $\liminf_{n\to\infty} f(x_n) - \epsilon \leq f(x_n)$ for all $n \geq n_2$, it is found that, for all $n \geq n_3 := \max \{n_1, n_2\}$,
\begin{align*}
f(x_n) - f \left( x^\star  \right)
&> I M_1 \lambda 
+ \left( \sqrt{M_1} + M_2 \right) \sum_{i\in{\mathcal{I}}} \left\| x_{n} - Q_{\alpha}^{\left(i \right)} \left( x_{n} \right)   \right\| + \epsilon.
\end{align*}
Therefore, Lemma \ref{lem:0}(ii) guarantees that, for all $n \geq n_3$,
\begin{align*}
\left\| x_{n+1} - x^\star \right\|^2
&\leq \left\| x_n -x^\star \right\|^2 + 2 M_1 \lambda^2
 + \frac{2\lambda}{I} \left( f \left(x^\star \right) - f(x_n)  \right)\\
&\quad + \frac{2\lambda}{I} \left( \sqrt{M_1} +  M_2 \right) \sum_{i\in {\mathcal{I}}} 
 \left\| x_n - Q_\alpha^{\left(i \right)} \left( x_n  \right)  \right\|\\
&< \left\| x_n -x^\star \right\|^2 + 2 M_1 \lambda^2\\
&\quad - \frac{2\lambda}{I} \left\{ I M_1 \lambda 
+ \left( \sqrt{M_1} + M_2 \right) \sum_{i\in{\mathcal{I}}} \left\| x_{n} - Q_{\alpha}^{\left(i \right)} \left( x_{n} \right)   \right\| + \epsilon \right\}\\
&\quad + \frac{2\lambda}{I} \left( \sqrt{M_1} +  M_2 \right) \sum_{i\in {\mathcal{I}}} 
 \left\| x_n - Q_\alpha^{\left(i \right)} \left( x_n  \right)  \right\|\\
&=  \left\| x_n -x^\star \right\|^2 - \frac{2\lambda}{I}\epsilon\\
&< \left\| x_{n_3} - x^\star \right\|^2 - \frac{2 \lambda}{I}\epsilon \left( n+1 -n_3 \right).
\end{align*}
Since the right side of the above inequality approaches minus infinity as $n$ diverges, there is a contradiction. Thus, \eqref{3} holds for all $k \geq k_0$. 
Therefore, \eqref{4} and \eqref{3} lead to the deduction that, for all $\epsilon > 0$,
\begin{align*}
\liminf_{n\to\infty} f(x_n)
&\leq f^\star  + I M_1 \lambda\\ 
&\quad + \left( \sqrt{M_1} + M_2 \right) \sum_{i\in{\mathcal{I}}} 
\left( 1- \alpha^{\left(i\right)} \right) \sqrt{\frac{I M_\lambda \lambda}{\alpha^{\left(i \right)} \left(1- \alpha^{\left(i\right)} \right)} + \epsilon}
+ 2 \epsilon.
\end{align*}
Since $\epsilon$ $(>0)$ is arbitrary, 
\begin{align*}
\liminf_{n\to\infty} f(x_n)
\leq f^\star  + I M_1 \lambda
+ \left( \sqrt{M_1} + M_2 \right) \sum_{i\in{\mathcal{I}}} 
 \sqrt{\frac{\left( 1- \alpha^{\left(i\right)} \right) I M_\lambda \lambda}{\alpha^{\left(i \right)}}}.
\end{align*}
This completes the proof.
\end{proof}

\subsection{Diminishing step-size rule}\label{subsec:3.2}
The discussion in this subsection is based on the following assumption.

\begin{assum}\label{assum:2}
User $i$ $(i\in {\mathcal{I}})$ has $(\lambda_n)_{n\in\mathbb{N}}$ satisfying
\begin{align*}
\text{{\em (C2)} } \lim_{n\to\infty} \lambda_n = 0   
\text{ and } \text{{\em (C3)} } \sum_{n=0}^{\infty} \lambda_n = \infty.
\end{align*}
Moreover, 
\begin{enumerate}
\item[{\em (A6)}] 
$\mathrm{Id} - Q^{(i)}$ $(i\in {\mathcal{I}})$ is demiclosed.
\end{enumerate}
\end{assum} 

An example of $(\lambda_n)_{n\in \mathbb{N}}$ satisfying (C2) and (C3) is $\lambda_n := 1/(n+1)^a$ $(n\in \mathbb{N})$, where $a\in (0,1]$.
Section \ref{sec:4} will provide an example of $Q^{(i)}$ $(i\in {\mathcal{I}})$ satisfying (A1) and (A6).

Let us perform a convergence analysis of Algorithm \ref{algo:1} under Assumption \ref{assum:2}.

\begin{thm}\label{thm:2}
Suppose that Assumptions (A1)--(A5), \ref{assum:0}, and \ref{assum:2} hold. 
Then there exists a subsequence of $(x_n)_{n\in\mathbb{N}}$ generated by Algorithm \ref{algo:1} that weakly converges to a point in $X^\star$.
Moreover, the whole sequence $(x_n)_{n\in\mathbb{N}}$ strongly converges to a unique point in $X^\star$ if one of the following holds:\footnote{Under (A4), the strict convexity of $f$ guarantees
the uniqueness of the solution to Problem \ref{prob:1} \cite[Corollary 25.15]{zeidler2b}.
If there exists an operator who manages the system, 
it is reasonable to assume that the operator has a strongly convex objective function so as to guarantee the convergence of $(x_n)_{n\in \mathbb{N}}$ in Algorithm \ref{algo:1} to the desired solution, i.e., one that makes the system stable and reliable.} 
\begin{enumerate}
\item[{\em (i)}]
One $f^{(i)}$ $(i\in \mathcal{I})$ is strongly convex.
\item[{\em (ii)}]
$H$ is finite-dimensional, and one $f^{(i)}$ $(i\in \mathcal{I})$ is strictly convex.
\end{enumerate}
\end{thm}

\begin{proof}
We consider two cases.

Case 1: 
Suppose that there exists $m_0 \in \mathbb{N}$ such that $\| x_{n+1} - x^\star \| \leq \| x_n - x^\star \|$ for all $n \geq m_0$ 
and for all $x^\star \in X^\star$. 
The existence of $\lim_{n\to\infty} \| x_n - x^\star \|$ is thus guaranteed for all $x^\star \in X^\star$. 
Hence, $(x_n)_{n\in\mathbb{N}}$ is bounded.
The quasi-nonexpansivity of $Q_\alpha^{(i)}$ $(i\in\mathcal{I})$ thus ensures that $(Q_\alpha^{(i)} (x_n))_{n\in\mathbb{N}}$ $(i\in \mathcal{I})$ is bounded.
Accordingly, Proposition \ref{prop:sub} guarantees that $M_1$ and $M_2$ defined as in Lemma \ref{lem:0} are finite.
From Lemma \ref{lem:0}(i), for all $n \geq m_0$ and for all $x^\star \in X^\star$,
\begin{align*}
&\quad \sum_{i\in{\mathcal{I}}} \frac{2 \alpha^{\left(i \right)} \left(1-\alpha^{\left(i\right)} \right)}{I}
\left\| x_n - Q^{\left(i\right)} \left(x_n\right) \right\|^2\\
&\leq \left\| x_n - x^\star \right\|^2 - \left\| x_{n+1} - x^\star \right\|^2 + 2 M_1 \lambda_n^2
 - \frac{2\lambda_n}{I} \sum_{i\in{\mathcal{I}}} \left\langle x_n - x^\star, g_n^{\left(i\right)} \right\rangle,
\end{align*}
which, together with (C2) and the boundedness of $(g_n^{(i)})_{n\in \mathbb{N}}$ $(i\in \mathcal{I})$, implies that
$\lim_{n\to\infty} (1/I)\sum_{i\in{\mathcal{I}}} 2 \alpha^{(i)}(1-\alpha^{(i)}) \| x_n - Q^{(i)} (x_n) \|^2 = 0$; i.e.,
\begin{align}\label{qi}
\lim_{n\to\infty} \left\| x_n - Q_\alpha^{\left(i\right)} \left(x_n\right) \right\| = 
\lim_{n\to\infty} \left\| x_n - Q^{\left(i\right)} \left(x_n\right) \right\| = 0 \text{ } \left(i\in{\mathcal{I}}\right).
\end{align}
Let us define, for all $n \in \mathbb{N}$ and for all $x\in X$,
\begin{align}\label{Mn}
M_n (x) := f(x_n) - f(x) - \left( \sqrt{M_1} + M_2 \right)\sum_{i\in{\mathcal{I}}}\left\| x_n - Q_\alpha^{(i)} (x_n) \right\| 
- I M_1 \lambda_n.
\end{align}
Then, Lemma \ref{lem:0}(ii) leads to the finding that, for all $n \in \mathbb{N}$ and for all $x\in X$,
\begin{align}\label{iii}
\frac{2\lambda_n}{I} M_n (x) \leq \left\| x_n - x \right\|^2 - \left\| x_{n+1} - x \right\|^2.
\end{align} 
Summing up this inequality from $n=0$ to $n=m$ $(m\in \mathbb{N})$ implies that 
$(2/I)\sum_{n=0}^m \lambda_n M_n (x) \leq \| x_0 - x \|^2 - \| x_{m+1} - x\|^2 \leq \| x_0 - x \|^2 < \infty$, so
\begin{align*}
\sum_{n=0}^\infty \lambda_n M_n (x) < \infty \text{ } \left( x\in X \right).
\end{align*}
Let us fix $x\in X$ arbitrarily.
Now, under the assumption that $\liminf_{n\to\infty} M_n (x) > 0$, 
$m_1 \in \mathbb{N}$ and $\gamma > 0$ can be chosen such that $M_n (x) \geq \gamma$ for all $n \geq m_1$. 
Accordingly, (C3) means that 
\begin{align*}
\infty = \gamma \sum_{n=m_1}^\infty \lambda_n \leq \sum_{n=m_1}^\infty \lambda_n M_n (x) < \infty,
\end{align*}
which is a contradiction. Therefore, for all $x\in X$, $\liminf_{n\to\infty} M_n (x) \leq 0$, i.e.,
\begin{align*}
\liminf_{n\to\infty}
\left\{
f(x_n) - f(x) - \left( \sqrt{M_1} + M_2 \right)\sum_{i\in{\mathcal{I}}}\left\| x_n - Q_\alpha^{(i)} (x_n) \right\| - I M_1 \lambda_n \right\} \leq 0,
\end{align*}
which, together with (C2) and \eqref{qi}, implies that
\begin{align*}
\liminf_{n\to\infty} f \left(x_n \right) \leq f(x) \text{ } \left( x\in X \right).
\end{align*}
Accordingly, there exists a subsequence $(x_{n_l})_{l\in \mathbb{N}}$ of $(x_n)_{n\in\mathbb{N}}$ such that 
\begin{align}\label{f1}
\lim_{l\to\infty} f\left( x_{n_l} \right) =  \liminf_{n\to\infty} f \left(x_n \right) \leq f(x) \text{ } \left( x\in X \right).
\end{align}
Since $(x_{n_l})_{l\in \mathbb{N}}$ is bounded, there exists 
$(x_{n_{l_m}})_{m\in \mathbb{N}}$ $(\subset (x_{n_l})_{l\in \mathbb{N}})$ such that $(x_{n_{l_m}})_{m\in \mathbb{N}}$ weakly converges to $x_* \in H$.
Hence, (A6) and \eqref{qi} ensure that $x_* \in \mathrm{Fix}(Q^{(i)})$ $(i\in{\mathcal{I}})$, i.e., 
$x_* \in X$.
Furthermore, the continuity and convexity of $f$ (see (A2)) imply that $f$ is weakly lower semicontinuous
\cite[Theorem 9.1]{b-c},
which means that $f(x_*) \leq \liminf_{m\to \infty} f(x_{n_{l_m}})$.
Therefore, \eqref{f1} leads to the finding that
\begin{align*}
f \left( x_* \right) \leq \liminf_{m\to\infty} f\left( x_{n_{l_m}} \right) = \lim_{m\to\infty} f\left( x_{n_{l_m}} \right) \leq f(x) \text{ } \left( x\in X \right), \text{ i.e., } x_* \in X^\star.
\end{align*}

Let us take another subsequence $(x_{n_{l_k}})_{k\in\mathbb{N}}$
$(\subset (x_{n_l})_{l\in \mathbb{N}})$ such that $(x_{n_{l_k}})_{k\in \mathbb{N}}$ weakly converges to 
$x_{**} \in H$.
A discussion similar to the one for obtaining $x_* \in X^\star$ guarantees that $x_{**} \in X^\star$.
Here, it is proven that $x_* = x_{**}$.
Now, let us assume that $x_* \neq x_{**}$. 
Then, the existence of $\lim_{n\to\infty} \|x_n - x^\star \|$ $(x^\star \in X^\star)$ and Opial's condition \cite[Lemma 1]{opial} imply that 
\begin{align*}
\lim_{n\to\infty} \left\| x_n - x_*  \right\| 
&= \lim_{m\to\infty} \left\| x_{n_{l_m}} - x_*  \right\|
< \lim_{m\to\infty} \left\| x_{n_{l_m}} - x_{**}  \right\|\\
&= \lim_{n\to\infty} \left\| x_{n} - x_{**}  \right\|
= \lim_{k\to\infty} \left\| x_{n_{l_k}} - x_{**}  \right\|
< \lim_{k\to\infty} \left\| x_{n_{l_k}} - x_{*}  \right\|\\
&= \lim_{n\to\infty} \left\| x_n - x_*  \right\|,
\end{align*}
which is a contradiction. Hence, $x_* = x_{**}$.
Accordingly, any subsequence of $(x_{n_l})_{l\in\mathbb{N}}$ converges weakly to $x_* \in X^\star$;
i.e., $(x_{n_l})_{l\in\mathbb{N}}$ converges weakly to $x_* \in X^\star$.
This means that $x_*$ is a weak cluster point of $(x_n)_{n\in\mathbb{N}}$ and belongs to $X^\star$.
A discussion similar to the one for obtaining $x_* = x_{**}$
guarantees that there is only one weak cluster point of $(x_n)_{n\in\mathbb{N}}$, so
we can conclude that, in Case 1, $(x_n)_{n\in\mathbb{N}}$ weakly converges to a point in $X^\star$.

Case 2: 
Suppose that $x_0^\star \in X^\star$ and $(x_{n_j})_{j\in\mathbb{N}}$ $(\subset (x_n)_{n\in\mathbb{N}})$ exist such that 
$\| x_{n_j} - x_0^\star \| < \| x_{n_j +1} - x_0^\star \|$ for all $j\in\mathbb{N}$.
Then, defining $\Gamma_n := \| x_n - x_0^\star \|$ $(n\in\mathbb{N})$ implies that 
$\Gamma_{n_j} < \Gamma_{n_j +1}$ for all $j\in \mathbb{N}$.
Assumption \ref{assum:0} and the definition of $x_n$ $(n\in\mathbb{N})$ guarantee the boundedness of $(x_n)_{n\in\mathbb{N}}$.
Moreover, from the quasi-nonexpansivity of $Q_\alpha^{(i)}$ $(i\in\mathcal{I})$, $(Q_\alpha^{(i)}(x_n))_{n\in\mathbb{N}}$ $(i\in\mathcal{I})$ is also bounded.
Accordingly, Proposition \ref{prop:sub} ensures that $M_1, M_2 < \infty$.
Proposition \ref{mainge} ensures the existence of $m_1 \in \mathbb{N}$ such that $\Gamma_{\tau(n)} < \Gamma_{\tau(n)+1}$ for all $n \geq m_1$,
where $\tau(n)$ is defined as in Proposition \ref{mainge}.
Lemma \ref{lem:0}(i) means that, for all $n \geq m_1$, 
\begin{align*}
&\quad \sum_{i\in {\mathcal{I}}} \frac{2 \alpha^{\left(i \right)} \left( 1 - \alpha^{\left(i \right)} \right)}{I}
\left\| x_{\tau(n)} - Q^{\left(i \right)} \left(x_{\tau(n)} \right) \right\|^2\\ 
&\leq \Gamma_{\tau(n)}^2  - \Gamma_{\tau(n)+1}^2 + 2 M_1 \lambda_{\tau(n)}^2
 - \frac{2\lambda_{\tau(n)}}{I} \sum_{i\in {\mathcal{I}}} \left\langle x_{\tau(n)} - x_0^\star, g_{\tau(n)}^{\left(i\right)} \right\rangle\\ 
&<  \left( 2 M_1 \lambda_{\tau(n)} - \frac{2}{I} \sum_{i\in {\mathcal{I}}} \left\langle x_{\tau(n)} - x_0^\star, g_{\tau(n)}^{\left(i\right)} \right\rangle \right) \lambda_{\tau(n)}.
\end{align*}
Hence, the condition $\lim_{n\to \infty} \tau(n)= \infty$ and (C2) imply that 
\begin{align}\label{t:tau} 
\lim_{n\to\infty} \left\| x_{\tau(n)} - Q_\alpha^{\left(i\right)} \left(x_{\tau(n)} \right) \right\| = 
\lim_{n\to\infty} \left\| x_{\tau(n)} - Q^{\left(i\right)} \left(x_{\tau(n)} \right) \right\| = 0 \text{ } 
\left(i\in {\mathcal{I}} \right).
\end{align}
Since \eqref{iii} implies
$(2\lambda_{\tau(n)}/I) M_{\tau(n)} (x_0^\star) \leq \Gamma_{\tau(n)}^2 - \Gamma_{\tau(n) +1}^2 < 0$ $(n\geq m_1)$
and $\lambda_{\tau(n)} > 0$ $(n \geq m_1)$,
$ M_{\tau(n)} (x_0^\star) < 0$ $(n \geq m_1)$ holds; i.e., for all $n \geq m_1$,
\begin{align}\label{fd}
f \left(x_{\tau(n)} \right) - f^\star < \left( \sqrt{M_1} + M_2 \right)\sum_{i\in{\mathcal{I}}}
\left\| x_{\tau(n)} - Q_\alpha^{(i)} \left(x_{\tau(n)} \right) \right\| + I M_1 \lambda_{\tau(n)}.
\end{align}
Accordingly, (C2) and \eqref{t:tau} imply that  
\begin{align}\label{f2}
\limsup_{n\to\infty} f \left(x_{\tau(n)} \right) \leq f^\star.
\end{align}

Choose a subsequence $(x_{\tau(n_k)})_{k\in\mathbb{N}}$ of $(x_{\tau(n)})_{n\geq m_1}$ arbitrarily.
The property of the limit superior of $(f (x_{\tau(n)}))_{n\geq m_1}$ and \eqref{f2} guarantee that  
\begin{align}\label{f_2}
\limsup_{k\to\infty} f \left(x_{\tau(n_k)} \right) \leq \limsup_{n\to\infty} f \left(x_{\tau(n)} \right) \leq f^\star.
\end{align}
The boundedness of $(x_{\tau(n_k)})_{k\in\mathbb{N}}$ ensures that there exists $(x_{\tau(n_{k_l})})_{l\in\mathbb{N}}$ $(\subset (x_{\tau(n_k)})_{k\in\mathbb{N}})$ such that $(x_{\tau(n_{k_l})})_{l\in\mathbb{N}}$ weakly converges to $x_\star \in H$. 
Then, (A6) and \eqref{t:tau} ensure that $x_\star \in X$.
Moreover, the weakly lower semicontinuity of $f$ and \eqref{f_2} guarantee that
\begin{align*}
f \left(x_\star \right) \leq \liminf_{l\to\infty} f\left(x_{\tau\left(n_{k_l}\right)}\right) 
\leq \limsup_{l \to \infty} f\left(x_{\tau \left(n_{k_l} \right)}   \right)
\leq f^\star; \text{ i.e., } x_\star \in X^\star.
\end{align*}
Therefore, $(x_{\tau(n_{k_l})})_{l\in\mathbb{N}}$ weakly converges to $x_\star \in X^\star$.
From Cases 1 and 2, there exists a subsequence of $(x_n)_{n\in \mathbb{N}}$ that weakly converges to a point in $X^\star$.

Suppose that assumption (i) in Theorem \ref{thm:2} holds.
Since $f := \sum_{i\in \mathcal{I}} f^{(i)}$ is strongly convex, $X^\star$ consists of 
one point, denoted by $x^\star$.
In Case 1, the strong convexity of $f$ guarantees that there exists $\beta > 0$ such that, for all $\alpha \in (0,1)$ and for all $l \in \mathbb{N}$, 
\begin{align*}
\frac{\beta}{2} \alpha \left(1-\alpha \right) \left\|x_{n_l} - x^\star \right\|^2
\leq \alpha f \left(x_{n_l}\right) + \left(1-\alpha \right) f^\star 
       - f \left(\alpha x_{n_l} + \left(1-\alpha \right) x^\star \right).
\end{align*}
Accordingly, from the existence of $\lim_{n\to\infty}\|x_{n} - x^\star\|$ 
and \eqref{f1},
\begin{align*}
\frac{\beta}{2} \alpha \left(1-\alpha \right) \lim_{l\to \infty}
\left\|x_{n_l} - x^\star \right\|^2
&\leq \lim_{l\to\infty} \left(\alpha f \left(x_{n_l}\right) + \left(1-\alpha \right) f^\star \right)\\
&\quad + \limsup_{l\to\infty} 
 \left(- f \left(\alpha x_{n_l} + \left(1-\alpha \right) x^\star \right) \right)\\
&\leq f^\star - \liminf_{l\to\infty}
  f \left(\alpha x_{n_l} + \left(1-\alpha \right) x^\star \right),
\end{align*}
which, together with the weak convergence of $(x_{n_l})_{l\in\mathbb{N}}$ to $x^\star$
and the weakly lower semicontinuity of $f$, implies that
\begin{align*}
\frac{\beta}{2} \alpha \left(1-\alpha \right) \lim_{l\to \infty}
\left\|x_{n_l} - x^\star \right\|^2
\leq f^\star - f \left(\alpha x^\star + \left(1-\alpha \right) x^\star \right)
= 0.
\end{align*}
That is, $(x_{n_l})_{l\in \mathbb{N}}$ strongly converges to $x^\star$.
Therefore, from \cite[Theorem 5.11]{b-c}, the whole sequence $(x_n)_{n\in\mathbb{N}}$ strongly converges to $x^\star$.

In Case 2, the strong convexity of $f$ leads to the deduction that, 
for all $\alpha \in (0,1)$ and for all $l\in \mathbb{N}$,
\begin{align*}
\frac{\beta}{2} \alpha \left(1-\alpha \right) \limsup_{l\to \infty}
\left\|x_{\tau(n_{k_l})} - x^\star \right\|^2
&\leq \alpha \limsup_{l\to\infty} f \left(x_{\tau(n_{k_l})}\right) + \left(1-\alpha \right) f^\star\\
&\quad - \liminf_{l\to\infty}
  f \left(\alpha x_{\tau(n_{k_l})} + \left(1-\alpha \right) x^\star \right).
\end{align*}
The weak convergence of $(x_{\tau(n_{k_l})})_{l\in \mathbb{N}}$ to $x^\star$,
the weakly lower semicontinuity of $f$, and \eqref{f_2} imply that 
\begin{align*}
\frac{\beta}{2} \alpha \left(1-\alpha \right) \limsup_{l\to \infty}
\left\|x_{\tau(n_{k_l})} - x^\star \right\|^2
&\leq  f^\star
- f \left(\alpha x^\star + \left(1-\alpha \right) x^\star \right) = 0,
\end{align*}
which implies that $(x_{\tau(n_{k_l})})_{l\in \mathbb{N}}$ strongly converges to 
$x^\star$.

When another subsequence $(x_{\tau(n_{k_m})})_{m\in \mathbb{N}}$ $(\subset (x_{\tau(n_{k})})_{k\in \mathbb{N}})$ can be chosen,
a discussion similar to the one for showing the weak convergence of $(x_{\tau(n_{k_l})})_{l\in \mathbb{N}}$ to a point in $X^\star$ guarantees that 
$(x_{\tau(n_{k_m})})_{m\in \mathbb{N}}$ also weakly converges to a point in $X^\star$.
Furthermore, a discussion similar to the one for showing the strong convergence 
of $(x_{\tau(n_{k_l})})_{l\in \mathbb{N}}$ to $x^\star$ ensures that 
$(x_{\tau(n_{k_m})})_{m\in \mathbb{N}}$ strongly converges to the same $x^\star$.
Hence, it is guaranteed that $(x_{\tau(n_k)})_{k\in \mathbb{N}}$ strongly converges to 
$x^\star$.
Since $(x_{\tau(n_k)})_{k\in \mathbb{N}}$ is an arbitrary subsequence of $(x_{\tau(n)})_{n \geq m_1}$,
$(x_{\tau(n)})_{n \geq m_1}$ strongly converges to $x^\star$; i.e., 
$\lim_{n\to\infty} \Gamma_{\tau(n)} = \lim_{n\to\infty} \| x_{\tau(n)} - x^\star \|=0$.
Accordingly, Proposition \ref{mainge} ensures that
\begin{align*}
\limsup_{n \to \infty} \left\| x_n - x^\star  \right\| \leq \limsup_{n\to \infty} \Gamma_{\tau(n)+1} = 0,
\end{align*}
which implies that, in Case 2, the whole sequence $(x_n)_{n\in\mathbb{N}}$ converges to $x^\star$.

Suppose that assumption (ii) in Theorem \ref{thm:2} holds.
Let $x^\star \in X^\star$ be the unique solution to Problem \ref{prob:1}.
In Case 1, it is guaranteed that $(x_n)_{n\in\mathbb{N}}$ converges to $x^\star \in X^\star$.
In Case 2, the convergence of $(x_{\tau(n_{k_l})})_{l\in\mathbb{N}}$ to $x^\star$ 
is guaranteed.
A discussion similar to the one for showing the strong convergence of 
$(x_{\tau(n)})_{n \geq m_1}$ to $x^\star$ (see the above paragraph) ensures that 
$(x_{\tau(n)})_{n \geq m_1}$ converges to $x^\star \in X^\star$.
Proposition \ref{mainge} thus leads to the convergence of the whole sequence $(x_n)_{n\in\mathbb{N}}$ to $x^\star$.
This completes the proof.
\end{proof}

\subsection{Convergence rate analysis of Algorithm \ref{algo:1} with diminishing step size}\label{subsec:3.3}
The following corollary establishes the rate of convergence of Algorithm \ref{algo:1}
for unconstrained nonsmooth convex optimization.

\begin{cor}\label{cor:1}
Consider Problem \ref{prob:1} when $Q^{(i)} = \mathrm{Id}$ $(i\in \mathcal{I})$
and suppose that the assumptions in Theorem \ref{thm:2} hold.
Then, for a large enough $n \in \mathbb{N}$,
\begin{align*}
f(x_n) - f^\star \leq I M_1 \lambda_n,
\end{align*}
where $M_1 := \max_{i\in \mathcal{I}} M_1^{{(i)}^2} < \infty$ 
and $M_1^{(i)}$ $(i\in \mathcal{I})$ is defined as in Assumption \ref{assumption0}.
\end{cor}
The larger the number of users $I$, the greater the $M_1 := \max_{i\in \mathcal{I}} M_1^{{(i)}^2}$.
Accordingly, Corollary \ref{cor:1} implies that, 
when the same step size sequence is used, the efficiency of Algorithm \ref{algo:1} 
with $Q^{(i)} = \mathrm{Id}$ $(i\in \mathcal{I})$ may decrease as the number of users 
$I$ increases.
 
\begin{proof}
In Case 1 in the proof of Theorem \ref{thm:2}, $\liminf_{n\to \infty} M_n(x^\star) \leq 0$ holds, where 
$M_n(x)$ $(n\in \mathbb{N}, x\in H)$ is defined by \eqref{Mn} and $\{x^\star\} = X^\star$.
Let us prove that there exists $k_1 \in \mathbb{N}$ such that, for all $n \geq k_1$,
$M_n(x^\star) \leq 0$.
If this assertion does not hold, there exist $\gamma > 0$ and a subsequence 
$(M_{n_j}(x^\star))_{j\in \mathbb{N}}$ $(\subset (M_n(x^\star))_{n\in\mathbb{N}})$ such that 
$\gamma < M_{n_j}(x^\star)$ for all $j\in \mathbb{N}$. 
Since Theorem \ref{thm:2} implies that $(x_n)_{n\in \mathbb{N}}$ strongly converges to $x^\star$, $0 < \gamma \leq \lim_{j\to \infty} M_{n_j}(x^\star) \leq 0$, which is a contradiction. 
Hence, $M_n(x^\star) \leq 0$ $(n \geq k_1)$.
Since $Q^{(i)} = \mathrm{Id}$ $(i\in \mathcal{I})$ implies that $\| x_n - Q_\alpha^{(i)}(x_n) \| = 0$ $(i\in \mathcal{I}, n\in \mathbb{N})$,
we have that, for all $n \geq k_1$,
$M_n(x^\star) = f(x_n) - f^\star - I M_1 \lambda_n \leq 0$.

In Case 2 in the proof of Theorem \ref{thm:2}, the condition $Q^{(i)} = \mathrm{Id}$ $(i\in \mathcal{I})$ 
and \eqref{fd} lead to the existence of $k_2 \in \mathbb{N}$ such that,
for all $n \geq k_2$, $f(x_{\tau(n)}) - f^\star < I M_1 \lambda_{\tau(n)}$.
This completes the proof. 
\end{proof}

The following provides the rate of convergence of Algorithm \ref{algo:1}
for constrained nonsmooth convex optimization under specific conditions.

\begin{cor}\label{cor:2}
Suppose that the assumptions in Theorem \ref{thm:2} hold. 
If there exists $\beta^{(i)} > 0$ $(i\in \mathcal{I})$ 
such that $\alpha^{(i)} > \beta^{{(i)}^2}/(\beta^{{(i)}^2} + 2)$ and 
$\mathrm{d}(x_n, X) := \| x_n - P_X (x_n) \| \leq \beta^{(i)} \|x_n - Q_{\alpha}^{(i)}(x_n)\|$ $(i\in \mathcal{I}, n\in\mathbb{N})$ and if $(\|x_n - Q^{(i)}(x_n)\|)_{n\in \mathbb{N}}$ $(i\in \mathcal{I})$ is monotone decreasing,
then, for all $i\in \mathcal{I}$ and for all $n\in \mathbb{N}$,
\begin{align*}
&\left\| x_n - Q^{(i)}(x_n)  \right\|^2
\leq \frac{I \left( \mathrm{d} \left(x_0, X \right)^2 + 3 M_1 \sum_{k=0}^n \lambda_k^2 \right)}{\left( 1 - \alpha^{(i)} \right) \left\{ \left(\beta^{{(i)}^2} + 2 \right) \alpha^{(i)} - \beta^{{(i)}^2} \right\} \left(n+1\right)},
\end{align*}
where $(\lambda_n)_{n\in\mathbb{N}}$ satisfies $\sum_{n=0}^\infty \lambda_n^2 < \infty$,
$M_1 := \max_{i\in \mathcal{I}} M_1^{{(i)}^2} < \infty$, and 
$M_1^{(i)}$ $(i\in \mathcal{I})$ is defined as in Assumption \ref{assumption0}.
Moreover, for a large enough $n\in \mathbb{N}$,
\begin{align*}
f(x_n) - f^\star
\leq I \left\{ \left(\sqrt{M_1} + M_2 \right) \sqrt{\frac{I M_3}{n+1}} + M_1 \lambda_n \right\},
\end{align*}
where $M_2 := \max_{i\in \mathcal{I}} M_2^{(i)} < \infty$,
$M_2^{(i)}$ $(i\in \mathcal{I})$ is defined as in Assumption \ref{assumption0},
$M_3 := \max_{i\in \mathcal{I}} M_3^{(i)} < \infty$,
and $M_{3}^{(i)} := (\mathrm{d}(x_0,X)^2 + 3 M_1 \sum_{k=0}^\infty \lambda_k^2)/
(( 1 - \alpha^{(i)}) \{ (\beta^{{(i)}^2} + 2) \alpha^{(i)} - \beta^{{(i)}^2} \})$ $(i\in \mathcal{I})$.
\end{cor}

Consider the case where $\alpha^{(i)} := 1/2$ and 
$Q^{(i)} := (1/(1-\alpha^{(i)})) (P_X - \alpha^{(i)} \mathrm{Id})$ $(i\in \mathcal{I})$; 
i.e., $Q_\alpha^{(i)} = P_X$ $(i\in \mathcal{I})$.
Then, $Q^{(i)}$ $(i\in \mathcal{I})$ is nonexpansive \cite[Proposition 4.25]{b-c}. 
Moreover, $\beta^{(i)} = 1$ $(i\in \mathcal{I})$ can be chosen such that $\alpha^{(i)} = 1/2 >  \beta^{{(i)}^2}/(\beta^{{(i)}2} + 2) = 1/3$ and $\mathrm{d}(x_n,X) = \beta^{(i)} \|x_n - Q_\alpha^{(i)}(x_n)\|$ $(i\in \mathcal{I}, n\in \mathbb{N})$.
Corollary \ref{cor:2} thus implies that, if $(\| x_n - P_X (x_n)\|)_{n\in\mathbb{N}}$ is monotone decreasing, Algorithm \ref{algo:1} with $\lambda_n : =1/(n+1)$ $(n\in \mathbb{N})$ satisfies 
$f(x_n) - f^\star \leq 
I \{ (\sqrt{M_1} + M_2) \sqrt{I M_3/(n+1)}
+ M_1/(n+1) \}$. 

The rate of convergence of Algorithm \ref{algo:1} depends on the number of users $I$ 
and the step size sequence $(\lambda_n)_{n\in\mathbb{N}}$.
Since the larger the $I$, the greater the $M_1 := \max_{i\in \mathcal{I}} M_1^{{(i)}^2}$ and the $M_2 := \max_{i\in \mathcal{I}} M_2^{(i)}$,
Corollary \ref{cor:2} implies that, 
when the same step size sequence is used, the efficiency of Algorithm \ref{algo:1} 
may decrease as the number of users $I$ increases, as seen in Corollary \ref{cor:1}.
Section \ref{sec:4} presents examples such that 
$(\| x_n - Q^{(i)}(x_n) \|)_{n\in\mathbb{N}}$ generated by Algorithm \ref{algo:1}
is monotone decreasing.

\begin{proof}
Set $z_n := P_X (x_n)$ $(n\in \mathbb{N})$. 
Then, $\mathrm{d}(x_{n+1},X):= \inf_{y\in X} \|x_{n+1} - y\| \leq \| x_{n+1} - z_n\|$ $(n\in \mathbb{N})$.
Accordingly, Lemma \ref{lem:0}(i) guarantees that, for all $n\in \mathbb{N}$,
\begin{align*}
\mathrm{d}\left(x_{n+1},X \right)^2 
&\leq \mathrm{d}\left(x_{n},X \right)^2 
-\frac{2}{I} \sum_{i\in \mathcal{I}} \alpha^{(i)} \left( 1 - \alpha^{(i)} \right)
\left\| x_n - Q^{(i)}(x_n) \right\|^2
+ 2 M_1 \lambda_n^2\\
&\quad + \frac{2\lambda_n}{I} \sum_{i\in \mathcal{I}} \left\langle z_n - x_n, g_n^{(i)} \right\rangle,
\end{align*}
which implies that, for all $N\in \mathbb{N}$,
\begin{align*}
&\frac{2}{I} \sum_{n=0}^N \sum_{i\in \mathcal{I}} \alpha^{(i)} \left( 1 - \alpha^{(i)} \right) \left\| x_n - Q^{(i)}(x_n) \right\|^2\\
&\quad \leq \mathrm{d} \left(x_0, X \right)^2 - \mathrm{d} \left(x_{N+1}, X \right)^2
+ 2 M_1 \sum_{n=0}^N \lambda_n^2
+ \frac{2}{I} \sum_{n=0}^N \lambda_n \sum_{i\in \mathcal{I}} \left\langle z_n - x_n, g_n^{(i)} \right\rangle.
\end{align*}
From $2\|x\| \|y\| \leq \|x\|^2 + \|y\|^2$ $(x,y\in H)$ and the Cauchy-Schwarz inequality,
$(2/I) \sum_{n=0}^N \sum_{i\in \mathcal{I}} \langle z_n - x_n, \lambda_n g_n^{(i)} \rangle
\leq (1/I) \sum_{n=0}^N \sum_{i\in \mathcal{I}}
      (\| z_n - x_n \|^2 + \lambda_n^2 \| g_n^{(i)}\|^2)$,
which, together with the definition of $M_1$ and 
$\| x_n - z_n \| \leq \beta^{(i)} \|x_n - Q_{\alpha}^{(i)}(x_n)\|$ 
$(i\in \mathcal{I}, n\in\mathbb{N})$,
implies that 
\begin{align*}
\frac{2}{I} \sum_{n=0}^N \sum_{i\in \mathcal{I}} \left\langle z_n - x_n, \lambda_n g_n^{(i)} \right\rangle
&\leq \frac{1}{I} \sum_{n=0}^N \sum_{i\in \mathcal{I}}
      \beta^{{(i)}^2} \left\|x_n - Q_\alpha^{(i)}(x_n) \right\|^2 
      + M_1 \sum_{n=0}^N \lambda_n^2.    
\end{align*}
Accordingly, for all $N\in \mathbb{N}$,
\begin{align*}
&\frac{1}{I} \sum_{n=0}^N \sum_{i\in \mathcal{I}}  
 \left( 1 - \alpha^{(i)} \right) \left\{ \left(\beta^{{(i)}^2} + 2 \right) \alpha^{(i)} - \beta^{{(i)}^2}  \right\} 
\left\| x_n - Q^{(i)}(x_n) \right\|^2\\
&\quad \leq \mathrm{d} \left(x_0,X \right)^2 + 3 M_1 \sum_{n=0}^N \lambda_n^2.
\end{align*}
From the monotone decreasing property of $(\|x_n - Q^{(i)}(x_n)\|)_{n\in \mathbb{N}}$ $(i\in \mathcal{I})$, 
for all $j\in \mathcal{I}$ and for all $N\in \mathbb{N}$, 
\begin{align*}
&\quad \frac{\left(N+1 \right)}{I} 
\left( 1 - \alpha^{(j)} \right) \left\{ \left(\beta^{{(j)}^2} + 2 \right) \alpha^{(j)} - \beta^{{(j)}^2}  \right\} \left\| x_N - Q^{(j)}\left(x_N\right) \right\|^2\\
&\leq \frac{\left(N+1 \right)}{I}
\sum_{i\in \mathcal{I}}  
 \left( 1 - \alpha^{(i)} \right) \left\{ \left(\beta^{{(i)}^2} + 2 \right) \alpha^{(i)} - \beta^{{(i)}^2}  \right\} \left\| x_N - Q^{(i)}\left(x_N\right) \right\|^2\\
&\leq \mathrm{d} \left(x_0,X \right)^2 + 3 M_1 \sum_{n=0}^N \lambda_n^2,
\end{align*}
which implies that, for all $j\in \mathcal{I}$ and for all $N\in \mathbb{N}$,
\begin{align*}
&\left\| x_N - Q^{(j)} \left(x_N \right)  \right\|^2
\leq \frac{I \left( \mathrm{d} \left(x_0,X \right)^2 + 3 M_1 \sum_{n=0}^N \lambda_n^2 \right)}{\left( 1 - \alpha^{(j)} \right) \left\{ \left(\beta^{{(j)}^2} + 2 \right) \alpha^{(j)} - \beta^{{(j)}^2} \right\} \left(N+1\right)}.
\end{align*}

In Case 1 in the proof of Theorem \ref{thm:2},
$\liminf_{n\to\infty} M_n(x^\star) \leq 0$, where $\{ x^\star \}= X^\star$. 
A discussion similar to the one for obtaining $M_n(x^\star) \leq 0$ $(n \geq k_1)$
in the proof of Corollary \ref{cor:1} implies that
there exists $k_3 \in \mathbb{N}$ such that, for all $n \geq k_3$, 
$M_n (x^\star) = f(x_n) - f^\star 
   - (\sqrt{M_1} + M_2) \sum_{i\in \mathcal{I}} \|x_n - Q_\alpha^{(i)}(x_n)\| - IM_1 \lambda_n \leq 0$.
In Case 2 in the proof of Theorem \ref{thm:2}, \eqref{fd} leads to the existence of 
$k_4 \in \mathbb{N}$ such that, for all $n \geq k_4$,
$f(x_{\tau(n)}) - f^\star < (\sqrt{M_1} + M_2) \sum_{i\in \mathcal{I}} \|x_{\tau(n)} 
- Q_\alpha^{(i)}(x_{\tau(n)})\| + IM_1 \lambda_{\tau(n)}$.
Accordingly, for a large enough $n\in \mathbb{N}$,
\begin{align*}
f(x_n) - f^\star
&\leq \left(\sqrt{M_1} + M_2 \right) \sum_{i\in \mathcal{I}} 
\left( 1 - \alpha^{(i)} \right) \left\|x_n - Q^{(i)}(x_n) \right\| + I M_1 \lambda_n\\ 
&\leq \left(\sqrt{M_1} + M_2 \right) \sum_{i\in \mathcal{I}} 
\left( 1 - \alpha^{(i)} \right) \sqrt{\frac{I M_3}{n+1}} + I M_1 \lambda_n\\
&\leq I \left(\sqrt{M_1} + M_2 \right) 
\sqrt{\frac{I M_3}{n+1}} + I M_1 \lambda_n.
\end{align*}
This completes the proof.
\end{proof}

\section{Incremental Subgradient Method}\label{sec:3-1}
The section presents a method for solving Problem \ref{prob:1}
under the assumption that 
\begin{enumerate}
\item[(A7)]
each user can communicate with his/her neighbors,
\end{enumerate}
where user $i$'s neighbors are users $(i-1)$ and $(i+1)$ $(i\in \mathcal{I})$ and user $0$ (resp. user $(I+1)$) 
stands for user $I$ (resp. user $1$).
This assumption implies that the network considered here is a ring-shaped network in which the users form a circle and pass along messages in cyclic order.

\begin{algo}\label{algo:2}
\text{}

Step 0. User $i$ $(i\in {\mathcal{I}})$ sets $\alpha^{(i)}$ and $(\lambda_n)_{n\in\mathbb{N}} \subset (0,\infty)$.
User $1$ sets $x_0 := x_0^{(0)} \in H$. 

Step 1. User $i$ $(i\in {\mathcal{I}})$ computes $x_n^{(i)} \in H$ using
\begin{align*}
x_n^{\left( i \right)} :=   
Q_{\alpha}^{\left( i \right)} \left( x_n^{\left(i-1\right)} \right) 
  - \lambda_n g_n^{\left(i \right)}, \text{ where } 
g_n^{\left(i\right)} \in \partial f^{\left(i \right)} \left( Q_{\alpha}^{\left(i\right)} \left(x_n^{\left(i-1\right)} \right) \right).
\end{align*} 

Step 2. User $I$ sets 
\begin{align*}
x_{n+1} = x_{n+1}^{\left(0\right)} :=  x_n^{\left(I\right)}
\end{align*}
and transmits it to user $1$. 
The algorithm sets $n := n+1$ and returns to Step 1.
\end{algo}

From (A3) and (A7), user $i$ $(i\in \mathcal{I})$ can compute 
$x_n^{(i)} :=   
Q_{\alpha}^{(i)} ( x_n^{(i-1)}) - \lambda_n g_n^{(i)}$, where 
$g_n^{(i)} \in \partial f^{(i)}( Q_{\alpha}^{(i)} (x_n^{(i-1)}))$, by using 
information $x_n^{(i-1)}$ transmitted from user $(i-1)$ and its own private information.

Now, let us consider the differences between Algorithms \ref{algo:1} and \ref{algo:2}.
In Algorithm \ref{algo:1}, user $i$ computes $x_n^{(i)}$ by using $x_n \in H$, $\lambda_n \in (0,\infty)$, and its own private information $\partial f^{(i)}$ and $Q_\alpha^{(i)}$, and each point $x_n^{(i)}$ is broadcast to all users.
As a result, all users have $(x_{n+1} := (1/I) \sum_{i\in\mathcal{I}} x_n^{(i)})_{n\in\mathbb{N}}$, which strongly converges to a point in $X^\star$ (Theorem \ref{thm:2}).
In Algorithm \ref{algo:2}, user $i$ computes $x_n^{(i)}$ by using $\lambda_n \in (0,\infty)$, $\partial f^{(i)}$, $Q_\alpha^{(i)}$, and the point $x_n^{(i-1)}$ transmitted from user $(i-1)$, and point $x_n^{(i)}$ is transmitted to user $(i+1)$.
User $i$ in Algorithm \ref{algo:2} has $(x_n^{(i)})_{n\in\mathbb{N}}$, which strongly converges to a point in $X^\star$
(Theorem \ref{thm:4}).

The following assumptions are made here.

\begin{assum}\label{a8}
For all $i\in \mathcal{I}$, there exist $N_1^{(i)}, N_2^{(i)} \in \mathbb{R}$ such that 
\begin{align*}
&\sup \left\{ \left\| g \right\|  \colon g\in \partial f^{\left(i\right)} \left(Q_{\alpha}^{\left(i\right)} \left(x_n^{(i-1)}\right) \right), \text{ } n\in \mathbb{N}  \right\} \leq N_1^{(i)},\\
&\sup \left\{ \left\| g \right\|  \colon g\in \partial f^{\left(i\right)} \left(x_n \right), \text{ } n\in \mathbb{N}  \right\} \leq N_2^{(i)}.
\end{align*} 
\end{assum}
\begin{assum}\label{a8_1}
The sequence $(x_n^{(i)})_{n\in\mathbb{N}}$ $(i\in\mathcal{I})$ is bounded.
\end{assum}

From a discussion similar to the one for obtaining the relationship between Assumptions \ref{assumption0} and \ref{assum:0},
Assumption \ref{a8_1} implies Assumption \ref{a8}.
Assumption \ref{a8} is used to perform a convergence analysis of Algorithm \ref{algo:2} with a constant step-size rule (Subsection \ref{subsec:4.1}) while Assumption \ref{a8_1} is used to analyze Algorithm \ref{algo:2} with a diminishing step-size rule 
for the same reason described in Section \ref{sec:3}.
The same discussion as for \eqref{equation:0} describing the existence of a simple,
bounded, closed convex set $X^{(i)}$ $(i\in\mathcal{I})$ satisfying $X^{(i)} \supset \mathrm{Fix}(Q^{(i)})$ leads to
\begin{align*}
x_n^{(i)} := P^{(i)} \left( Q_{\alpha}^{\left( i \right)} \left( x_n^{\left(i-1\right)} \right) 
  - \lambda_n g_n^{\left(i \right)}  \right)
\end{align*} 
instead of $x_n^{(i)}$ for Algorithm \ref{algo:2}.
The boundedness of $X^{(i)}$ guarantees that Assumption \ref{a8_1} holds (see also \eqref{equation:0}).

The following lemma can be shown by referring to the proof of Lemma \ref{lem:0}.

\begin{lem}\label{lem:0-1}
Suppose that $(x_n^{(i)})_{n\in\mathbb{N}}$ $(i\in\mathcal{I})$ is the sequence generated by Algorithm \ref{algo:2}
and that Assumptions (A1)--(A4), (A7), and \ref{a8} hold. 
The following properties then hold:
\begin{enumerate}
\item[{\em (i)}] 
For all $n\in \mathbb{N}$ and for all $x\in X$,
\begin{align*}
\left\| x_{n+1} - x \right\|^2 
&\leq \left\| x_n -x \right\|^2 
- 2 \sum_{i\in {\mathcal{I}}} \alpha^{\left(i \right)} \left( 1 -\alpha^{\left(i\right)} \right)
\left\| x_n^{\left(i-1 \right)} - Q^{\left(i \right)} \left( x_n^{\left(i-1 \right)}  \right)  \right\|^2\\
&\quad + 2 I N_1 \lambda_n^2 - 2\lambda_n \sum_{i\in {\mathcal{I}}} 
\left\langle x_n^{\left(i-1 \right)} - x, g_n^{\left(i\right)} \right\rangle,
\end{align*}
where $N_1 := \max_{i\in {\mathcal{I}}} N_1^{{(i)}^2} < \infty$.
\item[{\em (ii)}]
For all $n\in \mathbb{N}$ and for all $x\in X$,
\begin{align*}
\left\| x_{n+1} - x \right\|^2 
&\leq \left\| x_n -x \right\|^2 + 2\lambda_n \left( f(x) - f(x_n)  \right)\\
&\quad + 2 I N_1 \lambda_n^2 + 2\lambda_n \bigg\{
\sqrt{N_1} \sum_{i\in {\mathcal{I}}}  \left\| Q_{\alpha}^{\left( i \right)} \left( x_n^{(i-1)} \right) - x_n^{(i-1)} \right\|\\
&\quad + N_2 \sum_{i\in {\mathcal{I}}}  \left\| Q_{\alpha}^{\left( i \right)} \left( x_n^{(i-1)} \right) - x_n \right\| \bigg\},
\end{align*}
where $N_2 := \max_{i\in {\mathcal{I}}} N_2^{(i)} < \infty$.
\end{enumerate}
\end{lem}

\begin{proof}
(i) 
The sequence $(x_n^{(i)})_{n\in \mathbb{N}}$ $(i\in\mathcal{I})$ in Algorithm \ref{algo:1} is defined by
$x_n^{(i)} := Q_\alpha^{(i)} (x_n) -\lambda_n g_n^{(i)}$ while $(x_n^{(i)})_{n\in \mathbb{N}}$ $(i\in\mathcal{I})$ in Algorithm \ref{algo:2} is defined by
$x_n^{(i)} := Q_\alpha^{(i)} (x_n^{(i-1)}) -\lambda_n g_n^{(i)}$.
Hence, by replacing $x_n$ in the proof of Lemma \ref{lem:0}(i) by $x_{n}^{(i-1)}$, we find that, for all $n\in\mathbb{N}$, for all $i\in{\mathcal{I}}$, and for all $x\in X$, 
\begin{align}
\left\| x_n^{(i)} - x \right\|^2 
&\leq \left\| x_n^{(i-1)} - x \right\|^2 - 2 \alpha^{(i)}  \left( 1 - \alpha^{(i)} \right)\left\| x_n^{(i-1)} - Q^{(i)} \left(x_n^{(i-1)}\right) \right\|^2\nonumber\\ 
&\quad +2 N_1 \lambda_n^2 -2  \lambda_n \left\langle x_n^{(i-1)} - x, g_n^{\left(i\right)} \right\rangle,\label{bdd}
\end{align}
where $N_1 := \max_{i\in {\mathcal{I}}} N_1^{{(i)}^2}$ and $N_1 < \infty$ holds from Assumption \ref{a8}.
Therefore, for all $n\in\mathbb{N}$ and for all $x\in X$,
\begin{align*}
\left\| x_{n+1} - x \right\|^2
&= \left\| x_n^{(I)} - x \right\|^2\\
&\leq \left\| x_n -x \right\|^2 
- 2 \sum_{i\in {\mathcal{I}}} \alpha^{\left(i \right)} \left( 1 -\alpha^{\left(i\right)} \right)
\left\| x_n^{(i-1)} - Q^{\left(i \right)} \left( x_n^{(i-1)}  \right)  \right\|^2\\ 
&\quad + 2 I N_1 \lambda_n^2 - 2 \lambda_n\sum_{i\in {\mathcal{I}}} 
\left\langle x_n^{(i-1)} - x, g_n^{\left(i\right)} \right\rangle. 
\end{align*}

(ii)
The same discussion as in the proof of Lemma \ref{lem:0}(ii) implies that, for all $n\in\mathbb{N}$
and for all $x\in X$, 
\begin{align*}
\sum_{i\in {\mathcal{I}}} \left\langle x -  x_n^{\left(i-1\right)}, g_n^{\left(i\right)}  \right\rangle
&\leq f(x) - f(x_n) + \sum_{i\in {\mathcal{I}}} \left( f^{\left(i\right)}(x_n) 
  -  f^{\left(i \right)} \left( Q_{\alpha}^{\left( i \right)} \left( x_n^{\left(i-1\right)} \right) \right) \right)\\
&\quad +\sqrt{N_1} \sum_{i\in {\mathcal{I}}}  \left\| Q_{\alpha}^{\left( i \right)} \left( x_n^{\left(i-1\right)} \right) - x_n^{\left(i-1\right)} \right\|.
\end{align*}
Set $N_2 := \max_{i\in {\mathcal{I}}} N_2^{(i)} < \infty$.
Since $g\in \partial f^{(i)} (x_n)$ $(n\in\mathbb{N})$ implies that, for all $i\in {\mathcal{I}}$ and for all $n\in\mathbb{N}$,
$f^{(i)} (x_n) - f^{(i)} (Q_\alpha^{(i)} (x_n^{(i-1)})) \leq \langle x_n - Q_\alpha^{(i)} (x_n^{(i-1)}), g \rangle
\leq N_2 \|x_n - Q_\alpha^{(i)} (x_n^{(i-1)})\|$,
it is found that 
\begin{align*}
\sum_{i\in {\mathcal{I}}} \left\langle x -  x_n, g_n^{\left(i\right)}  \right\rangle
&\leq f(x) - f(x_n) 
+ \sqrt{N_1} \sum_{i\in {\mathcal{I}}}  \left\| Q_{\alpha}^{\left( i \right)} \left( x_n^{(i-1)} \right) - x_n^{(i-1)} \right\|\\
&\quad + N_2 \sum_{i\in {\mathcal{I}}}  \left\| Q_{\alpha}^{\left( i \right)} \left( x_n^{(i-1)} \right) - x_n \right\|.
\end{align*}
Accordingly, Lemma \ref{lem:0-1}(i) leads to Lemma \ref{lem:0-1}(ii).
This completes the proof.
\end{proof}

\subsection{Constant step-size rule}\label{subsec:4.1}
Let us perform a convergence analysis of Algorithm \ref{algo:2} with a constant step size.

\begin{thm}\label{thm:3}
Suppose that Assumptions (A1)--(A4), (A7), \ref{assum:1}, and \ref{a8} hold. Then, $(x_n)_{n\in\mathbb{N}}$ in Algorithm \ref{algo:2} satisfies the relations
\begin{align*}
&\liminf_{n\to\infty}  \left\| x_n^{\left(i-1\right)} -  Q^{\left(i \right)} \left(x_n^{\left(i-1\right)} \right)  \right\|^2 
  \leq \frac{N_\lambda \lambda}{\alpha^{\left(i\right)} \left(1-\alpha^{\left(i\right)} \right)}
  \text{ } \left(i\in {\mathcal{I}} \right),\\
&\liminf_{n\to \infty} f \left(x_n \right) 
\leq f^\star  + I \left( \frac{I\sqrt{N_1} N_2}{2} + N_1 \right) \lambda
+N_2 \sum_{i\in{\mathcal{I}}} \sum_{j=1}^{i-1} 
\sqrt{\frac{\left(1-\alpha^{(j)}\right) N_\lambda \lambda}{\alpha^{\left(j \right)}}}\\
&\qquad\qquad\qquad\quad + \left( \sqrt{N_1} + N_2\right) \sum_{i\in{\mathcal{I}}} \sqrt{\frac{\left(1-\alpha^{(j)}\right) N_\lambda \lambda}{\alpha^{\left(i \right)}}},
\end{align*} 
where $N_1$ and $N_2$ are as in Lemma \ref{lem:0-1} and, for some $x \in X$, 
$N_\lambda := \sup_{n\in\mathbb{N}} (I N_1 \lambda + |\sum_{i\in {\mathcal{I}}} \langle x - x_n^{(i-1)}, g_n^{(i)}\rangle|)$. 
\end{thm}

\begin{proof}
Choose $x\in X$ arbitrarily 
and set $N_\lambda := \sup_{n\in\mathbb{N}} (I N_1 \lambda + |\sum_{i\in {\mathcal{I}}} \langle x - x_n^{(i-1)}, g_n^{(i)}\rangle|)$.
Since Theorem \ref{thm:3} holds when $N_\lambda = \infty$, it can be assumed that $N_\lambda < \infty$.
First, let us show that
\begin{align}\label{infe:1}
\liminf_{n\to\infty} \sum_{i\in {\mathcal{I}}} \alpha^{\left(i\right)} \left(  1- \alpha^{\left(i\right)} \right) 
\left\| x_n^{\left(i-1\right)} - Q^{\left(i \right)} \left(x_n^{\left(i-1\right)} \right) \right\|^2 \leq N_\lambda \lambda.
\end{align}
Let us assume that \eqref{infe:1} does not hold. 
Accordingly, from the same discussion as in the proof of Theorem \ref{thm:1},
$\delta$ $(> 0)$ and $n_0$ $(\in \mathbb{N})$ can be chosen such that, for all $n \geq n_0$,
$\sum_{i\in {\mathcal{I}}} \alpha^{(i)} (  1- \alpha^{(i)} ) 
\| x_n^{(i-1)} - Q^{(i )} (x_n^{(i-1)} ) \|^2 > N_\lambda \lambda 
+  \delta$.
Hence, Lemma \ref{lem:0-1}(i) leads to the finding that, for all $n \geq n_0$,
\begin{align*}
\left\| x_{n+1} - x \right\|^2 
&< \left\| x_n -x \right\|^2  
- 2 \left\{ N_\lambda \lambda + \delta \right\} + 2 N_\lambda \lambda\\
&=  \left\| x_n -x \right\|^2 -  2 \delta.
\end{align*}
Therefore, induction shows that
$0 \leq \left\| x_{n+1} - x \right\|^2 < \left\| x_{n_0} -x \right\|^2 -  2 \delta (n + 1 -n_{0} )$
$(n \geq n_0)$,
which is a contradiction. 
Therefore, \eqref{infe:1} holds. 
This means that
\begin{align}\label{infe:2}
\liminf_{n\to\infty} \left\| x_n^{\left(i-1\right)} - Q^{\left(i \right)} \left(x_n^{\left(i-1\right)} \right) \right\|^2 \leq \frac{N_\lambda \lambda}{\alpha^{\left(i \right)} \left(1- \alpha^{\left(i\right)} \right)} \text{ } 
\left( i\in {\mathcal{I}} \right).
\end{align}

Let $i\in {\mathcal{I}}$ be fixed arbitrarily. 
Inequality \eqref{infe:2} and the property of the limit inferior of $(\| x_n^{(i)} - Q^{(i)} (x_n^{(i)}) \|^2)_{n\in\mathbb{N}}$ guarantee the existence of a subsequence $(x_{n_k}^{(i)})_{k\in\mathbb{N}}$ of $(x_n^{(i)})_{n\in\mathbb{N}}$ such that 
$\lim_{k\to\infty} \| x_{n_k}^{(i-1)} - Q^{(i)} (x_{n_k}^{(i-1)}) \|^2 =
\liminf_{n\to\infty} \| x_n^{(i-1)} - Q^{(i)} (x_n^{(i-1)} ) \|^2
\leq N_\lambda \lambda/(\alpha^{(i)}(1- \alpha^{(i)}))$.
Therefore, for all $\epsilon > 0$, there exists $k_0 \in \mathbb{N}$ such that, for all $k \geq k_0$,
\begin{align}\label{4-1}
\left\| x_{n_k}^{\left(i-1\right)} - Q^{\left(i \right)} \left(x_{n_k}^{\left(i-1\right)} \right) \right\| \leq 
\sqrt{\frac{N_\lambda \lambda}{\alpha^{\left(i \right)} \left(1- \alpha^{\left(i\right)} \right)} + \epsilon}.
\end{align}
Now, let us prove that, for all $k \geq k_0$,
\begin{align}\label{3-1}
\begin{split}
\liminf_{n\to\infty} f(x_n)
&\leq f^\star  + I N_1 \lambda 
+ \sqrt{N_1} \sum_{i\in{\mathcal{I}}} \left\| x_{n_k}^{\left(i-1\right)} - Q_{\alpha}^{\left(i \right)} \left( x_{n_k}^{\left(i-1\right)} \right)   \right\|\\ 
&\quad + N_2 \sum_{i\in{\mathcal{I}}} \left\| x_{n_k} - Q_{\alpha}^{\left(i \right)} \left( x_{n_k}^{\left(i-1\right)} \right)   \right\| + 2 \epsilon.
\end{split}
\end{align}
Let us assume that \eqref{3-1} does not hold for all $k \geq k_0$. 
Then, (A5) and the property of the limit inferior of $(f(x_n))_{n\in\mathbb{N}}$ guarantee that $x^\star \in X^\star$ and $n_1 \in \mathbb{N}$ exist such that, for all $n \geq n_1$,
\begin{align*}
f(x_n) - f \left( x^\star  \right)
&> I N_1 \lambda 
+ \sqrt{N_1} \sum_{i\in{\mathcal{I}}} \left\| x_{n}^{\left(i-1\right)} - Q_{\alpha}^{\left(i \right)} \left( x_{n}^{\left(i-1\right)} \right)   \right\|\\ 
&\quad + N_2 \sum_{i\in{\mathcal{I}}} \left\| x_{n} - Q_{\alpha}^{\left(i \right)} \left( x_{n}^{\left(i-1\right)} \right)   \right\| + \epsilon.
\end{align*}
Therefore, Lemma \ref{lem:0-1}(ii) means that, for all $n \geq n_1$,
\begin{align*}
&\quad \left\| x_{n+1} - x^\star \right\|^2\\ 
&\leq \left\| x_n -x^\star \right\|^2 + 2 I N_1 \lambda^2
 + 2\lambda \left( f \left(x^\star \right) - f(x_n)  \right)\\
&\quad + 2\lambda \bigg\{ \sqrt{N_1} \sum_{i\in {\mathcal{I}}} 
 \left\| x_n^{\left(i-1\right)} - Q_\alpha^{\left(i \right)} \left( x_n^{\left(i-1\right)}  \right)  \right\|
+ N_2 \sum_{i\in {\mathcal{I}}} 
 \left\| x_n - Q_\alpha^{\left(i \right)} \left( x_n^{\left(i-1\right)}  \right)  \right\| \bigg\}\\
&< \left\| x_n -x^\star \right\|^2 + 2 I N_1 \lambda^2\\
&\quad - 2\lambda 
\bigg\{ I N_1 \lambda 
+ \sqrt{N_1} \sum_{i\in {\mathcal{I}}} 
 \left\| x_n^{\left(i-1\right)} - Q_\alpha^{\left(i \right)} \left( x_n^{\left(i-1\right)}  \right)  \right\|\\
&\quad + N_2 \sum_{i\in {\mathcal{I}}} 
 \left\| x_n - Q_\alpha^{\left(i \right)} \left( x_n^{\left(i-1\right)}  \right)  \right\|
+ \epsilon \bigg\} + 2\lambda 
\bigg\{
\sqrt{N_1} \sum_{i\in {\mathcal{I}}} 
 \left\| x_n^{\left(i-1\right)} - Q_\alpha^{\left(i \right)} \left( x_n^{\left(i-1\right)}  \right)  \right\|\\
&\quad + N_2 \sum_{i\in {\mathcal{I}}} 
 \left\| x_n - Q_\alpha^{\left(i \right)} \left( x_n^{\left(i-1\right)}  \right)  \right\|
\bigg\}\\
&=  \left\| x_n -x^\star \right\|^2 - 2\lambda \epsilon,
\end{align*}
which means that $0\leq \| x_{n+1} - x^\star \|^2 < \| x_{n_1} - x^\star \|^2 -2 \lambda \epsilon (n+1 -n_1)$ $(n\geq n_1)$. 
Hence, there is a contradiction.
Accordingly, \eqref{3-1} holds for all $k \geq k_0$.

Furthermore, the triangle inequality implies that, for all $k \geq k_0$,
\begin{align*}
&\quad \sqrt{N_1} \sum_{i\in{\mathcal{I}}} \left\| x_{n_k}^{\left(i-1\right)} - Q_{\alpha}^{\left(i \right)} \left( x_{n_k}^{\left(i-1\right)} \right)   \right\| 
 + N_2 \sum_{i\in{\mathcal{I}}} \left\| x_{n_k} - Q_{\alpha}^{\left(i \right)} \left( x_{n_k}^{\left(i-1\right)} \right)   \right\|\\
&\leq \sqrt{N_1} \sum_{i\in{\mathcal{I}}} \left\| x_{n_k}^{\left(i-1\right)} - Q_{\alpha}^{\left(i \right)} \left( x_{n_k}^{\left(i-1\right)} \right)   \right\| 
+  N_2 \sum_{i\in{\mathcal{I}}} \left\| x_{n_k}^{(0)} - x_{n_k}^{\left(i-1\right)} \right\|\\
&\quad + N_2 \sum_{i\in{\mathcal{I}}} \left\| x_{n_k}^{\left(i-1\right)} - Q_{\alpha}^{\left(i \right)} \left( x_{n_k}^{\left(i-1\right)} \right)   \right\|\\
&\leq \left( \sqrt{N_1} + N_2\right) \sum_{i\in{\mathcal{I}}} \left( 1- \alpha^{(i)} \right) \left\| x_{n_k}^{\left(i-1\right)} - Q^{\left(i \right)} \left( x_{n_k}^{\left(i-1\right)} \right)   \right\|\\
&\quad + N_2 \sum_{i\in{\mathcal{I}}} \sum_{j=1}^{i-1} \left\| x_{n_k}^{(j-1)} - x_{n_k}^{\left(j\right)} \right\|.
\end{align*}
Moreover, the definition of $x_n^{(i)}$ $(n\in \mathbb{N}, i\in \mathcal{I})$ and the triangle inequality mean that, for all $k \geq k_0$,
\begin{align*}
\sum_{i\in{\mathcal{I}}} \sum_{j=1}^{i-1} \left\| x_{n_k}^{(j-1)} - x_{n_k}^{\left(j\right)} \right\|
&\leq \sum_{i\in{\mathcal{I}}} \sum_{j=1}^{i-1} 
\left\| x_{n_k}^{(j-1)} - Q_\alpha^{\left(j\right)} \left( x_{n_k}^{\left(j-1\right)} \right)  \right\|
 +  \sum_{i\in{\mathcal{I}}} \sum_{j=1}^{i-1} \sqrt{N_1} \lambda\\ 
&=  \sum_{i\in{\mathcal{I}}} \sum_{j=1}^{i-1} 
 \left\| x_{n_k}^{(j-1)} - Q_\alpha^{\left(j\right)} \left( x_{n_k}^{\left(j-1\right)} \right)  \right\|
 + \frac{I^2 \sqrt{N_1} \lambda}{2}.
\end{align*}
Accordingly, \eqref{4-1} guarantees that
\begin{align*}
&\quad \sqrt{N_1} \sum_{i\in{\mathcal{I}}} \left\| x_{n_k}^{\left(i-1\right)} - Q_{\alpha}^{\left(i \right)} \left( x_{n_k}^{\left(i-1\right)} \right)   \right\| 
 + N_2 \sum_{i\in{\mathcal{I}}} \left\| x_{n_k} - Q_{\alpha}^{\left(i \right)} \left( x_{n_k}^{\left(i-1\right)} \right)   \right\|\\
&\leq \left( \sqrt{N_1} + N_2\right) \sum_{i\in{\mathcal{I}}} \left( 1- \alpha^{(i)} \right) \sqrt{\frac{N_\lambda \lambda}{\alpha^{\left(i \right)} \left(1- \alpha^{\left(i\right)} \right)} + \epsilon}\\
&\quad +N_2 \sum_{i\in{\mathcal{I}}} \sum_{j=1}^{i-1} 
\left(1-\alpha^{(j)}\right)\sqrt{\frac{N_\lambda \lambda}{\alpha^{\left(j \right)} \left(1- \alpha^{\left(j\right)} \right)} + \epsilon}
+ \frac{I^2 \sqrt{N_1} N_2 \lambda}{2}.
\end{align*}
Therefore, \eqref{3-1} leads to the finding that, for all $\epsilon > 0$,
\begin{align*}
\liminf_{n\to\infty} f(x_n)
&\leq f^\star  + I \left(\frac{I\sqrt{N_1}N_2}{2} + N_1 \right) \lambda\\ 
&\quad + \left( \sqrt{N_1} + N_2\right) \sum_{i\in{\mathcal{I}}} \left( 1- \alpha^{(i)} \right) \sqrt{\frac{N_\lambda \lambda}{\alpha^{\left(i \right)} \left(1- \alpha^{\left(i\right)} \right)} + \epsilon}\\
&\quad +N_2 \sum_{i\in{\mathcal{I}}} \sum_{j=1}^{i-1} 
\left(1-\alpha^{(j)}\right)\sqrt{\frac{N_\lambda \lambda}{\alpha^{\left(j \right)} \left(1- \alpha^{\left(j\right)} \right)} + \epsilon}
+ 2\epsilon.
\end{align*}
Hence, the arbitrary property of $\epsilon$ $(>0)$ leads to the deduction that 
\begin{align*}
&\quad \liminf_{n\to\infty} f(x_n)\\
&\leq f^\star  + I \left( \frac{I \sqrt{N_1} N_2}{2} + N_1 \right) \lambda
 +N_2 \sum_{i\in{\mathcal{I}}} \sum_{j=1}^{i-1} 
\left(1-\alpha^{(j)}\right)\sqrt{\frac{N_\lambda \lambda}{\alpha^{\left(j \right)} \left(1- \alpha^{\left(j\right)} \right)}}\\
&\quad + \left( \sqrt{N_1} + N_2\right) \sum_{i\in{\mathcal{I}}} \left( 1- \alpha^{(i)} \right) \sqrt{\frac{N_\lambda \lambda}{\alpha^{\left(i \right)} \left(1- \alpha^{\left(i\right)} \right)}}\\
&= f^\star  + I \left( \frac{I \sqrt{N_1} N_2}{2} + N_1 \right) \lambda\\ 
&\quad + \left( \sqrt{N_1} + N_2\right) \sum_{i\in{\mathcal{I}}} \sqrt{\frac{\left(1-\alpha^{(j)}\right) N_\lambda  \lambda}{\alpha^{\left(i \right)}}}
+N_2 \sum_{i\in{\mathcal{I}}} \sum_{j=1}^{i-1} 
\sqrt{\frac{\left(1-\alpha^{(j)}\right) N_\lambda \lambda}{\alpha^{\left(j \right)}}}.
\end{align*}
This completes the proof.
\end{proof}

\subsection{Diminishing step-size rule}\label{subsec:4.2}
Let us perform a convergence analysis of Algorithm \ref{algo:2} with a diminishing step size.

\begin{thm}\label{thm:4}
Suppose that Assumptions (A1)--(A4), (A7), \ref{assum:2}, and \ref{a8_1} hold. 
Then there exists a subsequence of $(x_n^{(i)})_{n\in\mathbb{N}}$ $(i\in\mathcal{I})$ generated by Algorithm \ref{algo:2} 
that weakly converges to a point in $X^\star$.
If either (i) or (ii) in Theorem \ref{thm:2} holds,
$(x_n^{(i)})_{n\in\mathbb{N}}$ $(i\in\mathcal{I})$ strongly converges to a unique point in $X^\star$.\footnote{Figure \ref{fig:10} shows the existence of a subsequence of $(x_n)_{n\in \mathbb{N}}$ generated by Algorithm \ref{algo:2} that converges to a solution to Problem \ref{prob:2} when all $f^{(i)}$ are convex while Figure \ref{fig:12} indicates the convergence of $(x_n)_{n\in \mathbb{N}}$ generated by Algorithm \ref{algo:2} to the solution to Problem \ref{prob:2}
when only $f^{(1)}$ is strongly convex.}
\end{thm}

\begin{proof}
Case 1: 
Suppose there exists $m_0 \in \mathbb{N}$ such that 
$\| x_{n+1} - x^\star \| \leq \| x_n -x^\star \|$ for all $n \geq m_0$ and for all $x^\star \in X^\star$. 
Then, there exists $\lim_{n\to\infty} \| x_n - x^\star \|$ for all $x^\star \in X^\star$. 
Hence, $(x_n)_{n\in\mathbb{N}}$ is bounded.
From the quasi-nonexpansivity of $Q_\alpha^{(1)}$, $(Q_\alpha^{(1)}(x_n))_{n\in\mathbb{N}}$ is also bounded.
Hence, Proposition \ref{prop:sub} guarantees the boundedness of $(g_n^{(1)})_{n\in \mathbb{N}}$. 
Inequality \eqref{bdd} when $i=1$, $x_n^{(0)} := x_n$ $(n\in\mathbb{N})$, and (C2) lead to the boundedness of $(x_n^{(1)})_{n\in \mathbb{N}}$.
Therefore, induction shows that $(x_n^{(i)})_{n\in \mathbb{N}}$ and $(g_n^{(i)})_{n\in \mathbb{N}}$ $(i\in \mathcal{I})$ are bounded; i.e., $N_1$ and $N_2$ defined as in Lemma \ref{lem:0-1} are finite.
Lemma \ref{lem:0-1}(i) implies that, for all $n \geq m_0$ and for all $x^\star \in X^\star$,
\begin{align*}
&\quad 2 \sum_{i\in{\mathcal{I}}} \alpha^{\left(i \right)} \left(1-\alpha^{\left(i\right)} \right)
\left\| x_n^{\left(i-1\right)} - Q^{\left(i\right)} \left(x_n^{\left(i-1\right)} \right) \right\|^2\\
&\leq \left\| x_n - x^\star \right\|^2 - \left\| x_{n+1} - x^\star \right\|^2 + 2 I N_1 \lambda_n^2 
 - 2\lambda_n \sum_{i\in{\mathcal{I}}} \left\langle x_n^{\left(i-1\right)} - x^\star, g_n^{\left(i\right)} \right\rangle.
\end{align*}
Accordingly, the existence of $\lim_{n\to\infty} \| x_n - x^\star \|$ $(x^\star \in X^\star)$ and (C2) guarantee that
\begin{align}\label{qi-1}
\lim_{n\to\infty} \left\| x_n^{\left(i-1\right)} - Q_\alpha^{(i)} \left(x_n^{\left(i-1\right)} \right) \right\| =
\lim_{n\to\infty} \left\| x_n^{\left(i-1\right)} - Q^{\left(i\right)} \left(x_n^{\left(i-1\right)} \right) \right\| = 0 \text{ } \left(i\in{\mathcal{I}}\right).
\end{align}
Moreover, since 
$\|x_n^{(i-1)} - x_n^{(i)}\| = \| x_n^{(i-1)} - Q_\alpha^{(i)} (x_n^{(i-1)}) + \lambda_n g_n^{(i)} \|
\leq \| x_n^{(i-1)} - Q_\alpha^{(i)} (x_n^{(i-1)})\| +  \sqrt{N_1}\lambda_n$ $(n\in\mathbb{N}, i\in \mathcal{I})$,
\eqref{qi-1} and (C2) ensure that $\lim_{n\to\infty} \|x_n^{(i-1)} - x_n^{(i)}\| = 0$ $(i\in \mathcal{I})$.
Since the triangle inequality implies that 
$\| x_n - x_n^{(i-1)} \| \leq \sum_{j=1}^{i-1} \| x_n^{(j-1)} - x_n^{(j)} \|$ $(n\in\mathbb{N}, i\in \mathcal{I})$,
\begin{align}\label{qi-4}
\lim_{n\to\infty} \left\| x_n - x_n^{\left(i-1\right)} \right\| = 0 \text{ } (i\in\mathcal{I}).
\end{align}
From $\| x_n - Q_\alpha^{(i)} (x_n^{(i-1)}) \| \leq \| x_n - x_n^{(i-1)} \| + \| x_n^{(i-1)} - Q_\alpha^{(i)} (x_n^{(i-1)}) \|$ $(n\in\mathbb{N}, i\in \mathcal{I})$,
\begin{align}\label{qi-3}
\lim_{n\to\infty} \left\| x_n - Q_\alpha^{(i)} \left(x_n^{\left(i-1\right)}\right) \right\| = 0
\text{ } (i\in\mathcal{I}).
\end{align}
Here, let us define that, for all $n \in \mathbb{N}$ and for all $x\in X$,
\begin{align}\label{Nn}
\begin{split}
N_n (x) &:= f(x_n) - f(x) - \sqrt{N_1} \sum_{i\in{\mathcal{I}}}\left\| Q_\alpha^{(i)} \left(x_n^{\left(i-1\right)} \right) - x_n^{\left(i-1\right)}  \right\|\\
&\quad - N_2 \sum_{i\in{\mathcal{I}}}\left\| Q_\alpha^{(i)} \left(x_n^{\left(i-1\right)} \right) - x_n \right\|
- I N_1 \lambda_n.
\end{split}
\end{align}
Then, Lemma \ref{lem:0-1}(ii) leads to the finding that, for all $n \in \mathbb{N}$ and for all $x\in X$,
\begin{align}\label{iii-1}
2\lambda_n N_n(x) \leq \left\| x_n - x \right\|^2 - \left\| x_{n+1} - x \right\|^2.
\end{align} 
A discussion similar to the one for obtaining $\liminf_{n\to\infty} M_n(x) \leq 0$ $(x\in X)$
guarantees that 
$\liminf_{n\to\infty} N_n(x) \leq 0$ $(x\in X)$, which, together with (C2), \eqref{qi-1}, and \eqref{qi-3}, implies that
$\liminf_{n\to\infty} f (x_n ) \leq f(x)$ $( x\in X )$.
Accordingly, there exists a subsequence $(x_{n_l})_{l\in \mathbb{N}}$ of $(x_n)_{n\in\mathbb{N}}$ such that 
$\lim_{l\to\infty} f ( x_{n_l} ) =  \liminf_{n\to\infty} f(x_n ) \leq f(x)$ $( x\in X )$.
Since $(x_{n_l})_{l\in \mathbb{N}}$ is bounded, there exists 
$(x_{n_{l_m}})_{m\in \mathbb{N}}$ $(\subset (x_{n_l})_{l\in \mathbb{N}})$ such that $(x_{n_{l_m}})_{m\in \mathbb{N}}$ weakly converges to $x_* \in H$.
Equation \eqref{qi-4} guarantees that $(x_{n_{l_m}}^{(i-1)})_{m\in \mathbb{N}}$ $(i\in\mathcal{I})$ weakly converges to $x_*$.
Thus, (A6) and \eqref{qi-1} ensure that $x_* \in X$.
From the same discussion as in the proof of Theorem \ref{thm:2}, 
$(x_n)_{n\in\mathbb{N}}$ weakly converges to a point in $X^\star$.
Moreover, \eqref{qi-4} implies that $(x_n^{(i)})_{n\in\mathbb{N}}$ $(i\in\mathcal{I})$
weakly converges to a point in $X^\star$.

Case 2: 
Suppose that $x_0^* \in X^\star$ and $(x_{n_j})_{j\in\mathbb{N}}$ $(\subset (x_n)_{n\in\mathbb{N}})$ exist such that 
$\Gamma_{n_j} := \| x_{n_j} - x_0^* \| < \| x_{n_j +1} - x_0^* \|$ for all $j\in\mathbb{N}$.
Assumption \ref{a8_1} and the quasi-nonexpansivity of $Q_\alpha^{(i)}$ $(i\in\mathcal{I})$ guarantee the boundedness of 
$(Q_\alpha^{(i)} (x_n^{(i-1)}))_{n\in\mathbb{N}}$ $(i\in\mathcal{I})$.
Hence, Proposition \ref{prop:sub} ensures that $N_1, N_2 < \infty$.
Proposition \ref{mainge} means the existence of $m_1 \in \mathbb{N}$ such that 
$\Gamma_{\tau(n)} < \Gamma_{\tau(n)+1}$ for all $n \geq m_1$, where $\tau(n)$ is as in Proposition \ref{mainge}.
Lemma \ref{lem:0-1}(i) means that, for all $n\geq m_1$, 
\begin{align*}
&\quad 2 \sum_{i\in{\mathcal{I}}} \alpha^{\left(i \right)} \left(1-\alpha^{\left(i\right)} \right)
\left\| x_{\tau(n)}^{\left(i-1\right)} - Q^{\left(i\right)} \left(x_{\tau(n)}^{\left(i-1\right)} \right) \right\|^2\\
&\leq \Gamma_{\tau(n)}^2 - \Gamma_{\tau(n) +1}^2 + 2 I N_1 \lambda_{\tau(n)}^2 
 - 2\lambda_{\tau(n)} \sum_{i\in{\mathcal{I}}} \left\langle x_{\tau(n)}^{\left(i-1\right)} - x_0^*, g_{\tau(n)}^{\left(i\right)} \right\rangle\\
&< \left( 2 I N_1 \lambda_{\tau(n)} -2 \sum_{i\in{\mathcal{I}}} \left\langle x_{\tau(n)}^{\left(i-1\right)} - x_0^*, g_{\tau(n)}^{\left(i\right)} \right\rangle \right) \lambda_{\tau(n)},
\end{align*}
which, together with $\lim_{n\to\infty} \tau(n)=\infty$ and (C2), implies that 
\begin{align}\label{Q0}
\lim_{n\to\infty} \left\| x_{\tau(n)}^{\left(i-1\right)} - Q^{\left(i\right)} \left(x_{\tau(n)}^{\left(i-1\right)} \right) \right\| = 0 \text{ } 
\left(i\in {\mathcal{I}} \right).
\end{align}
The same discussions for obtaining \eqref{qi-1}, \eqref{qi-4}, and \eqref{qi-3} imply that
\begin{align}
&\lim_{n\to\infty} \left\| x_{\tau(n)}^{\left(i-1\right)} - Q_\alpha^{\left(i\right)} \left(x_{\tau(n)}^{\left(i-1\right)} \right) \right\| = 0 \text{ } (i\in\mathcal{I}),\label{Q1}\\
&\lim_{n\to\infty} \left\| x_{\tau(n)} - x_{\tau(n)}^{\left(i-1\right)} \right\| = 0 \text{ } (i\in\mathcal{I}),\label{Q2}\\
&\lim_{n\to\infty} \left\| x_{\tau(n)} - Q_\alpha^{\left(i\right)} \left(x_{\tau(n)}^{\left(i-1\right)} \right) \right\| = 0 \text{ } (i\in\mathcal{I}).\label{Q3}
\end{align}
Inequality \eqref{iii-1} and $\lambda_{\tau(n)} > 0$ $(n \geq m_1)$ mean that
$N_{\tau(n)} (x_0^*) < 0$ $(n \geq m_1)$; i.e., for all $n \geq m_1$,
\begin{align}\label{ftau}
\begin{split}
f \left(x_{\tau(n)} \right) - f^\star
&< 
\sqrt{N_1} \sum_{i\in{\mathcal{I}}}\left\| Q_\alpha^{(i)} \left(x_{\tau(n)}^{\left(i-1\right)} \right) - x_{\tau(n)}^{\left(i-1\right)}  \right\|
+ I N_1 \lambda_{\tau(n)}\\
&\quad + N_2 \sum_{i\in{\mathcal{I}}}\left\| Q_\alpha^{(i)} \left(x_{\tau(n)}^{\left(i-1\right)} \right) - x_{\tau(n)} \right\|.
\end{split}
\end{align}
Accordingly, (C2), \eqref{Q1}, and \eqref{Q3} imply that  
$\limsup_{n\to\infty} f (x_{\tau(n)} ) \leq f^\star$,
which implies that, for any subsequence $(x_{\tau(n_k)})_{k\in\mathbb{N}}$ $(\subset (x_{\tau(n)})_{n \geq m_1})$,
$\lim_{k\to\infty} f (x_{\tau(n_k)} ) \leq  \limsup_{n\to\infty} f (x_{\tau(n)} ) \leq f^\star$.
From the boundedness of $(x_{\tau(n_k)})_{k\in\mathbb{N}}$, 
there is $(x_{\tau(n_{k_l})})_{l\in\mathbb{N}}$ $(\subset (x_{\tau(n_k)})_{k\in\mathbb{N}})$, which weakly converges to $x_\star \in H$.
Equation \eqref{Q2} implies that $(x_{\tau(n_{k_l})}^{(i-1)})_{l\in\mathbb{N}}$ $(i\in\mathcal{I})$ weakly converges to $x_\star$.
Hence, (A6) and \eqref{Q0} lead to $x_\star \in X$.
The same discussion as in the proof of Theorem \ref{thm:2} guarantees that $x_\star \in X^\star$.
Therefore, there exists a subsequence of $(x_n^{(i)})_{n\in\mathbb{N}}$ $(i\in \mathcal{I})$
that weakly converges to a point in $X^\star$.

Let us assume that either (i) or (ii) is satisfied. 
A discussion similar to the one for proving the strong convergence of $(x_n)_{n\in \mathbb{N}}$ in Algorithm \ref{algo:1} to a unique point in $X^\star$ guarantees 
that $(x_n)_{n\in \mathbb{N}}$ in Algorithm \ref{algo:2} strongly converges to 
$x^\star \in X^\star$.
From $\lim_{n\to \infty} \| x_n - x_n^{(i-1)} \| = 0$ $(i\in \mathcal{I})$ (see \eqref{qi-4} and \eqref{Q2}), we can conclude that $(x_n^{(i)})$ $(i\in \mathcal{I})$ strongly converges to $x^\star$. 
This completes the proof.
\end{proof}

Regarding the relationship between the proposed algorithms (Algorithms \ref{algo:1} and \ref{algo:2}) and the distributed random projection method \cite{lee2013}, we have the following remark.
\begin{rem}\label{rem:1}
{\em
Suppose that user $i$'s objective function $f^{(i)}$ is convex and differentiable and that user $i$'s constraint set $C^{(i)}$ is 
defined as the intersection of finitely many simple closed convex constraints; i.e., 
\begin{align*}
C^{(i)} := \bigcap_{k \in \mathcal{J}^{(i)}} C_k^{(i)},
\end{align*} 
where $\mathcal{J}^{(i)}$ is finite and $C_k^{(i)}$ $(k\in \mathcal{J}^{(i)})$ is a nonempty, closed convex set of $\mathbb{R}^N$
such that $P_{C_k^{(i)}}$ can be computed efficiently. 
At iteration $n$ of the method \cite{lee2013}, user $i$ calculates the weighted average of the $x_n^{(j)}$ received
from its local neighbors $j$ and determines the iteration value by using the gradient information of its own objective function and 
the metric projection onto a constraint $C_{\Omega_n^{(i)}}^{(i)}$ $(\Omega_n^{(i)} \in \mathcal{J}^{(i)})$ selected {\em randomly} from its constraint set $C^{(i)}$; i.e.,
\begin{align}\label{lee}
\begin{split}
&v_n^{(i)} := \sum_{j \in N_n^{(i)}} w_{ij,n}x_n^{(j)},\\
&x_{n+1}^{(i)} := P_{C_{\Omega_n^{(i)}}^{(i)}} \left( v_n^{(i)} - \alpha_n \nabla f^{(i)} \left( v_n^{(i)}  \right)  \right),
\end{split}
\end{align} 
where $N_n^{(i)}$ stands for the set of user $i$ and the users that send information to user $i$, 
$w_{ij,n} \geq 0$ $(j\in N_n^{(i)})$ with $\sum_{j \in N_n^{(i)}} w_{ij,n} = 1$ $(i\in \mathcal{I})$, and $\alpha_n > 0$.
Proposition 1 in \cite{lee2013} indicates that, under certain assumptions, the sequence $(x_n^{(i)})_{n\in\mathbb{N}}$ $(i\in\mathcal{I})$ generated by Algorithm \eqref{lee}
converges almost surely to the minimizer of $\sum_{i\in\mathcal{I}} f^{(i)}$ over $\bigcap_{i\in\mathcal{I}} C^{(i)}$.  

Algorithm \ref{algo:1} (resp. Algorithm \ref{algo:2}) can be applied to the problem considered in \cite{lee2013}
under Assumption (A5) (resp. Assumption (A7)) and the assumption that
user $i$ can use {\em all} $P_{C_k^{(i)}}$ $(k\in \mathcal{J}^{(i)})$ 
at each iteration. 
Since the product of metric projections or the weighted average of metric projections is a special case of a quasi-nonexpansive mapping,
$Q^{(i)}$ in Algorithms \ref{algo:1} and \ref{algo:2} can be given, for example, by 
\begin{align*}
Q^{(i)} := \prod_{k\in \mathcal{J}^{(i)}} P_{C_k^{(i)}} \text{ or } Q^{(i)} := \sum_{k\in \mathcal{J}^{(i)}} w_k^{(i)} P_{C_k^{(i)}},
\end{align*}
where $(w_k^{(i)})_{k\in \mathcal{J}^{(i)}}$ $(i\in\mathcal{I})$ satisfies $\sum_{k\in \mathcal{J}^{(i)}} w_k^{(i)} = 1$.  
}
\end{rem}

\subsection{Convergence rate analysis of Algorithm \ref{algo:2} with diminishing step size}\label{subsec:4.3}
Here we first discuss the rate of convergence of Algorithm \ref{algo:2} for unconstrained nonsmooth convex optimization.

\begin{cor}\label{cor:3}
Consider Problem \ref{prob:1} when $Q^{(i)} = \mathrm{Id}$ $(i\in \mathcal{I})$
and suppose that the assumptions in Theorem \ref{thm:4} hold.
Then, for a large enough $n \in \mathbb{N}$,
\begin{align*}
f(x_n) - f^\star \leq I \left\{ \frac{(I-1)}{2} \sqrt{N_1} N_2  + N_1 \right\} \lambda_n,
\end{align*}
where $N_1 := \max_{i\in \mathcal{I}} N_1^{{(i)}^2} < \infty$, 
$N_2 := \max_{i\in \mathcal{I}} N_2^{{(i)}} < \infty$, 
and $N_1^{(i)}$ and $N_2^{(i)}$ $(i\in \mathcal{I})$ are defined as in Assumption \ref{a8}.
\end{cor}
Corollary \ref{cor:3} indicates that, 
when the same step size sequence is used, the efficiency of Algorithm \ref{algo:2} 
with $Q^{(i)} = \mathrm{Id}$ $(i\in \mathcal{I})$ may decrease as the number of users 
$I$ increases.
This can also be seen in Corollary \ref{cor:1}, indicating the rate of convergence of 
Algorithm \ref{algo:1} with $Q^{(i)} = \mathrm{Id}$ $(i\in \mathcal{I})$.
 
\begin{proof}
The triangle inequality ensures that
$\| x_n - x_n^{(i-1)} \| \leq \sum_{j=1}^{i-1} \| x_n^{(j-1)} - x_n^{(j)} \|$
$(n\in \mathbb{N})$, which, together with 
the definition of $x_n^{(i)}$ $(i\in \mathcal{I}, n\in \mathbb{N})$, implies 
$\| x_n - x_n^{(i-1)} \| \leq \sum_{j=1}^{i-1} \sqrt{N_1} \lambda_n = (i-1) \sqrt{N_1} \lambda_n$
$(i\in \mathcal{I}, n\in \mathbb{N})$.
Accordingly, for all $n\in \mathbb{N}$, 
\begin{align}\label{difference}
\sum_{i\in \mathcal{I}} \left\| x_n - x_n^{(i-1)} \right\| 
\leq \sum_{i\in \mathcal{I}} (i-1) \sqrt{N_1} \lambda_n = \frac{I(I-1)}{2} \sqrt{N_1} \lambda_n. 
\end{align}

In Case 1 in the proof of Theorem \ref{thm:4},
$\liminf_{n\to \infty} N_n(x^\star) \leq 0$ holds, where 
$N_n(x)$ $(n\in \mathbb{N}, x\in H)$ is defined by \eqref{Nn} and $\{x^\star\} = X^\star$.
The same discussion as in the proof of Corollary \ref{cor:1} leads to the existence of 
$k_1 \in \mathbb{N}$ such that, for all $n \geq k_1$,
$N_n(x^\star) \leq 0$.
From $Q^{(i)} = \mathrm{Id}$ $(i\in \mathcal{I})$, for all $n \geq k_1$,
$N_n(x^\star) = f(x_n) - f^\star - N_2 \sum_{i\in \mathcal{I}} \| x_n - x_n^{(i-1)} \|  - I N_1 \lambda_n \leq 0$.
Hence, \eqref{difference} implies that, for all $n \geq k_1$, 
\begin{align*}
f(x_n) - f^\star 
&\leq N_2 \sum_{i\in \mathcal{I}} \left\| x_n - x_n^{(i-1)} \right\|  + I N_1 \lambda_n\\
&\leq \frac{I(I-1)}{2} \sqrt{N_1} N_2 \lambda_n + I N_1 \lambda_n.
\end{align*}

In Case 2 in the proof of Theorem \ref{thm:4}, 
the condition $Q^{(i)} = \mathrm{Id}$ $(i\in \mathcal{I})$,
\eqref{ftau}, and \eqref{difference} lead to the existence of $k_2 \in \mathbb{N}$ such that,
for all $n \geq k_2$, 
$f(x_{\tau(n)}) - f^\star < N_2 \sum_{i\in \mathcal{I}} \| x_{\tau(n)} - x_{\tau(n)}^{(i-1)} \| + I N_1 \lambda_{\tau(n)} \leq (I(I-1)/2) \sqrt{N_1} N_2 \lambda_{\tau(n)} + I N_1 \lambda_{\tau(n)}$.
This completes the proof. 
\end{proof}

The following corollary establishes the rate of convergence of Algorithm \ref{algo:2}
for constrained nonsmooth convex optimization under specific conditions.

\begin{cor}\label{cor:4}
Suppose that the assumptions in Theorem \ref{thm:4} hold. 
If there exists $\beta^{(i)} > 0$ $(i\in \mathcal{I})$ 
such that $\alpha^{(i)} > \beta^{{(i)}^2}/(\beta^{{(i)}^2} + 2)$ and 
$\mathrm{d}(x_n^{(i-1)}, X) := \| x_n^{(i-1)} - P_X (x_n^{(i-1)}) \| \leq \beta^{(i)} \|x_n^{(i-1)} - Q_{\alpha}^{(i)}(x_n^{(i-1)})\|$ $(i\in \mathcal{I}, n\in\mathbb{N})$ and if $(\|x_n^{(i-1)} - Q^{(i)}(x_n^{(i-1)})\|)_{n\in \mathbb{N}}$ $(i\in \mathcal{I})$ is monotone decreasing,
then, for all $i\in \mathcal{I}$ and for all $n\in \mathbb{N}$,
\begin{align*}
&\left\| x_n^{(i-1)} - Q^{(i)}\left(x_n^{(i-1)}\right)  \right\|^2
\leq \frac{\mathrm{d}\left(x_0,X \right)^2 + 3 I N_1 \sum_{k=0}^n \lambda_k^2}{\left( 1 - \alpha^{(i)} \right) \left\{ \left(\beta^{{(i)}^2} + 2 \right) \alpha^{(i)} - \beta^{{(i)}^2} \right\} \left(n+1\right)},
\end{align*}
where $(\lambda_n)_{n\in\mathbb{N}}$ satisfies $\sum_{n=0}^\infty \lambda_n^2 < \infty$,
$N_1 := \max_{i\in \mathcal{I}} N_1^{{(i)}^2} < \infty$, and 
$N_1^{(i)}$ $(i\in \mathcal{I})$ is defined as in Assumption \ref{a8}.
Moreover, for a large enough $n\in \mathbb{N}$,
\begin{align*}
f(x_n) - f^\star 
\leq 
I \left\{ \left( \sqrt{N_1} + \frac{(I+1)N_2}{2} \right) \sqrt{\frac{N_3}{n+1}}
+ \left( \frac{(I-1) \sqrt{N_1} N_2}{2} + N_1\right) \lambda_n
\right\},
\end{align*}
where $N_2 := \max_{i\in \mathcal{I}} N_2^{(i)} < \infty$,
$N_2^{(i)}$ $(i\in \mathcal{I})$ is defined as in Assumption \ref{a8},
$N_3 := \max_{i\in \mathcal{I}} N_3^{(i)} < \infty$,
and $N_3^{(i)} := (\mathrm{d}(x_0,X)^2 + 3 I N_1 \sum_{k=0}^\infty \lambda_k^2)/
(( 1 - \alpha^{(i)}) \{ (\beta^{{(i)}^2} + 2) \alpha^{(i)} - \beta^{{(i)}^2} \})$ $(i\in \mathcal{I})$.
\end{cor}

Consider the case where $\alpha^{(i)} := 1/2$ and 
$Q^{(i)} := (1/(1-\alpha^{(i)})) (P_X - \alpha^{(i)} \mathrm{Id})$ $(i\in \mathcal{I})$; 
i.e., $Q_\alpha^{(i)} = P_X$ $(i\in \mathcal{I})$ and $Q^{(i)}$ $(i\in \mathcal{I})$ is nonexpansive \cite[Proposition 4.25]{b-c}.
Then,
$\beta^{(i)} = 1$ can be chosen such that $\alpha^{(i)} = 1/2 >  \beta^{{(i)}^2}/(\beta^{{(i)}^2} + 2) = 1/3$ and 
$\mathrm{d}(x_n^{(i-1)}, X) = \beta^{(i)} \|x_n^{(i-1)} - Q_{\alpha}^{(i)}(x_n^{(i-1)})\|$ $(i\in \mathcal{I}, n\in\mathbb{N})$.

Corollary \ref{cor:4} implies that, 
when the same step size sequence is used, the efficiency of Algorithm \ref{algo:2} 
may decrease as the number of users $I$ increases.
This can also be seen in Corollary \ref{cor:2}, indicating the rate of convergence of 
Algorithm \ref{algo:1} for constrained nonsmooth convex optimization.

\begin{proof}
Define $\mathrm{d}(x,X):= \|x- P_X(x)\|$ $(x\in H)$ and $z_n^{(i)} := P_X(x_n^{(i)})$
$(i\in \mathcal{I}, n\in \mathbb{N})$.
From \eqref{bdd}, for all $i\in \mathcal{I}$ and for all $n\in \mathbb{N}$,
\begin{align}
\left\| x_n^{(i)} - z_n^{(i-1)} \right\|^2 
&\leq \mathrm{d}\left(x_n^{(i-1)}, X \right)^2 - 2 \alpha^{(i)}  \left( 1 - \alpha^{(i)} \right)\left\| x_n^{(i-1)} - Q^{(i)} \left(x_n^{(i-1)}\right) \right\|^2\nonumber\\ 
&\quad +2 N_1 \lambda_n^2 + 2 \lambda_n \left\langle z_n^{(i-1)} - x_n^{(i-1)}, g_n^{\left(i\right)} \right\rangle,
\end{align}
which, together with $\mathrm{d}(x_n^{(i)},X) \leq \| x_n^{(i)} - z_n^{(i-1)} \|$
$(i\in \mathcal{I}, n\in \mathbb{N})$ and the definition of $x_n$ $(n\in \mathbb{N})$, implies that, for all $n\in \mathbb{N}$,
\begin{align*}
\mathrm{d}\left(x_{n+1}, X \right)^2 
&\leq \mathrm{d}\left(x_n, X \right)^2 
   - 2 \sum_{i\in \mathcal{I}} \alpha^{(i)}  \left( 1 - \alpha^{(i)} \right)\left\| x_n^{(i-1)} - Q^{(i)} \left(x_n^{(i-1)}\right) \right\|^2\nonumber\\ 
&\quad +2 I N_1 \lambda_n^2 + 2 \lambda_n \sum_{i\in \mathcal{I}} \left\langle z_n^{(i-1)} - x_n^{(i-1)}, g_n^{\left(i\right)} \right\rangle.
\end{align*}
Furthermore, the Cauchy-Schwarz inequality and $2\|x\| \|y\| \leq \|x\|^2 + \|y\|^2$ $(x,y\in H)$ ensure that,
for all $i\in \mathcal{I}$ and for all $n\in \mathbb{N}$,
$2 \langle z_n^{(i-1)} - x_n^{(i-1)}, \lambda_n g_n^{(i)} \rangle
\leq \| z_n^{(i-1)} - x_n^{(i-1)} \|^2 + \lambda_n^2 \| g_n^{(i)} \|^2$,
which, together with the definition of $N_1$ and 
$\| x_n^{(i-1)} - z_n^{(i-1)} \| \leq \beta^{(i)} \|x_n^{(i-1)} - Q_{\alpha}^{(i)}(x_n^{(i-1)})\|$ 
$(i\in \mathcal{I}, n\in\mathbb{N})$,
implies that 
\begin{align*}
2\sum_{n=0}^N \sum_{i\in \mathcal{I}} \left\langle z_n^{(i-1)} - x_n^{(i-1)}, \lambda_n g_n^{(i)} \right\rangle
&\leq \sum_{n=0}^N \sum_{i\in \mathcal{I}}
      \beta^{{(i)}^2}   
      \left\|x_n^{(i-1)} - Q_\alpha^{(i)}\left(x_n^{(i-1)}\right) \right\|^2\\
&\quad  + I N_1 \sum_{n=0}^N \lambda_n^2.    
\end{align*}
Accordingly, for all $N\in \mathbb{N}$,
\begin{align*}
&\sum_{n=0}^N \sum_{i\in \mathcal{I}}  
 \left( 1 - \alpha^{(i)} \right) \left\{ \left(\beta^{{(i)}^2} + 2 \right) \alpha^{(i)} - \beta^{{(i)}^2}  \right\} 
\left\| x_n^{(i-1)} - Q^{(i)} \left(x_n^{(i-1)} \right) \right\|^2\\
&\quad \leq \mathrm{d}\left(x_0, X \right)^2 + 3 I N_1 \sum_{n=0}^N \lambda_n^2.
\end{align*}
Since $(\|x_n^{(i-1)} - Q^{(i)}(x_n^{(i-1)})\|)_{n\in \mathbb{N}}$ $(i\in \mathcal{I})$ is monotone decreasing, 
for all $j\in \mathcal{I}$ and for all $N\in \mathbb{N}$, 
\begin{align*}
&\quad \left(N+1 \right) 
\left( 1 - \alpha^{(j)} \right) \left\{ \left(\beta^{{(j)}^2} + 2 \right) \alpha^{(j)} - \beta^{{(j)}^2}  \right\} \left\| x_N^{(j-1)} - Q^{(j)}\left(x_N^{(j-1)}\right) \right\|^2\\
&\leq \left(N+1 \right)
\sum_{i\in \mathcal{I}}  
 \left( 1 - \alpha^{(i)} \right) \left\{ \left(\beta^{{(i)}^2} + 2 \right) \alpha^{(i)} - \beta^{{(i)}^2}  \right\} \left\| x_N^{(i-1)} - Q^{(i)}\left(x_N^{(i-1)}\right) \right\|^2\\
&\leq \mathrm{d}\left( x_0,X \right)^2 + 3 I N_1 \sum_{n=0}^N \lambda_n^2,
\end{align*}
which implies that, for all $j\in \mathcal{I}$ and for all $N\in \mathbb{N}$,
\begin{align*}
&\left\| x_N^{(j-1)} - Q^{(j)} \left(x_N^{(j-1)} \right)  \right\|^2
\leq \frac{\mathrm{d}\left(x_0,X \right)^2 + 3 I N_1 \sum_{n=0}^N \lambda_n^2}{\left( 1 - \alpha^{(j)} \right) \left\{ \left(\beta^{{(j)}^2} + 2 \right) \alpha^{(j)} - \beta^{{(j)}^2} \right\} \left(N+1\right)}.
\end{align*}

Since $\| x_n^{(i-1)} - x_n^{(i)} \| \leq \| x_n^{(i-1)} - Q_\alpha^{(i)}(x_n^{(i-1)}) \| + \sqrt{N_1} \lambda_n$ $(i\in \mathcal{I}, n\in \mathbb{N})$,
for all $i\in \mathcal{I}$ and $n\in \mathbb{N}$,
\begin{align*}
\left\| x_n^{(i-1)} - x_n^{(i)} \right\| 
\leq \left( 1- \alpha^{(i)} \right) \sqrt{\frac{N_3}{n+1}} + \sqrt{N_1} \lambda_n.
\end{align*}
Moreover, since the triangle inequality implies that 
$\| x_n - x_n^{(i-1)} \| \leq \sum_{j=1}^{i-1} \| x_n^{(j-1)} - x_n^{(j)} \|$
$(i\in \mathcal{I},n\in \mathbb{N})$, for all $n\in \mathbb{N}$,
\begin{align*}
\sum_{i\in \mathcal{I}} \left\| x_n - x_n^{(i-1)} \right\|
&\leq \sum_{i\in \mathcal{I}} \sum_{j=1}^{i-1} 
     \left\{ \left( 1- \alpha^{(j)} \right) \sqrt{\frac{N_3}{n+1}} + \sqrt{N_1} \lambda_n \right\}\\
&\leq \frac{I(I-1)}{2} \left( \sqrt{\frac{N_3}{n+1}} + \sqrt{N_1} \lambda_n \right). 
\end{align*}
Accordingly, for all $n\in\mathbb{N}$,
\begin{align}
\sum_{i\in \mathcal{I}} \left\| x_n - Q_\alpha^{(i)} \left(x_n^{(i-1)}\right) \right\|
&\leq \sum_{i\in \mathcal{I}} \left\| x_n^{(i-1)} - Q_\alpha^{(i)} \left(x_n^{(i-1)}\right) \right\| 
+ \sum_{i\in \mathcal{I}} \left\| x_n - x_n^{(i-1)} \right\|\nonumber\\
&\leq \frac{I}{2} \left\{ (I+1) \sqrt{\frac{N_3}{n+1}} + (I-1)\sqrt{N_1} \lambda_n \right\}.\label{difference1}
\end{align}

In Case 1 in the proof of Theorem \ref{thm:4},
$\liminf_{n\to\infty} N_n(x^\star) \leq 0$, where $\{ x^\star \}= X^\star$. 
A discussion similar to the one for proving $M_n(x^\star) \leq 0$ $(n \geq k_1)$
(see proof of Corollary \ref{cor:1}) guarantees that 
there exists $k_3 \in \mathbb{N}$ such that, for all $n \geq k_3$, 
$N_n (x^\star) = f(x_n) - f^\star 
   - \sqrt{N_1} \sum_{i\in \mathcal{I}} \|x_n^{(i-1)} - Q_\alpha^{(i)}(x_n^{(i-1)})\| - N_2 \sum_{i\in \mathcal{I}} \| Q_\alpha^{(i)}(x_n^{(i-1)}) -x_n  \| - IN_1 \lambda_n \leq 0$.
From \eqref{ftau} in Case 2 in the proof of Theorem \ref{thm:4}, 
there exists $k_4 \in \mathbb{N}$ such that, for all $n \geq k_4$,
$f(x_{\tau(n)}) - f^\star < 
\sqrt{N_1} \sum_{i\in \mathcal{I}} \|x_{\tau(n)}^{(i-1)} - Q_\alpha^{(i)}(x_{\tau(n)}^{(i-1)})\| + N_2 \sum_{i\in \mathcal{I}} \| Q_\alpha^{(i)}(x_{\tau(n)}^{(i-1)}) 
 -x_{\tau(n)}  \|+ IN_1 \lambda_n$.
Therefore, from $\|x_n^{(i-1)} - Q_\alpha^{(i)} (x_n^{(i-1)})\|^2 \leq N_3/(n+1)$
$(n\in \mathbb{N})$
and \eqref{difference1}, for a large enough $n\in \mathbb{N}$,
\begin{align*}
&\quad f(x_n) - f^\star\\
&\leq 
\sqrt{N_1} \sum_{i\in \mathcal{I}} \left\|x_n^{(i-1)} - Q_\alpha^{(i)}\left(x_n^{(i-1)} \right) \right\| 
+ N_2 \sum_{i\in \mathcal{I}} \left\| Q_\alpha^{(i)}\left(x_n^{(i-1)}\right) -x_n \right\| + IN_1 \lambda_n\\
&\leq \sqrt{N_1} \sum_{i\in \mathcal{I}} \left( 1- \alpha^{(i)}\right) \sqrt{\frac{N_3}{n+1}} + IN_1 \lambda_n\\
&\quad + N_2 \frac{I}{2} \left\{ (I+1) \sqrt{\frac{N_3}{n+1}} + (I-1)\sqrt{N_1} \lambda_n \right\}\\
&\leq 
I \left\{ \left( \sqrt{N_1} + \frac{(I+1)N_2}{2} \right) \sqrt{\frac{N_3}{n+1}}
+ \left( \frac{(I-1) \sqrt{N_1} N_2}{2} + N_1\right) \lambda_n
\right\}.
\end{align*}
This completes the proof.
\end{proof}

\section{Numerical Examples}\label{sec:4}
This section considers the following problem over the intersection of sublevel sets of convex functions \cite[Section 3.2]{neto2009} 
and numerically compares Algorithms \ref{algo:1} and \ref{algo:2} with the method in \cite[(2.1), (3.1), (3.14), (4.3)]{neto2009}.
\begin{prob}\label{prob:2}
Let $f^{(i)} \colon \mathbb{R}^N \to \mathbb{R}$ and 
$g^{(i)} \colon \mathbb{R}^N \to \mathbb{R}$ $(i\in{\mathcal{I}})$ be convex.
\begin{align*}
\text{Minimize } f (x) := \sum_{i\in{\mathcal{I}}} f^{\left(i\right)} (x)  
\text{ subject to } x \in X := \bigcap_{i\in{\mathcal{I}}} \mathrm{lev}_{\leq 0} g^{\left(i \right)} \neq \emptyset,
\end{align*}
where $\mathrm{lev}_{\leq 0} g^{(i)} := \{ x\in \mathbb{R}^N \colon g^{(i)} (x) \leq 0 \}$.
\end{prob}

Let us define the {\em subgradient projection} \cite[Proposition 2.3]{b-c2001}, \cite[Subchapter 4.3]{vasin}
relative to $g^{(i)}$ $(i\in{\mathcal{I}})$ for all $x\in \mathbb{R}^N$ by 
\begin{align*}
Q_{\mathrm{sp}}^{\left(i\right)} (x) := 
\begin{cases}
\displaystyle{x - \frac{g^{\left(i\right)}(x)}{\left\| z^{(i)}(x) \right\|^2}z^{(i)}(x)} &\text{ if } g^{\left(i\right)}(x) > 0,\\
x &\text{ otherwise},
\end{cases}
\end{align*}
where $z^{(i)}(x) \in \partial g^{(i)}(x)$ $(i\in {\mathcal{I}}, x\in \mathbb{R}^N)$.
The mapping $Q_{\mathrm{sp}}^{\left(i\right)}$ $(i\in {\mathcal{I}})$
is quasi-firmly nonexpansive, and $\mathrm{Id} - Q_{\mathrm{sp}}^{(i)}$ $(i\in{\mathcal{I}})$ is demiclosed in the sense of the Euclidean space setting \cite[Lemma 3.1]{b-c2014}.  
Moreover, $\mathrm{Fix}(Q_{\mathrm{sp}}^{(i)})  
= \mathrm{lev}_{\leq 0} g^{(i)}$. 
Hence, Problem \ref{prob:2} is an example of Problem \ref{prob:1}
that can be solved by Algorithms \ref{algo:1} 
and \ref{algo:2} (see Theorems \ref{thm:1}, \ref{thm:2}, \ref{thm:3}, and \ref{thm:4}).

Here it is assumed that $\mathrm{lev}_{\leq 0} g^{(p)}$ is bounded for some $p\in {\mathcal{I}}$ 
(see also \cite[Proposition 3.4]{neto2009}).
Accordingly, a closed ball $Y$ with a large enough radius can be chosen so that $Y \supset \mathrm{lev}_{\leq 0} g^{(p)} \supset X$.
Hence, setting $X^{(i)} := Y$ $(i\in\mathcal{I})$ in \eqref{equation:0} satisfies 
Assumptions \ref{assum:0} and \ref{a8_1}.

The following is the incremental subgradient method (ISM) \cite[(2.1), (3.1), (3.14), (4.3)]{neto2009} 
used for solving Problem \ref{prob:2} given $x_0 \in \mathbb{R}^N$ and $(\lambda_n)_{n\in\mathbb{N}}$ $(\subset (0,\infty))$:
\begin{align}\label{neto}
\begin{cases}
x_n^{\left(0\right)} := x_n,\\
x_n^{\left(i\right)} := P_Y \left( x_n^{\left(i-1\right)} - \lambda_n g_n^{\left(i\right)} \right), \text{ } 
g_n^{\left(i\right)} \in \partial f^{\left(i \right)} \left( x_n^{\left(i-1\right)} \right) \text{ } 
\left( i \in \mathcal{I} \right),\\
y_n^{\left(0\right)} := x_n^{\left(I\right)},\\
y_n^{\left(i\right)} := Q_{\mathrm{sp}}^{\left(i\right)} \left(y_n^{\left(i-1\right)}  \right) \text{ }
\left( i \in \mathcal{I} \right),\\
x_{n+1} := y_n^{\left(I\right)}.
\end{cases}
\end{align}
Theorem 2.5 in \cite{neto2009} guarantees that, if $(\| x_n - P_X (x_n) \|)_{n\in\mathbb{N}}$ is bounded and if 
$\lim_{n\to\infty} \max \{0, f(P_X (x_n)) - f(x_n) \} = 0$,
$(x_n)_{n\in\mathbb{N}}$ generated by \eqref{neto} with (C2) and (C3) satisfies 
$\lim_{n\to\infty} \| x_n - P_X (x_n) \|=0$ and $\lim_{n\to\infty} f(x_n) = f^\star$.

In an experiment, we define that, for all $i\in \mathcal{I}$, 
$f^{(i)}(x) := | a^{(i)} x + b^{(i)}|$ $(x\in \mathbb{R})$ and 
$g^{(i)}(x) := \langle c^{(i)}, x \rangle + d^{(i)}$ 
$(\langle c^{(i)}, x \rangle > - d^{(i)})$ or $0$ $(\langle c^{(i)}, x \rangle \leq - d^{(i)})$, 
where $a^{(i)} > 0$, $b^{(i)}, d^{(i)} \in \mathbb{R}$, and $c^{(i)} \in \{ x:= (x_1,x_2,\ldots,x_I) \in\mathbb{R}^I \colon x_i > 0 \text{ } (i\in\mathcal{I}) \}$.
We modified $g^{(1)}(x) := \| x \| - 2 C$, where $C > 0$, to satisfy $\mathrm{lev}_{\leq 0}g^{(1)} \subset 
Y := \{ x\in \mathbb{R}^I \colon \|x\| \leq 2C \}$. 
The experiment was one using a 27-inch iMac with a 3.20 GHz Intel(R) Core(TM) i5-4570 CPU processor, 
24 GB, 1600 MHz DDR3 memory, and Mac OSX Yosemite (Version 10.10.3) operating system. 
ISM (Algorithm \eqref{neto}), Algorithm \ref{algo:1}, and Algorithm \ref{algo:2} were written in Python 3.4.3, 
and gnuplot 5.0 (patchlevel 0) was used to graph the results.
We set $I := 2, 8, 16, 64, 256$ and $\alpha^{(i)} := 1/2$ $(i\in\mathcal{I})$ and used 
$a^{(i)} \in (0,100], b^{(i)} \in [-100,100], c^{(i)}$ with $\| c^{(i)}\| = 1$, $d^{(i)} \in [-\sqrt[I]{C}, \sqrt[I]{C}]$, $\bar{a}^{(i)} \in \partial f^{(i)}(-b^{(i)}/a^{(i)})$, and 
$\bar{c}^{(i)} \in \partial g^{(i)}(x)$ $(\langle c^{(i)}, x \rangle = - d^{(i)})$ generated randomly by 
\texttt{numpy.random}\footnote{\url{http://docs.scipy.org/doc/numpy/reference/routines.random.html}} (a Mersenne Twister pseudo-random number generator).

To see how the choice of step size affects the convergence rate of the algorithms, we used
\begin{align}\label{step}
\begin{split}
&\text{Constant step sizes: } \lambda_n := 10^{-3}, 10^{-5} \text{ } \left( n\in\mathbb{N} \right),\\
&\text{Diminishing step sizes: } \lambda_n := \frac{10^{-3}}{(n+1)^{a}}
\text{ } \left(a := 1, 0.1, 0.01, n\in\mathbb{N} \right).
\end{split}
\end{align}
From Theorems \ref{thm:1} and \ref{thm:3}, it can be expected that Algorithms \ref{algo:1} and \ref{algo:2} with small enough constant step sizes approximate solutions to Problem \ref{prob:2}. Numerical results in \cite{iiduka_siopt2013,iiduka} indicate that the existing fixed point optimization algorithms with small step sizes (e.g., $\lambda_n := 10^{-2}/(n+1)^a, 10^{-3}/(n+1)^a, 10^{-5}/(n+1)^a$ $(a := 0.1, 0.01, n\in\mathbb{N})$) have fast convergence. Accordingly, the experiment described in this section used the step sizes in \eqref{step}. We also found that, under the same conditions as in the above paragraph, ISM, Algorithm \ref{algo:1}, and Algorithm \ref{algo:2} when $\lambda_n := 10^{-3}/(n+1)^{a}$ and $\lambda_n := 10^{-5}/(n+1)^{a}$ $(a := 0.1, 0.01, n\in\mathbb{N})$ perform better than when $\lambda_n := 1/(n+1)^{a}$ $(a = 0.1, 0.01, n\in\mathbb{N})$. Only the results for the step sizes in \eqref{step} are given due to lack of space.\
The step size $\lambda_n := 10^{-3}/(n+1)$ $(n\in \mathbb{N})$
satisfying $\sum_{n=0}^\infty \lambda_n^2 < \infty$
was used to illustrate the proposed methods' efficiency and support the convergence rate analysis of the methods (Corollaries \ref{cor:1}, \ref{cor:2}, \ref{cor:3}, and \ref{cor:4}).

We used two performance measures for each $n\in \mathbb{N}$:
\begin{align*}
D_n := \frac{1}{100} \sum_{s=1}^{100} \sum_{i\in\mathcal{I}} \left\| x_n\left(s\right) - Q_{\mathrm{sp}}^{\left(i\right)} \left(x_n\left(s\right)\right) \right\|,
\text{ } 
F_n := \frac{1}{100} \sum_{s=1}^{100} \sum_{i\in\mathcal{I}} f^{\left(i\right)} \left(x_n^{\left(i\right)}\left(s\right)\right), 
\end{align*}
where $(x_n(s))_{n\in\mathbb{N}}$ defined by $x_n (s):= (x_n^{(i)}(s))$ $(n\in\mathbb{N},s=1,2,\ldots,100)$
is the sequence generated by the initial point $x(s)$ 
$(s = 1,2,\ldots,100)$ and each of ISM, Algorithm \ref{algo:1}, and Algorithm \ref{algo:2}.
If $(D_n)_{n\in\mathbb{N}}$ converges to $0$, they converge to some point in 
$\bigcap_{i\in\mathcal{I}} \mathrm{Fix}(Q_{\mathrm{sp}}^{(i)}) = X$.

First, let us consider the case where $I := 64$ and $\lambda_n := 10^{-3}$ $(n\in\mathbb{N})$.
Figures \ref{fig:1} and \ref{fig:3} illustrate the results for ISM, Algorithm \ref{algo:1}, and Algorithm \ref{algo:2}. The y-axes in Figure \ref{fig:1} represent the value of $D_n$ while the y-axes in Figure \ref{fig:3} represent the value of $F_n$. The x-axes in Figures \ref{fig:1}(a) and \ref{fig:3}(a) represent the number of iterations while the x-axes in Figures \ref{fig:1}(b) and \ref{fig:3}(b) represent elapsed time. Figure \ref{fig:1} shows that $(D_n)_{n\in \mathbb{N}}$ generated by Algorithm \ref{algo:1} was stable and monotone decreasing while those generated by ISM and Algorithm \ref{algo:2} were unstable and approximately zero during the early iterations. Figure \ref{fig:3} shows that ISM, Algorithm \ref{algo:1}, and Algorithm \ref{algo:2} minimized $F_n$.

Figure \ref{fig:2} and \ref{fig:4} illustrate the results when $I := 64$ and $\lambda_n := 10^{-5}$ $(n\in\mathbb{N})$. Figures \ref{fig:1} and \ref{fig:2} show that Algorithm \ref{algo:1} when $\lambda_n := 10^{-5}$ ($D_{1000} \approx 10^{-6}$) performed slightly better than when $\lambda_n := 10^{-3}$ ($D_{1000} \approx 10^{-4}$). 
In particular, the figures indicate that $(D_n)_{n\in\mathbb{N}}$ for Algorithm \ref{algo:2} when $\lambda_n := 10^{-5}$ was more stable than when $\lambda_n := 10^{-3}$ and that the behavior of ISM when $\lambda_n = 10^{-5}$ was unstable and almost the same as when $\lambda_n := 10^{-3}$. Figure \ref{fig:4} shows that $(F_n)_{n\in\mathbb{N}}$ for ISM and Algorithm \ref{algo:2} decreased during the early iterations compared with $(F_n)_{n\in\mathbb{N}}$ for Algorithm \ref{algo:1}.

Next, let us consider the case where $I := 64$ and $\lambda_n := 10^{-3}/(n+1)^{0.1}$ $(n\in\mathbb{N})$. Figure \ref{fig:5} shows that $(D_n)_{n\in \mathbb{N}}$ generated by Algorithm \ref{algo:1} was stable while those generated by ISM and Algorithm \ref{algo:2} were unstable and approximately zero during the early iterations, as in the case with $\lambda_n := 10^{-3}$ (Figure \ref{fig:1}). Figure \ref{fig:7} shows that $F_n$ decreased faster with ISM and Algorithm \ref{algo:2} than with Algorithm \ref{algo:1}. Figures \ref{fig:6} and \ref{fig:8} illustrate the behaviors of $D_n$ and $F_n$ when $I:= 64$ and $\lambda_n := 10^{-3}/(n+1)^{0.01}$ $(n\in\mathbb{N})$ and show that the behaviors were almost the same as the ones when $\lambda_n := 10^{-3}/(n+1)^{0.1}$ $(n\in\mathbb{N})$ (Figures \ref{fig:5} and \ref{fig:7}).


Let us fix the step size $\lambda_n := 10^{-3}/(n+1)^{0.01}$ $(n\in\mathbb{N})$ and see how the number of users affects the efficiency of Algorithms \ref{algo:1} and \ref{algo:2}. The behaviors of $D_n$ and $F_n$ for Algorithm \ref{algo:1} when $I := 16, 64, 256$ are illustrated in Figure \ref{fig:9}. Although $(D_n)_{n\in \mathbb{N}}$ and $(F_n)_{n\in\mathbb{N}}$ were stable, the larger the $I$, the greater the number of iterations that were required (Figure \ref{fig:9}(a), (c)) and the longer the elapsed time (Figure \ref{fig:9}(b), (d)). That is, the efficiency of Algorithm \ref{algo:1} decreases as the number of users increases.
The behaviors of $D_n$ and $F_n$ for Algorithm \ref{algo:2} when $I := 16, 64, 256$ are illustrated in Figure \ref{fig:10}. Although $(D_n)_{n\geq 40}$ were unstable, $D_{10} \approx 10^{-5}$ held for the three cases (Figure \ref{fig:10}(a), (b)), and $(F_n)_{n\in\mathbb{N}}$ for the three cases converged in the early stages (Figure \ref{fig:10}(c), (d)). 

Finally, let us consider the case when $\lambda_n := 10^{-3}/(n+1)$ $(n\in \mathbb{N})$ and $f^{(1)}$ replaced by $f^{(1)}(x) := a^{(1)} \| x + b^{(1)} \|^2$ $(x\in \mathbb{R}^I)$, where $a^{(1)} \in (0,100]$ and $b^{(1)} \in [-100,100]^I$ were chosen randomly, to support the convergence analysis of Algorithms \ref{algo:1} 
and \ref{algo:2} discussed in Subsections \ref{subsec:3.2}, \ref{subsec:3.3}, \ref{subsec:4.2}, and \ref{subsec:4.3}
(see also assumption (i) in Theorems \ref{thm:2} and \ref{thm:4} and condition $\sum_{n=0}^\infty \lambda_n^2 < \infty$ in Corollaries \ref{cor:2} and \ref{cor:4}).
Since $f^{(1)}$ is strongly convex, Theorems \ref{thm:2} and \ref{thm:4} guarantee 
that Algorithms \ref{algo:1} and \ref{algo:2} converge to the solution to Problem \ref{prob:2}.   
Moreover, Corollaries \ref{cor:2} and \ref{cor:4} indicate that, under certain assumptions, Algorithm \ref{algo:1} satisfies inequality
\begin{align}\label{rate1}
\begin{split}
&\sum_{i\in \mathcal{I}} \left\| x_n - Q_{\mathrm{sp}}^{(i)} (x_n)  \right\| \leq 
\frac{I \sqrt{I M_3}}{\sqrt{n+1}},\\
&f(x_n) - f^\star \leq  \frac{I \left\{ \left(\sqrt{M_1} + M_2 \right)\sqrt{I M_3} + M_1  \right\}}{\sqrt{n+1}},
\end{split}
\end{align}
while Algorithm \ref{algo:2} satisfies inequality
\begin{align}\label{rate2}
\begin{split}
&\sum_{i\in \mathcal{I}} \left\| x_n^{(i-1)} - Q_{\mathrm{sp}}^{(i)} \left(x_n^{(i-1)} \right)  \right\| \leq 
\frac{I\sqrt{N_3}}{\sqrt{n+1}},\\
&f(x_n) - f^\star \leq 
\frac{I \left\{ \left(2 \sqrt{N_1} + (I+1)N_2 \right) \sqrt{N_3} + 
      \left((I-1) \sqrt{N_1} N_2 + 2 N_1 \right)\right\}}{2 \sqrt{n+1}}.
\end{split}      
\end{align}
Inequalities \eqref{rate1} and \eqref{rate2} imply that the efficiencies of Algorithms \ref{algo:1} and \ref{algo:2} may decrease as the number of users $I$ increases. Figure \ref{fig:11} shows that $(D_n)_{n\in \mathbb{N}}$ generated by Algorithm \ref{algo:1} was monotone decreasing and that, the larger the $I$, the greater the number of iterations that were required (Figure \ref{fig:11}(a), (c)) and the longer the elapsed time (Figure \ref{fig:11}(b), (d)), as seen in Figure \ref{fig:9}. This can be seen in \eqref{rate1}. Figure \ref{fig:12} illustrates the behaviors of $D_n$ and $F_n$ for Algorithm \ref{algo:2}. 
It shows that the behaviors of Algorithm \ref{algo:2} when one $f^{(i)}$ was strongly convex were more stable than when all $f^{(i)}$ were convex (Figures \ref{fig:5}--\ref{fig:8} and \ref{fig:10}). The strong convexity condition of $f^{(1)}$ (i.e., the uniqueness of the solution to Problem \ref{prob:2}) apparently affects the stability of Algorithm \ref{algo:2}. This is consistent with Theorem \ref{thm:4} and indicates that the whole sequence $(x_n)_{n\in \mathbb{N}}$ in Algorithm \ref{algo:2} converges when one $f^{(i)}$ is strongly convex while a subsequence of $(x_n)_{n\in \mathbb{N}}$ converges when all $f^{(i)}$ are convex. Although \eqref{rate2} and Figure \ref{fig:12} show that the efficiency of Algorithm \ref{algo:2} decreases as $I$ increases, Algorithm \ref{algo:2} has fast convergence regardless of the number of users. Furthermore, as shown by Figures \ref{fig:11} and \ref{fig:12}, when $I:=2$, Algorithm \ref{algo:1} performed better than Algorithm \ref{algo:2} in the early stages. This means that Algorithm \ref{algo:1} is well suited for use when the number of users is small.

From the above discussion, we conclude that Algorithm \ref{algo:1} is robust in terms of stability regardless of the number of users and is well suited for small-scale convex optimization problems over fixed point sets of quasi-nonexpansive mappings. We also conclude that Algorithm \ref{algo:2} has fast convergence regardless of the number of users and is well suited for solving large-scale convex optimization problems over fixed point sets of quasi-nonexpansive mappings.

\section{Conclusion and Future Work}\label{sec:5}
This paper described parallel and incremental subgradient methods for minimizing the sum of nondifferentiable, convex functions over the intersection of fixed point sets of quasi-nonexpansive mappings in a real Hilbert space. Investigation of the convergence properties for a constant step-size rule and a diminishing step-size rule showed that, with a small constant step size, the two methods give an approximate solution to the minimization problem and that, with a diminishing sequence, the sequence generated by each of the two methods strongly converges to the solution to the minimization problem under certain assumptions.
The convergence rate of the two methods was analyzed under certain situations.

This paper also numerically compared the proposed methods with an existing method for nonsmooth convex optimization over sublevel sets of convex functions. Numerical examples demonstrated that, for concrete convex optimization problems when the number of users is fixed, the parallel subgradient method with a constant or diminishing step size is more stable than the incremental subgradient method with the same step size while the incremental subgradient method has faster convergence. The numerical examples also demonstrated that the efficiency of the parallel subgradient method decreased as the number of users increased while the incremental subgradient method was robust even with a large number of users.

The proposed methods work well only when each user makes the best use of its own private information while the distributed random projection method \cite{lee2013} works well even when each user randomly sets one projection selected from many projections. This means that consideration should be given to developing distributed random fixed point algorithms that work when one user randomly chooses one quasi-nonexpansive mapping at a time. Consideration should also be given to devising nonsmooth convex optimization algorithms that combine stability and fast convergence, in contrast to the proposed methods. For example, an algorithm combining the parallel and incremental subgradient methods could be devised on the basis of the ideas in \cite{iiduka_hishinuma_siopt2014}. Such an algorithm should be numerically evaluated to see whether it performs better than the two proposed methods.\\

\textbf{Acknowledgements} 
I am sincerely grateful to the editor, Alexander Shapiro, the anonymous associate editor, and the two anonymous reviewers for helping me improve the original manuscript. I also thank Kazuhiro Hishinuma for his input on the numerical examples.


\newpage


\begin{figure}[H]
\subfigure[$D_n$ vs. no. of iterations]{\includegraphics[width=0.8\textwidth]{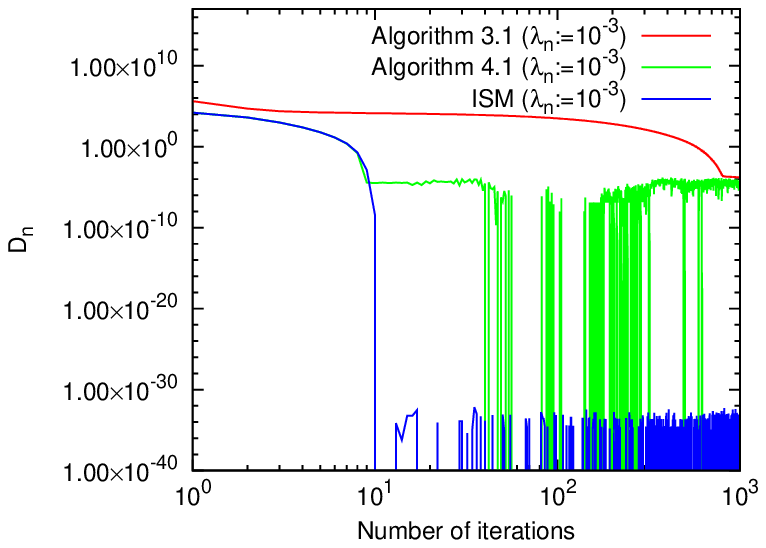}}
\subfigure[$D_n$ vs. elapsed time]{\includegraphics[width=0.8\textwidth]{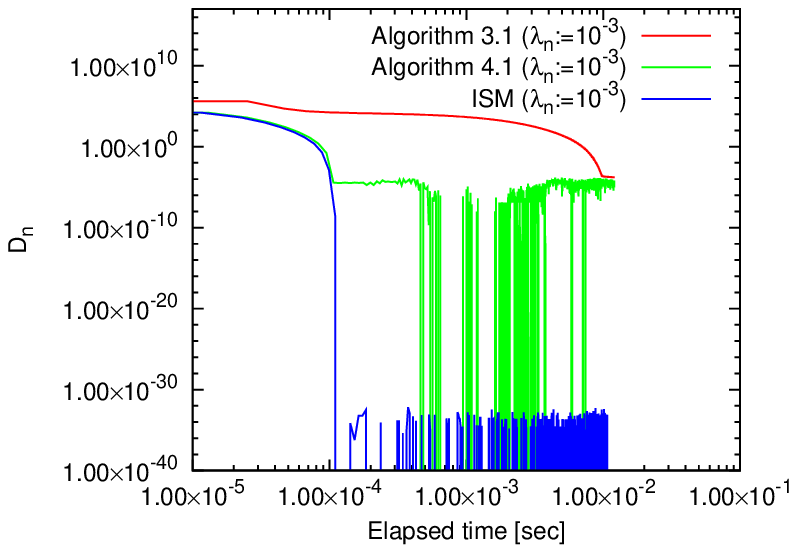}}
\caption{Behavior of $D_n$ for ISM, Algorithm \ref{algo:1}, and Algorithm \ref{algo:2} when $\lambda_n := 10^{-3}$ and $I:= 64$}\label{fig:1}
\end{figure}

\begin{figure}[H]
\subfigure[$F_n$ vs. no. of iterations]{\includegraphics[width=0.8\textwidth]{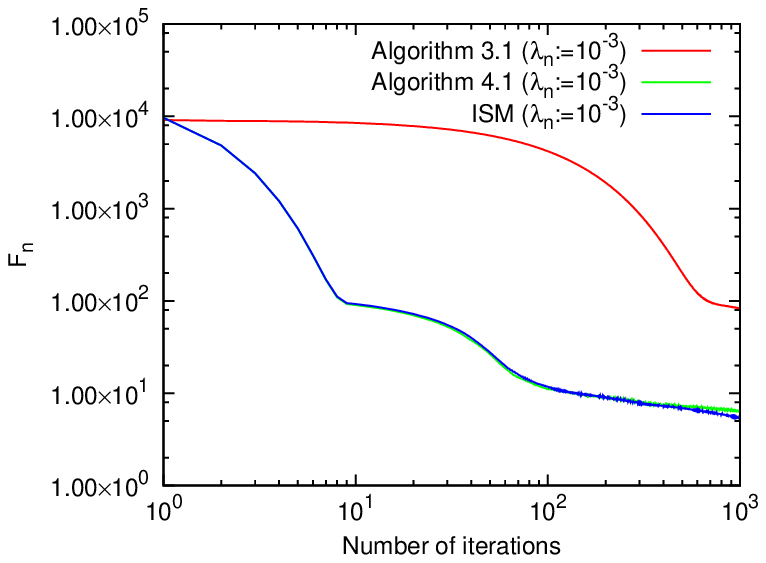}}
\subfigure[$F_n$ vs. elapsed time]{\includegraphics[width=0.8\textwidth]{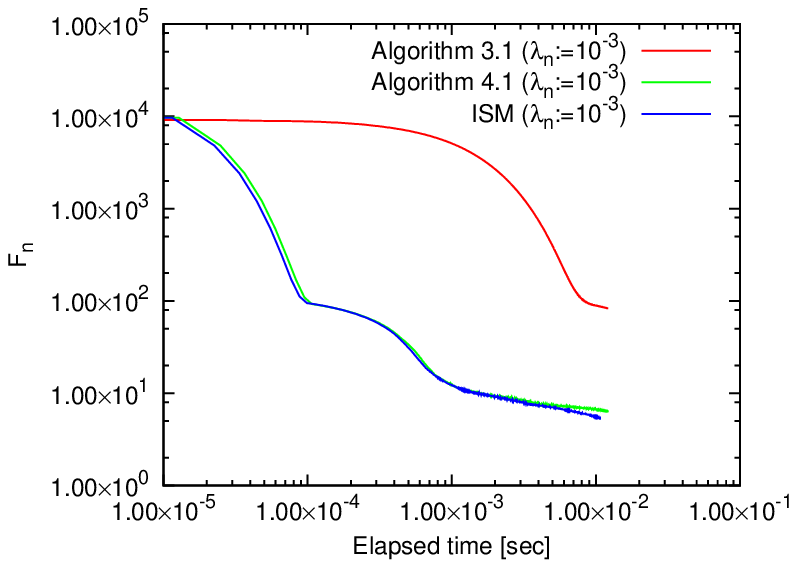}}
\caption{Behavior of $F_n$ for ISM, Algorithm \ref{algo:1}, and Algorithm \ref{algo:2} when $\lambda_n := 10^{-3}$ and $I:= 64$}\label{fig:3}
\end{figure}

\begin{figure}[H]
\subfigure[$D_n$ vs. no. of iterations]{\includegraphics[width=0.8\textwidth]{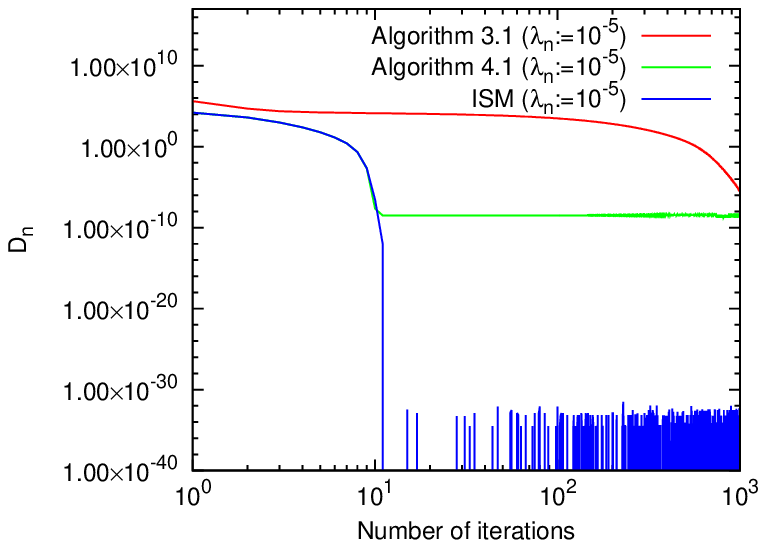}}
\subfigure[$D_n$ vs. elapsed time]{\includegraphics[width=0.8\textwidth]{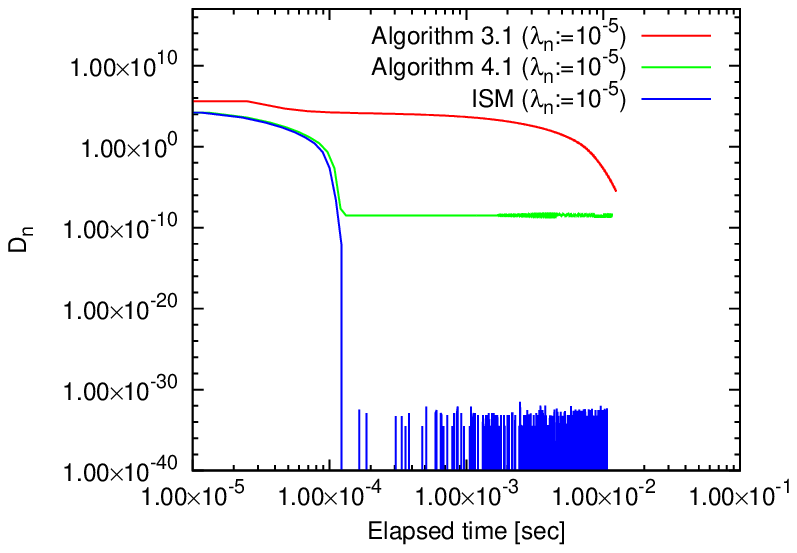}}
\caption{Behavior of $D_n$ for ISM, Algorithm \ref{algo:1}, and Algorithm \ref{algo:2} when $\lambda_n := 10^{-5}$ and $I:= 64$}\label{fig:2}
\end{figure}

\begin{figure}[H]
\subfigure[$F_n$ vs. no. of iterations]{\includegraphics[width=0.8\textwidth]{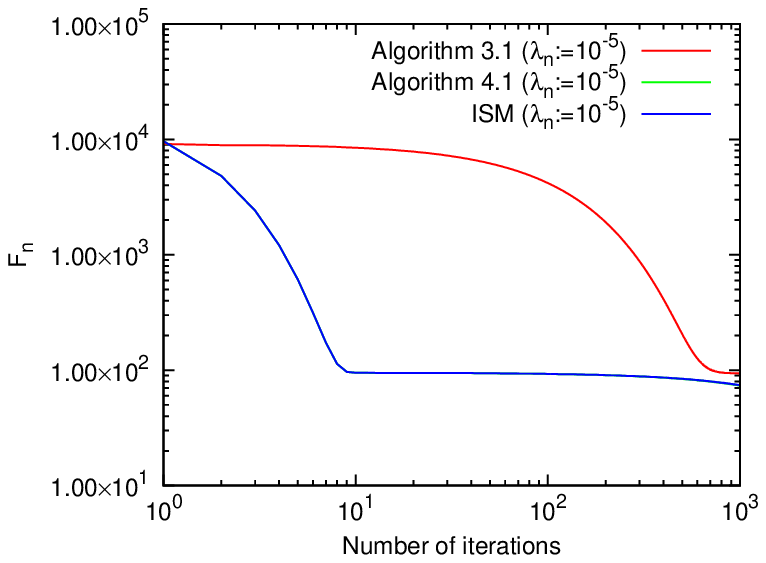}}
\subfigure[$F_n$ vs. elapsed time]{\includegraphics[width=0.8\textwidth]{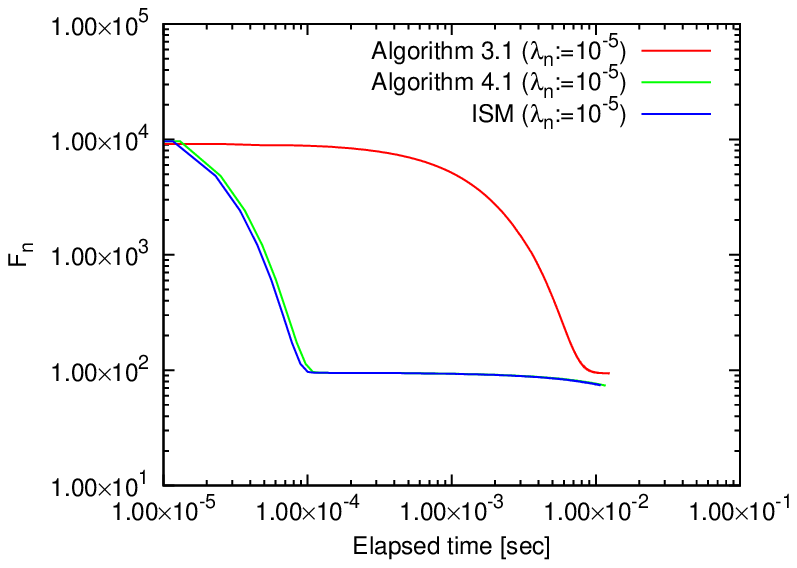}}
\caption{Behavior of $F_n$ for ISM, Algorithm \ref{algo:1}, and Algorithm \ref{algo:2} when $\lambda_n := 10^{-5}$ and $I:= 64$}\label{fig:4}
\end{figure}

\begin{figure}[H]
\subfigure[$D_n$ vs. no. of iterations]{\includegraphics[width=0.8\textwidth]{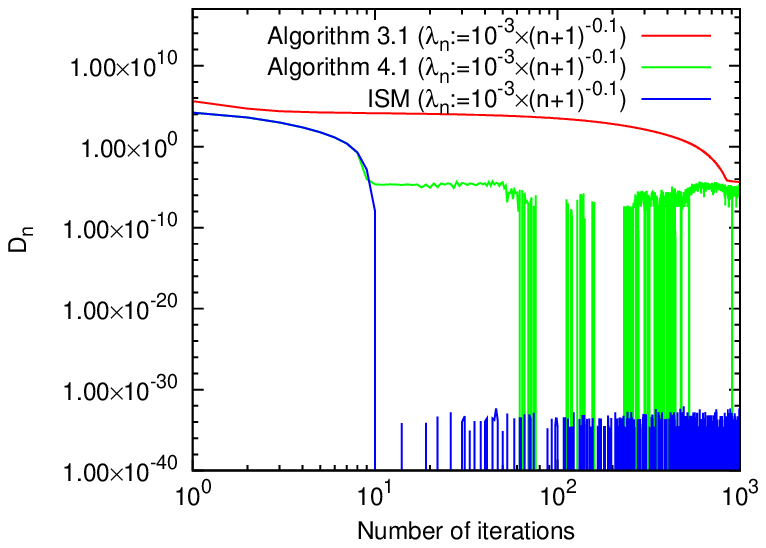}}
\subfigure[$D_n$ vs. elapsed time]{\includegraphics[width=0.8\textwidth]{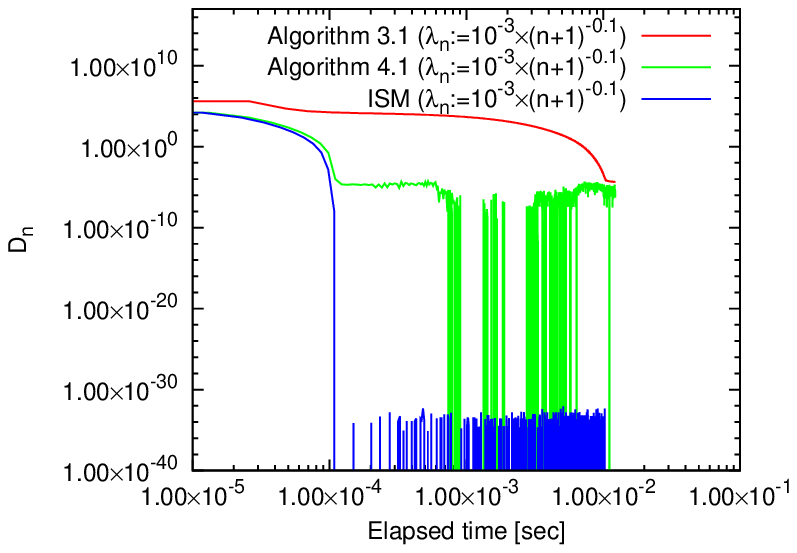}}
\caption{Behavior of $D_n$ for ISM, Algorithm \ref{algo:1}, and Algorithm \ref{algo:2} when $\lambda_n := 10^{-3}/(n+1)^{0.1}$ and $I:= 64$}\label{fig:5}
\end{figure}

\begin{figure}[H]
\subfigure[$F_n$ vs. no. of iterations]{\includegraphics[width=0.8\textwidth]{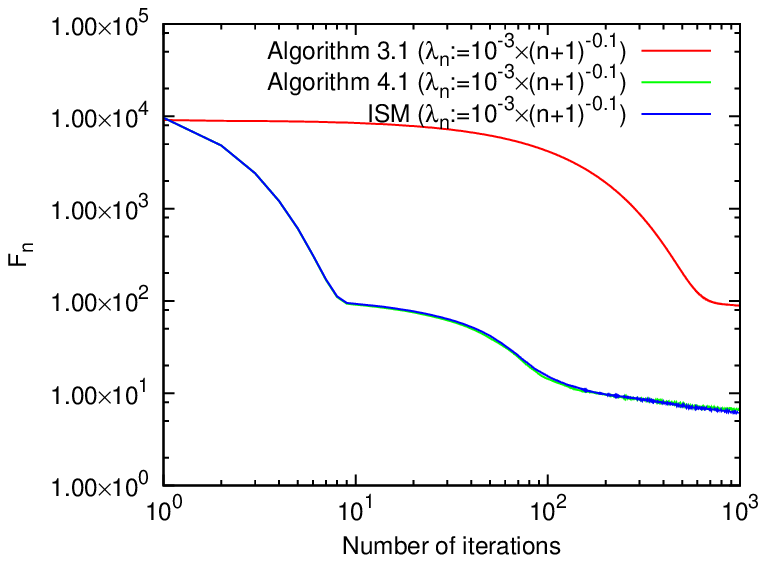}}
\subfigure[$F_n$ vs. elapsed time]{\includegraphics[width=0.8\textwidth]{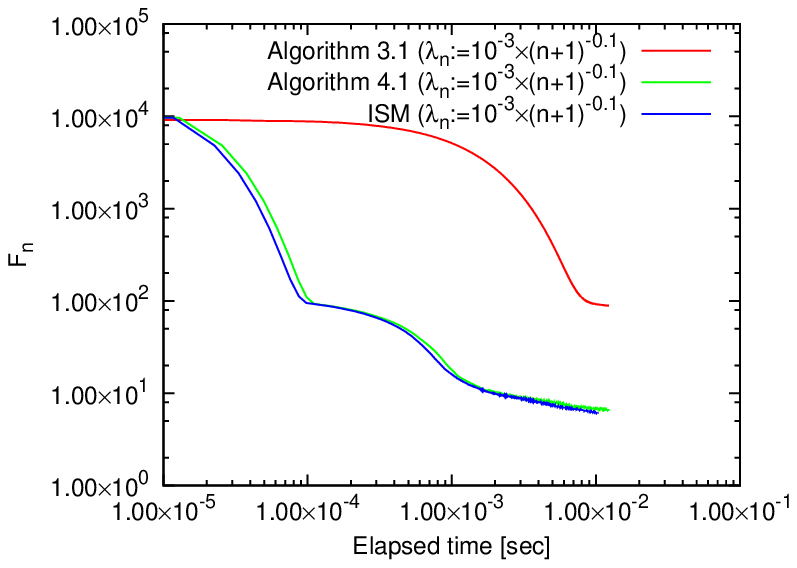}}
\caption{Behavior of $F_n$ for ISM, Algorithm \ref{algo:1}, and Algorithm \ref{algo:2} when $\lambda_n := 10^{-3}/(n+1)^{0.1}$ and $I:= 64$}\label{fig:7}
\end{figure}

\begin{figure}[H]
\subfigure[$D_n$ vs. no. of iterations]{\includegraphics[width=0.8\textwidth]{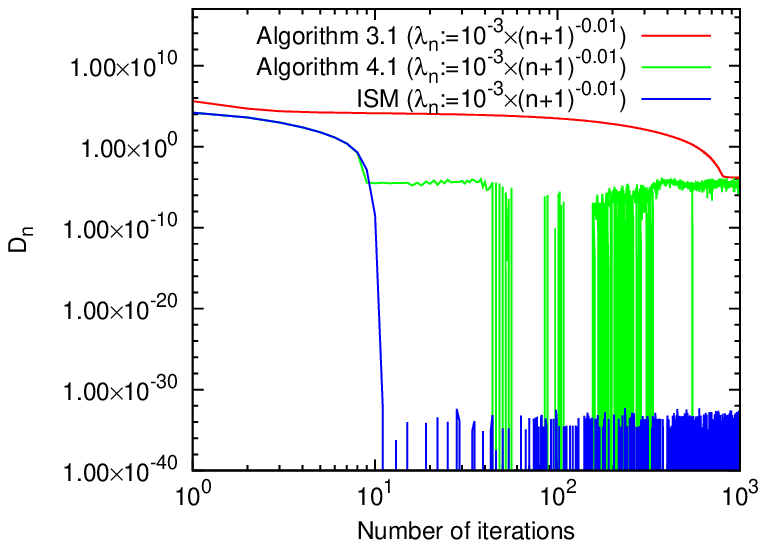}}
\subfigure[$D_n$ vs. elapsed time]{\includegraphics[width=0.8\textwidth]{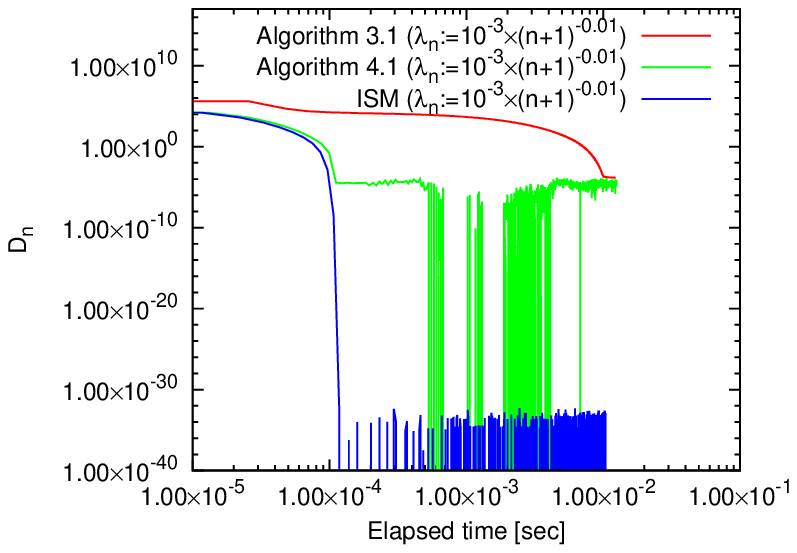}}
\caption{Behavior of $D_n$ for ISM, Algorithm \ref{algo:1}, and Algorithm \ref{algo:2} when $\lambda_n := 10^{-3}/(n+1)^{0.01}$ and $I:= 64$}\label{fig:6}
\end{figure}

\begin{figure}[H]
\subfigure[$F_n$ vs. no. of iterations]{\includegraphics[width=0.8\textwidth]{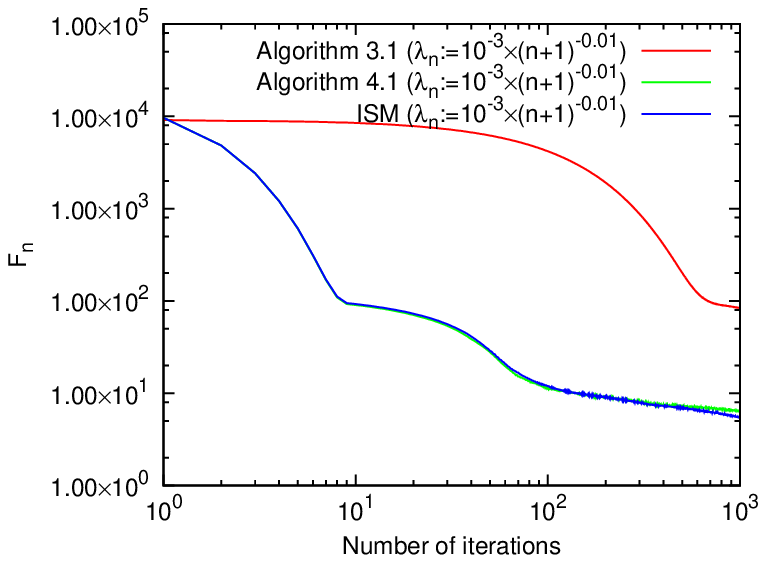}}
\subfigure[$F_n$ vs. elapsed time]{\includegraphics[width=0.8\textwidth]{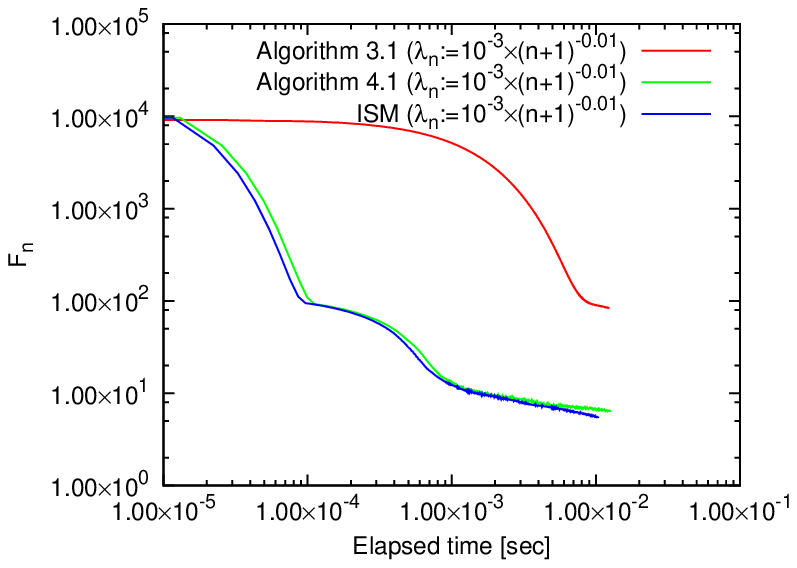}}
\caption{Behavior of $F_n$ for ISM, Algorithm \ref{algo:1}, and Algorithm \ref{algo:2} when $\lambda_n := 10^{-3}/(n+1)^{0.01}$ and $I:= 64$}\label{fig:8}
\end{figure}


\begin{figure}[H]
\subfigure[$D_n$ vs. no. of iterations]{\includegraphics[width=0.5\textwidth]{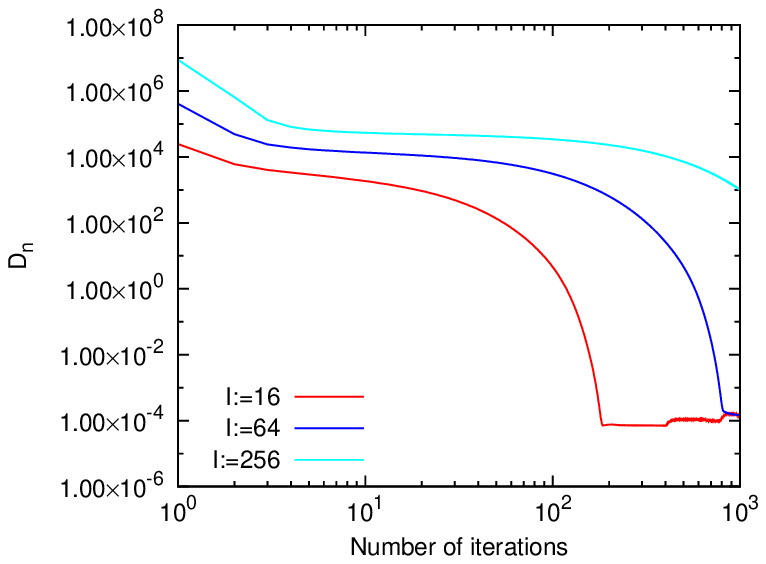}}
\subfigure[$D_n$ vs. elapsed time]{\includegraphics[width=0.5\textwidth]{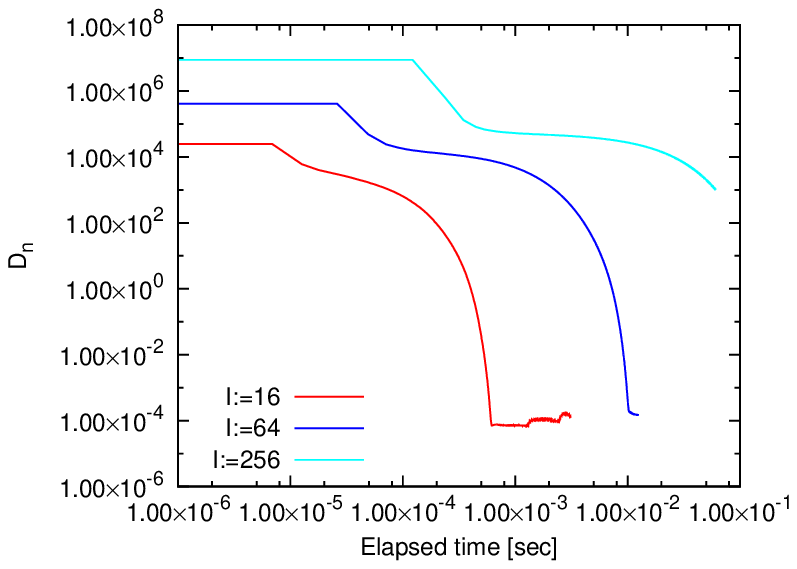}}
\\{}
\subfigure[$F_n$ vs. no. of iterations]{\includegraphics[width=0.5\textwidth]{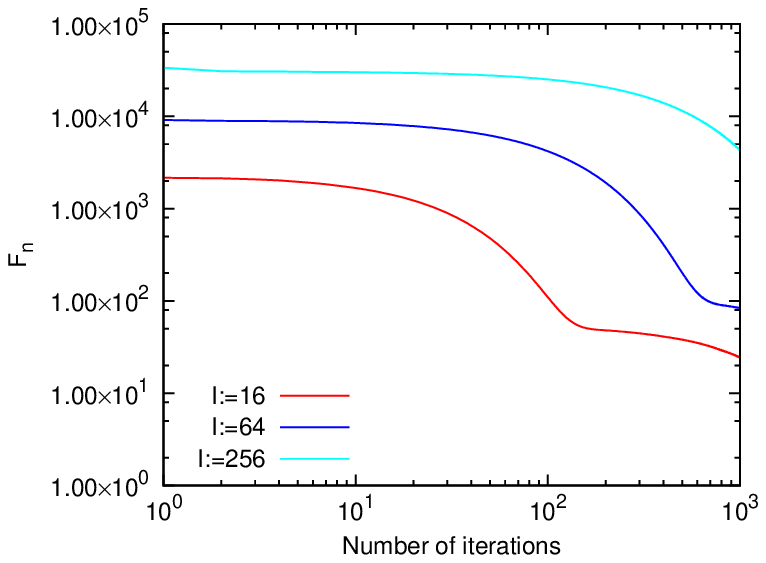}}
\subfigure[$F_n$ vs. elapsed time]{\includegraphics[width=0.5\textwidth]{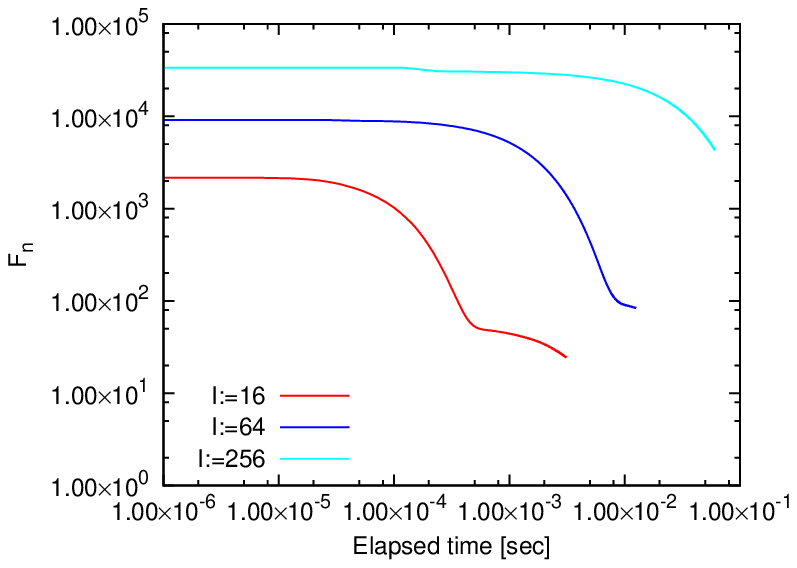}}
\caption{Efficiency of Algorithm \ref{algo:1} for $I$ when $\lambda_n := 10^{-3}/(n+1)^{0.01}$}\label{fig:9}
\end{figure}
\begin{figure}[H]
\subfigure[$D_n$ vs. no. of iterations]{\includegraphics[width=0.5\textwidth]{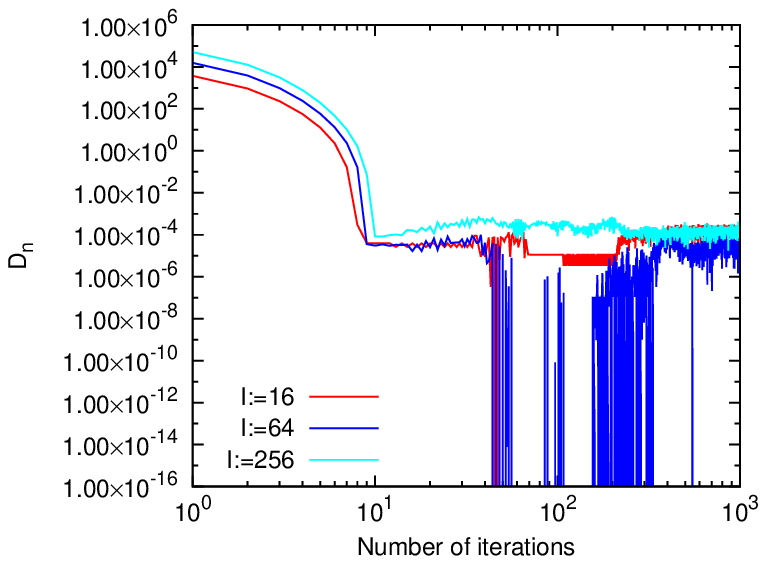}}
\subfigure[$D_n$ vs. elapsed time]{\includegraphics[width=0.5\textwidth]{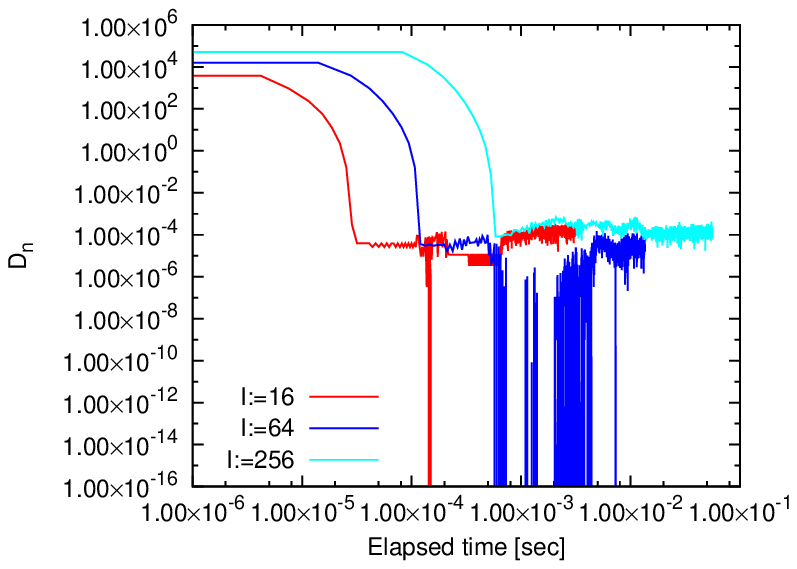}}
\\{}
\subfigure[$F_n$ vs. no. of iterations]{\includegraphics[width=0.5\textwidth]{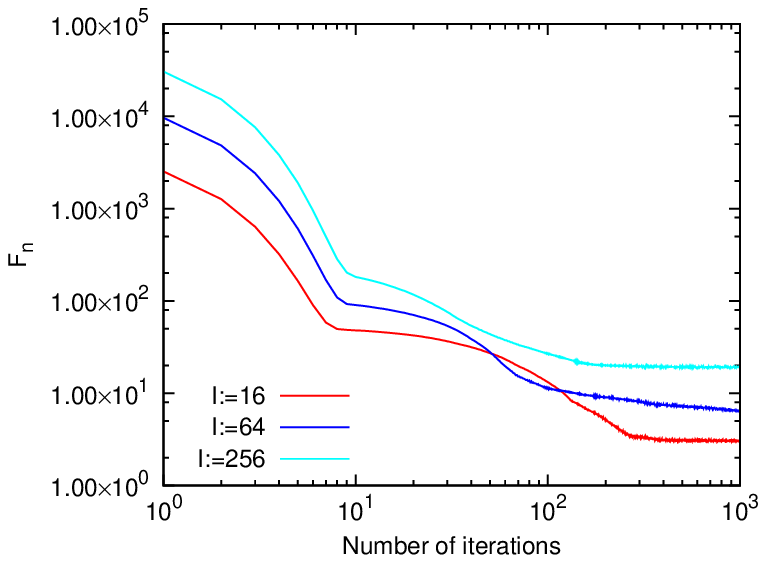}}
\subfigure[$F_n$ vs. elapsed time]{\includegraphics[width=0.5\textwidth]{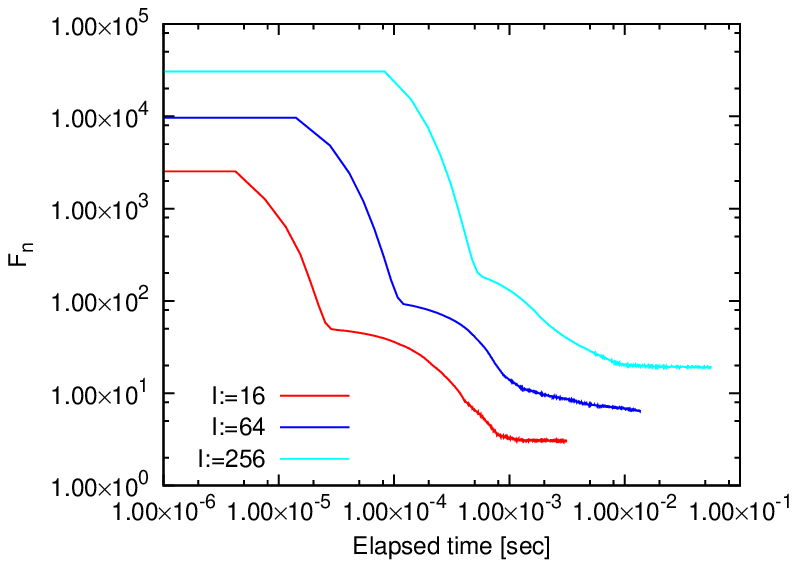}}
\caption{Efficiency of Algorithm \ref{algo:2} for $I$ when $\lambda_n := 10^{-3}/(n+1)^{0.01}$}\label{fig:10}
\end{figure}

\begin{figure}[H]
\subfigure[$D_n$ vs. no. of iterations]{\includegraphics[width=0.5\textwidth]{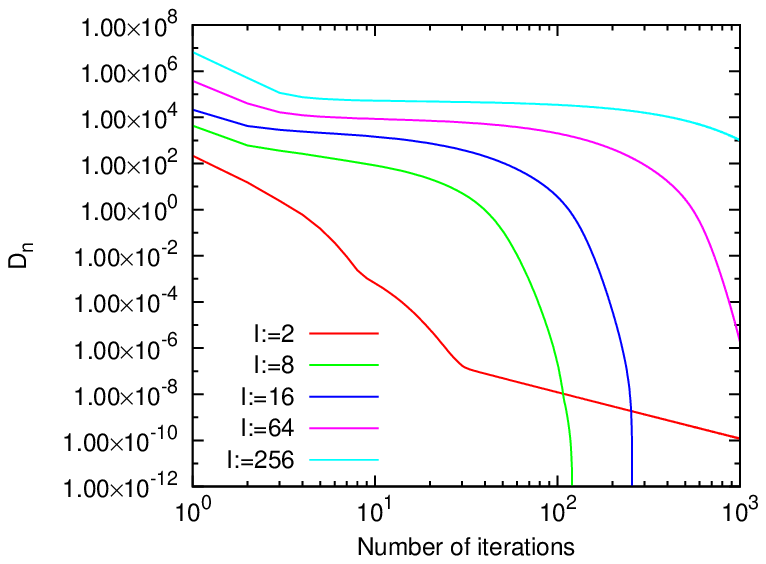}}
\subfigure[$D_n$ vs. elapsed time]{\includegraphics[width=0.5\textwidth]{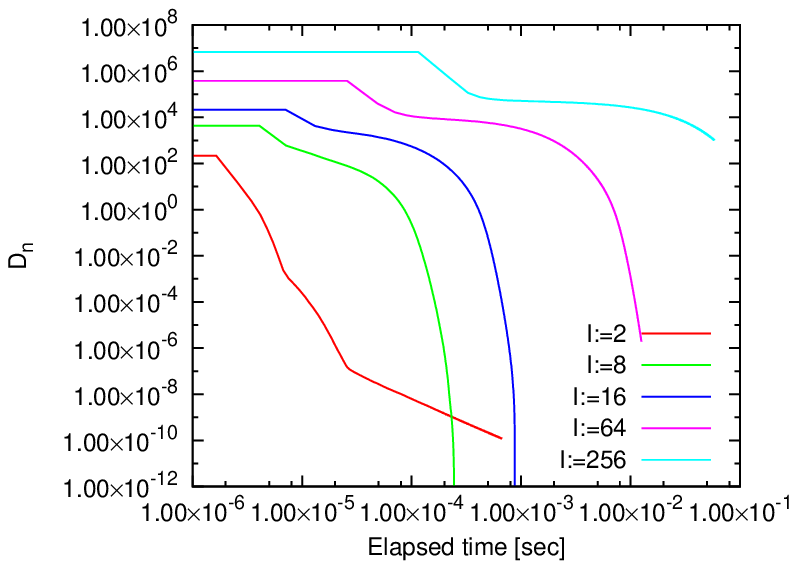}}
\\{}
\subfigure[$F_n$ vs. no. of iterations]{\includegraphics[width=0.5\textwidth]{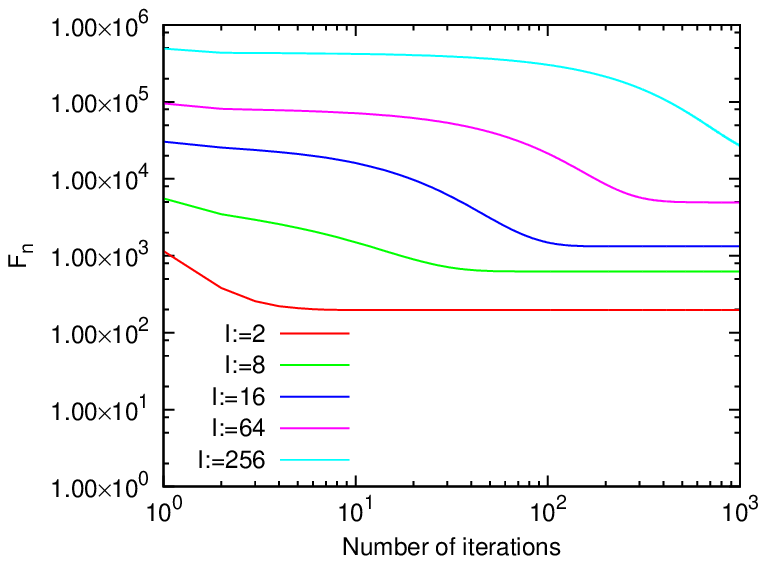}}
\subfigure[$F_n$ vs. elapsed time]{\includegraphics[width=0.5\textwidth]{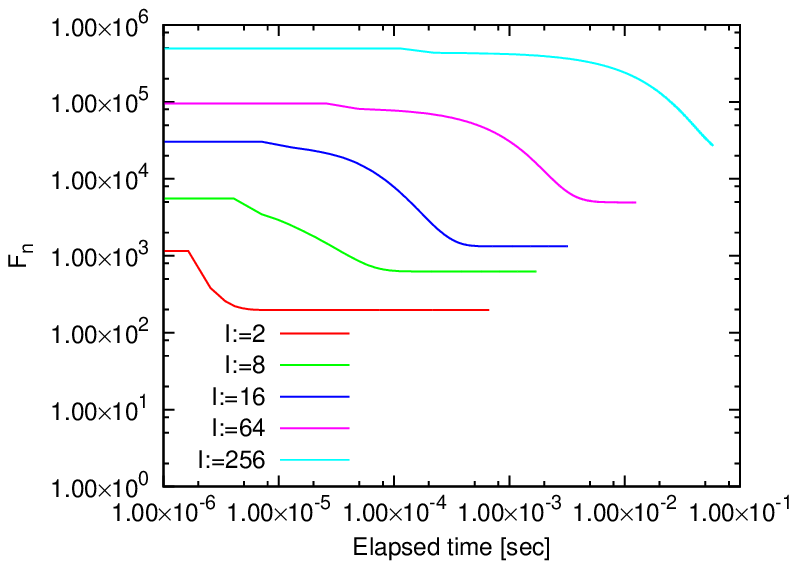}}
\caption{Efficiency of Algorithm \ref{algo:1} for $I$ when only $f^{(1)}$ is strongly convex and $\lambda_n := 10^{-3}/(n+1)$}\label{fig:11}
\end{figure}
\begin{figure}[H]
\subfigure[$D_n$ vs. no. of iterations]{\includegraphics[width=0.5\textwidth]{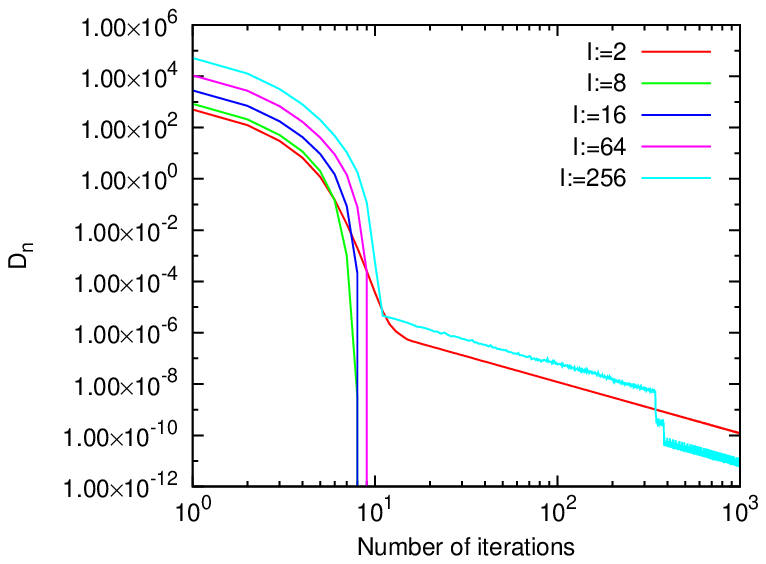}}
\subfigure[$D_n$ vs. elapsed time]{\includegraphics[width=0.5\textwidth]{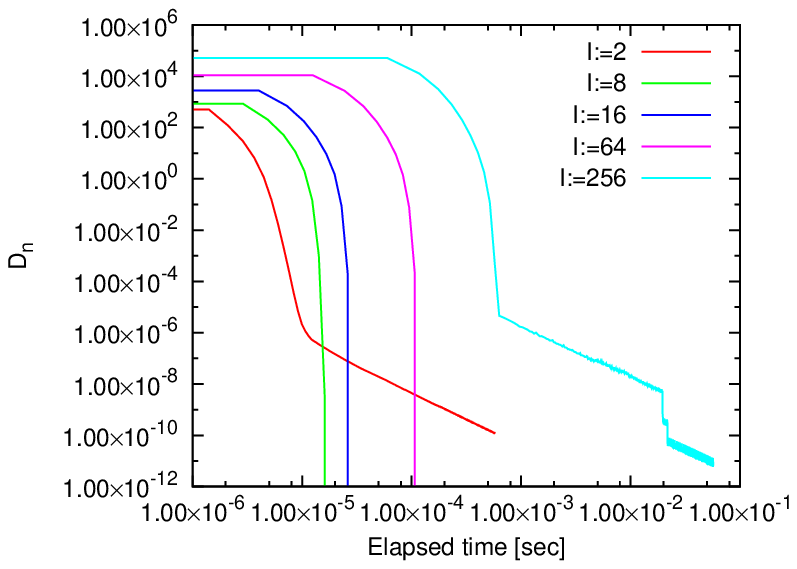}}
\\{}
\subfigure[$F_n$ vs. no. of iterations]{\includegraphics[width=0.5\textwidth]{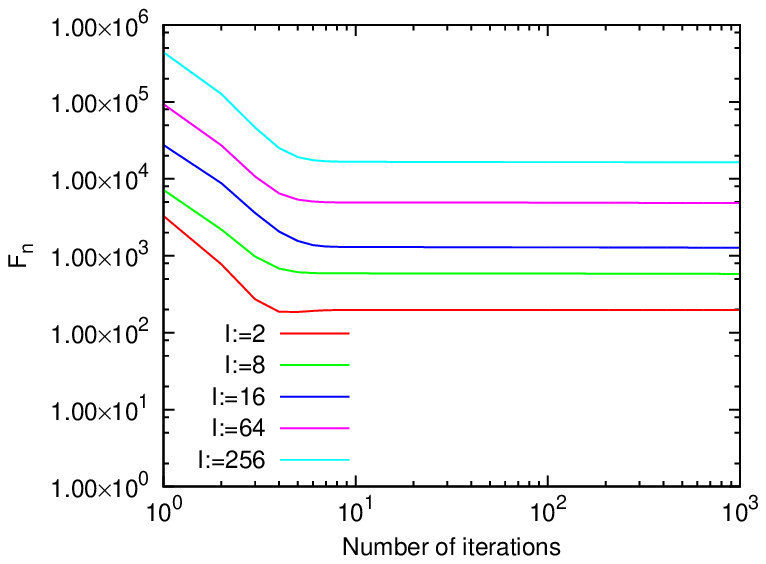}}
\subfigure[$F_n$ vs. elapsed time]{\includegraphics[width=0.5\textwidth]{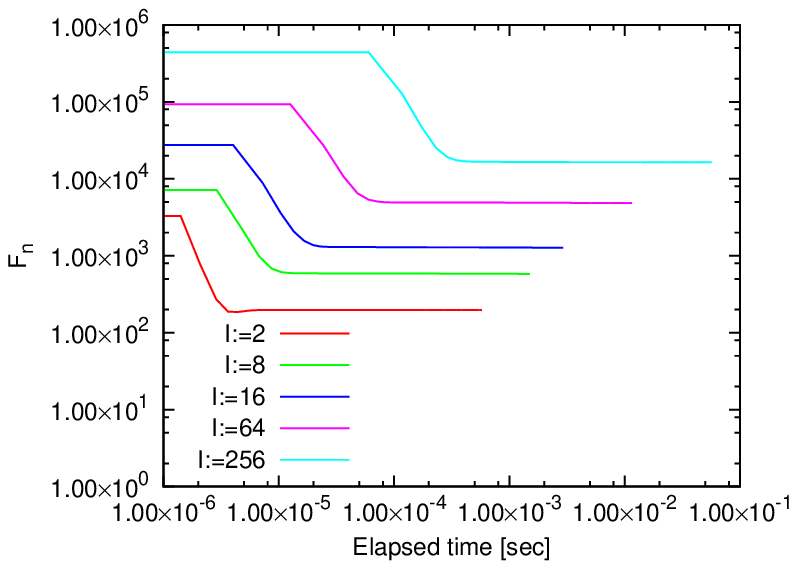}}
\caption{Efficiency of Algorithm \ref{algo:2} for $I$ when only $f^{(1)}$ is strongly convex and $\lambda_n := 10^{-3}/(n+1)$}\label{fig:12}
\end{figure}

\end{document}